\documentclass[11pt, a4paper,reqno]{amsart}
\usepackage{amsmath, amsthm, amscd, amsfonts, amssymb,color}
\usepackage[utf8x]{inputenc}
\usepackage{fullpage,verbatim}
\usepackage{pdftricks}
\begin{psinputs}
	\usepackage{graphicx,psfrag,subfigure}
	\usepackage{pstricks}
	\usepackage{multido}
	\usepackage{pst-all}
\end{psinputs}

\newcommand\ada[1]{{\color{blue} #1}}
\newcommand\dela[1]{}
\newcommand{\intg}{\int_{\partial D}}
\newcommand{\goto}{\rightarrow}
\newcommand{\vecx}{x}
\newcommand{\vecy}{y}
\newcommand{\mE}{\mathbb E}
\newcommand{\mP}{\mathbb P}
\newcommand{\norm}[2]{\left\|{#1}\right\|_{#2}}

\newcommand\bsig{\boldsymbol{\sigma}}
\newcommand\btau{\boldsymbol{\tau}}
\newcommand\bn{\boldsymbol{n}}
\newcommand\bcW{\boldsymbol{\mathcal{W}}}
\newcommand\OS{\O_{\mathrm S}}
\newcommand\bdiv{\mathop{\mathbf{div}}\nolimits}

\newcommand\qin{\qquad\hbox{in }}
\newcommand\OF{\O_{\mathrm F}}
\newcommand\HOF{\mathrm{H}^1(\OF)}
\newcommand*{\norma}[1]{%
\left\|\hspace{-1pt}\left|#1\right\|\hspace{-1pt}\right|}
\newcommand{\inpro}[2]{\left\langle{#1},{#2}\right\rangle}

\newcommand{\0}{\boldsymbol{0}}
\newcommand\bu{\boldsymbol{u}}

\newcommand\br{\boldsymbol{r}}
\newcommand\bs{\boldsymbol{s}}

\newcommand{\Corr}{{\rm{Cor}}}
\newcommand{\Covv}{{\rm{Cov}}}

\newcommand\bF{\boldsymbol{f}}

\newcommand\bG{\boldsymbol{G}}

\newcommand\bT{\boldsymbol{T}}

\newcommand\beps{\boldsymbol{\mathcal{E}}}

\DeclareMathOperator{\divv}{{div \/}}
\DeclareMathOperator{\var}{{Var \/}}

\newcommand\cC{\mathcal{C}}

\newcommand\cL{\mathcal{L}}

\newcommand\bcQ{\boldsymbol{\mathcal{Q}}}

\newcommand\bcX{\mathbb X}

\newcommand{\veczero}{\boldsymbol{0}}

\newcommand\bbA{\mathbb{A}}
\newcommand\bbB{\mathbb{B}}

\newcommand\R{\mathbb{R}}

\renewcommand\H{\mathrm{H}}
\renewcommand\L{\mathrm{L}}

\renewcommand\S{\Sigma_{C}}
\renewcommand\O{\Omega}

\newcommand\DOF{\partial\OF}

\newcommand\GG{\Gamma}
\newcommand\GD{\GG_{\mathrm D}}
\newcommand\GN{\GG_{\mathrm N}}
\newcommand\GO{\GG_F}
\newcommand\mas{\mathrm{S}}
\newcommand\F{\mathrm{F}}

\newcommand\rS{\rho_{\mathrm S}}
\newcommand\rF{\rho_{\mathrm F}}
\renewcommand\sp{\mathop{\mathrm{sp}}\nolimits}

\renewcommand\t{\mathtt{t}}

\newcommand{\set}[1]{\left\{#1\right\}}

\newcommand\qon{\qquad\hbox{on }}
\newcommand{\bV}{\boldsymbol{V}}

\newcommand{\vecn}{\boldsymbol{n}}
\newcommand\dist{\mathop{\mathrm{dist}}\nolimits}

\newcommand\dxd{d\times d}

\theoremstyle{plain}
\newtheorem{theorem}{Theorem}[section]
\theoremstyle{plain}

\newtheorem{lem}[theorem]{Lemma}
\newtheorem{prop}[theorem]{Proposition}

\theoremstyle{definition}

\newtheorem{defi}{Definition}[section]
\newtheorem{rem}{Remark}

\def\Rd{\mathbb{R}^d}
\def\H{\mathbb{H}}

\def\n{\mathbf{n}}

\newcommand{\eps}{\varepsilon}
\newcommand{\kerab}{[\ker(a^{\eps}) \times \bcQ^{\eps}]^{\perp_{\mathbb{B^{\eps}}}}}

\def\u{\mathbf{u}}
\def\H{\mathbf{H}}
{%
\setcounter{enumi}{0}

\begin{enumerate}}
{\end{enumerate} }

{
\setcounter{enumi}{0}

\begin{enumerate}}
{\end{enumerate} }

\numberwithin{equation}{section} 
\allowdisplaybreaks

\begin{document}
\begin{abstract}
The present paper deals with the interior solid-fluid interaction problem in harmonic regime with
randomly perturbed boundaries. Analysis of the  shape derivative and shape Hessian of vector- and
tensor-valued functions is
provided. Moments of the random solutions are approximated by those of the shape derivative and shape
Hessian, and the approximations are of third order accuracy in terms of the size of the boundary
perturbation. Our theoretical results are supported by an analytical example on a square domain.
\end{abstract}

\title[A Shape calculus apprpoach for solid--fluid interaction problem in stochastic domain]{A Shape calculus approach for time harmonic solid--fluid interaction problem in stochastic domains }

%\author[Debopriya Mukherjee]{Debopriya Mukherjee}
\author{Debopriya Mukherjee}
\address{%
School of Mathematics and Statistics\\
The University of New South Wales, Sydney 2052, Australia.}
\email{debopriya249@gmail.com}

\author{Thanh Tran}
\address{%
School of Mathematics and Statistics\\ 
The University of New South Wales, 
Sydney 2052, Australia.}
 \email{thanh.tran@unsw.edu.au}‎
\thanks{This  work is
supported by  the Australian Research Council Project DP160101755.
}  

\date{\today}
\maketitle  
%\tableofcontents
%\section{Introduction}

%\textbf{AMS subject classification (2002):} {Primary 60H15;
%Secondary 60G57.}
\section{Introduction}
In this paper, we consider the time harmonic forced vibrations of an elastic solid encircling in its interior an inviscid compressible fluid with randomly located boundaries. Since the domain and its perturbed boundaries are stochastic,
the solution depends on the `random event' $\omega$ and the parameter $\eps \geq 0$ controlling the amplitude of the perturbation.  
The usual approach to generate a large number $N$ of `sample' domains and to solve the deterministic boundary value problem on each sample is overpriced.
To overcome this costly computation, we approximate the moments of the solution by those of its shape derivative and shape Hessian, as has been done for other simpler models; see \cite{CPT,H10,HSS08}.
The problem considered in this article is much more complex involving vector-valued functions and tensors which requires careful analysis.

Over the past few years, the  solid--fluid interaction problems gain much attention due to its
applications in different engineering fields \cite{Chak,Dowell+Hall,Morand+Ohayon},
magneto-hydrodynamic flows  \cite{Grigoriadis}, electro-hydrodynamics \cite{Hoburg}, etc.
The model problem is represented by a vector--valued equation describing
time harmonic elastodynamic equations in the solid domain and the Helmholtz equation in the fluid region. On the common boundary the two systems are coupled via adequate transmission conditions. 

It is well known that mathematical models are approximations of physical phenomena.
Most often, the base model is too complicated or the scales are too disparate to include all
the parameters successfully. The neglected parameters are often replaced 
by some randomness in the deterministic model. In this way, loadings, coefficients and the
underlying domains are considered as stochastic input parameters. In the present article, we
focus on randomness in the domains.
%This leads to consider partial differential equations with stochastic coefficients, with applications to physical phenomenon. 
The authors of \cite{H10,HSS08} exploit shape calculus tools to compute statistical moments of
the random solutions of elliptic boundary value problems on uncertain domains.
The authors of~\cite{CPT} use the same tool to solve elliptic transmission problems in 
unbounded stochastic domains. They also provide rigorous derivation and properties of shape derivatives.

Shape calculus tools involving the computation of shape derivative and shape Hessian (known as the so-called material derivative approach of continuum mechanics) is studied in
the books \cite{Has,SokZol92} in the deterministic framework. There is a growing literature
rationalized to shape optimization problems in estimating the first and second order shape
derivatives in the deterministic setup. For a quick survey we refer to, for example,
\cite{Bacani,Del+Zol,Tiihonen}.

Authors in \cite{Dambrine_2015} have derived asymptotic expansions of the first moments of the distribution of the output functional considered on a random domain through a boundary value problem.  Computation and use of second order derivative for vector valued states in the context of linear elasticity goes back to the works of Murat and Simon. Application of shape optimization methods in fluid mechanics in determinitic set-up is well-known in the literature; see, for example \cite{Caubet_2011} for the Stokes equations with both Dirichlet and Neumann boundary conditions, \cite{Caubet_2013} for stationary Navier-Stokes equations with non-homogeneous Dirichlet boundary conditions, \cite{Hettlich_1998} for transmission boundary value problem.

The solid--fluid interaction problems have been studied by Estecahandy and her co-authors in a nice series of works in deterministic set-up. To be more specific, we refer to \cite{Estecahandy_2014} (Discontinuous Galerkin based approach for higher-order polynomial-shape functions with the high frequency propagation regime) and \cite{Estecahandy_2018_SIAM} (finite element method approach to the Lipschitz continuous polygonal domains) and references therein.

%Shape derivative introduced in \cite{Has,SokZol92} is successfully used to solve the Dirichlet problem for the Poisson equation in \cite{H10,HSS08}. Furthermore, using a second--order shape calculus, the authors in \cite{H10,HSS08} define the first and second order local shape derivative at any interior point of the unperturbed nominal domain.
%Article \cite{Babuska1} analyzes Monte Carlo Galerkin finite element method to compute statistics of random solutions for the homogeneous Dirichlet problem for linear second order elliptic equations with random coefficients. 
%Computation of statistics of the random solution is practised in the literature (see \cite{Babuska1, H10,  HSS08} and the references cited therein). 
%by random input data as an improvement over completely neglecting them. Also, randomness is introduced to
%the deterministic base model for better understanding of various statistical properties
%(e.g. invariant measures, ergodicity, random attractors, turbulence, vortex structures,
%moderate and large deviations, control, filtering etc.) of solutions of the models.
%\par
%Describing changes in the geometry and differentiating functions with varying domain is challenging. 
%In the current paper, we have implemented tools from shape derivative calculus governing material derivative approach (see \cite{SokZol92,Has}).
%\subsection{Technical challenges and contribution of the paper}
\par
In this article, we develop a precise mathematical theory for computing the statistics of the solution of the solid-fluid interaction problem with randomly perturbed boundaries.
Our main contribution in this article is the derivation of the second order shape Taylor expansion of
vector-valued and tensor-valued functions. The results are presented in 
Theorem~\ref{thm.mat.der} (material derivative), Theorem \ref{thm.sh.der} (shape derivative) and Theorem
\ref{thm.sh.hes} (shape Hessian).
As a consequence, we obtain the stochastic shape Taylor expansion for the moments of the
solution (Theorem \ref{thm.main}).
To the best of our knowledge, this current work appears to be the first systematic treatment of
second order shape calculus for vector-valued and tensor-valued functions which form the
solution of the solid--fluid interaction problem under consideration.

In order to apply shape calculus to our particular model problem (the solid-fluid problem) a technical
issue requires us to study the spectrum of the solution operator of the problem 
(Proposition~\ref{prop.spT}). This result, another contribution of the paper, has its own interest.

Let us briefly describe the content of this paper. Section~\ref{sec.model} deals with the
description of the model problem with perturbed random boundaries
and details of the function spaces involved during the course of analysis. 
Section~\ref{sec:spe sol ope} consists of the spectral properties of the solution operator. 
Section~\ref{shape.sec} contains details of first and second order shape calculus. It also
shows the approximation of the solution moments with those of its shape derivatives. 
Section~\ref{sec.ex} provides an analytical example on a square domain perturbed by a uniform
distribution. This example illustrates the accuracy of the approximation.
Finally in the Appendix we recollect some basic definitions of tensors and their
properties, present some 
technical lemmas, and recall elementary  concepts of material and shape derivatives.

In the paper, $C$ stands for (with or without subscripts) a generic constant independent of the
discretization parameter and the wave number. These constants may take different values at different
places.

\section{ Time harmonic solid--fluid interaction problem on perturbed domain}
\label{sec.model}
%\section{Shape Calculus}
%\section{Domain and boundary perturbation}\label{sec.perturbation}

In this section we describe the problem and provide preliminaries for the forth-coming analysis.

\subsection{Statistical moments}\label{pre.sec}
%In the present paper we utilize the domain perturbation model 
%based {on} the \emph{speed method} (see e.g. 
%the monograph \cite{SokZol92} and the references therein) and random domain perturbation model from \cite{CPT,H10,HSS08} and the references cited therein. 
Throughout this paper, we denote by $(\mathfrak{U}, \mathcal{U},\mathbb{P})$ a generic complete probability space.
Let $D$ be a bounded domain in $\R^3$
with boundary $\partial D$ of class $C^k,\,k \geq 2$.
\begin{defi}
For a random field $v\in L^k(\mathfrak{U}, D)$,
its $k$-order moment $\mathcal{M}^k[v]$ is an element of $D^{(k)}$ defined by
%\begin{equation}\label{moments-def}
\[
\mathcal{M}^k [v] :=  \int_{\Omega} 
\big(
\underbrace{v(\omega)\otimes\cdots\otimes v(\omega)}_{k\textrm{-times}}
\big)
\,d\mathbb P(\omega).
\]
%\end{equation}
\end{defi}
In the case $k=1$, the statistical moment $\mathcal{M}^1 [v]$ is same as the \textit{mean value}
of $v$ and is denoted by $\mE [v]$. If $k\geq 2$, 
the statistical moment $\mathcal{M}^k [v]$ is known as the 
\textit{$k$-point autocorrelation function} of $v$.
The quantity $\mathcal{M}^k [v - \mE[v]]$ is termed the $k$-th central moment of $v$. In particular, the
second order moments: the \textit{correlation} and \textit{covariance} are
defined by
\begin{equation}\label{eq.cor}
\Corr[v] := \mathcal{M}^2[v] \quad\text{and}\quad  \Covv[v] := \mathcal{M}^2[v - \mE[v]].
\end{equation}
\dela{To specify random boundary variations,we assume that ....is a random
field on ∂D taking values in R. We denote by X a space of admissible boundary perturbation
functions κ. The random perturbations of D will be described by a suitable
probability space  consisting of (a) a set 	 of realizations 
(i.e., realizations of particular perturbations κ(·)), (b) a sigma algebra 
 and (c) a
probability measure on the space X.}

\subsection{Representation of random interfaces}\label{subsec.perturbation}
Let us consider a solid body represented by a $C^2$ domain $\OS\subset \mathbb R^d$, $d=2,3$, with $\partial \OS = \GD\cup\GN\cup \S$, where 
$\GD$, $\GN$, and $\S$ are disjoint parts of $\partial \OS$. We assume that the solid structure is fixed at $\GD\neq \emptyset$ and free of stress on$~\GN$.
The solid interacts through the interface $\S$ with a homogeneous, inviscid and compressible
fluid occupying a bounded  domain $\OF$.  The boundary $\DOF$ of the fluid domain is  $\S$. We
denote by $\bn_{\mas}$ (~$\bn_{\F}$ respectively) the outward-pointing unit normal vector to
the boundary $\partial \OS$ ($\partial \OF$ respectively) of the fluid-solid domain
$\O:=\OS\cup\OF$; see Figure~\ref{fig:Ome}. It can be observed that on $\S$, one has
$\bn_{\mas}=-\bn_{\F}.$ For more details about the model problem, we refer to \cite{Garcia,MMT}
and the references cited therein.

Following \cite{CPT} and the references therein, we present the random domain.
Suppose $\kappa \in L^k(\mathfrak{U}, {C^{2,1}(\partial \OS)})$ is a random field. 
%Suppose $\kappa$ is a random field satisfying
%\begin{align}\label{reg.kappa}
%\|\kappa(\cdot,\omega)\|_{C^{2,1}(\partial \OS)} \leq 1 \quad \mbox{for} \quad \mathbb{P}-\mbox{almost all} \quad \omega \in \mathfrak{U}.
%\end{align}
%for almost any realization $\omega \in \Omega$, we have $\kappa(\cdot,\omega) \in C^{2,1}(\partial D^0)$. 
For some sufficiently small value $\eps \geq 0$, we consider a family of random interfaces of the form
\begin{equation}\label{equ:rand inter}
\partial \OS^\eps(\omega)
=
\{
\vecx 
+
\eps
\kappa(\vecx,\omega)\vecn(\vecx):
\vecx\in \partial \OS
\},
\quad
\omega \in \mathfrak{U},
\end{equation}
where $\bn$ is given by
\begin{align*}
\bn= \left\{\begin{aligned}
                  & \bn_{\F} \qquad \mbox{on} \quad \S,\\
                   & \bn_{\mas} \quad \quad \mbox{on} \quad \GD \cup \GN.
                   \end{aligned}
   \right.
\end{align*}
Here, the randomness of the surfaces $\partial \OS^\eps(\omega)$ is represented by the
randomness in $\kappa(\cdot,\omega)$. 
We observe that the interface $\partial \OS^\eps(\omega)|_{\eps = 0}$ coincides with $\partial
\OS$ and therefore is a deterministic closed manifold. 
%Moreover, the limit $\partial \OS^\eps(\omega) \to \partial \OS$ as $\eps \to 0$ is well
%defined in $L^k(\mathfrak{U},C^{2,1}(\partial \OS))$. 
% \todo{Indeed, suppose $\G^0 = \cup_{j=1}^J \partial D_j$ where $\gamma_j = \big\{\vecx \in \R^3: \vecx = \chi_j(\xi) \text{ for } \xi \in \tau_i \subset \R^{d-1}\big\}$.} 
If we identify $\partial \OS^\eps$ and $\partial \OS$ with the functions 
defining their graphs, then
%\begin{equation}\label{Gamma-conv}
\[
\begin{split}
\|\partial \OS^\eps - \partial \OS\|_{L^k(\mathfrak{U},C^{2,1})} 
&
%= \eps \left( 
 %\int_\Omega\|\kappa(\cdot,\omega) \vecn\|_{C^{2,1}(\partial \OS)}^\infty \, d \mathbb{P}(\omega)\right)^{\frac{1}{k}}
\leq 
\eps \|\kappa\|_{L^k(\mathfrak{U},{C^{2,1}(\partial \OS)})}\|\vecn\|_{C^{2,1}(\partial \OS)}.
\end{split}
\]
%\end{equation}

We will specify the required smoothness assumptions on $\kappa$ in \emph{shape calculus} in Section \ref{shape.sec}.
From \eqref{equ:rand inter} we observe that the \emph{mean random interface} is represented by
\[
\mE[\partial \OS^\eps]
=
\big\{ \vecx + \eps {\mE[\kappa(\vecx,\cdot)] \bn}(\vecx), \ \vecx\in\partial \OS \big\}.
\]
Without loss of generality, we may assume 
that the random perturbation amplitude $\kappa(\vecx,\omega)$ is centred, i.e.,
\begin{equation}\label{equ:kappa sym}
{\mE[\kappa(\vecx,\cdot)]}
=
0
\qquad
\forall
\vecx\in\partial \OS.
\end{equation}
{In this case}
\[
\mE[\partial \OS^\eps]
= \partial \OS \qquad \text{and} \qquad
\Covv[\kappa](\vecx,\vecy)
=
\Corr[\kappa](\vecx,\vecy).
\]
%Let $\Gamma$ be a surface in $\mathbb{R}^d$ such that $\Gamma:=\Gamma_D \cup \Gamma_N\cup \Sigma_C.$ 
%Let us assume that the perturbation function $\kappa$ be a fixed deterministic function in $\mathbb{W}^{1,\infty}(\partial \OS \cup \partial \OF)\cap C^{2,1}(\partial \OS \cup \partial \OF).$
%% such that $\kappa=0$ on $\Gamma_D \cup \Gamma_N \cup \Gamma_0.$
%Let us define
%$$ \Gamma^{\eps}:= \Big\{\vecx +\eps \kappa(\vecx)\n(\vecx): \vecx \in \Gamma \Big\}  , \,\, \eps \geq 0,$$ 
%Note that $\Gamma^0=\Gamma.$ 
%Concerning the existence of the $k$-th order moment (cf. Theorem \ref{the:mean cov}), we impose later on the condition that $\kappa \in L^k(\Omega,C^1(\partial \OS))$ for some integer $k$.
\begin{figure}[h]
\begin{center}
\begin{pdfpic}
\psset{unit=0.6cm}
%\psset{xunit=0.2cm,yunit=0.2cm}
\begin{pspicture}
%\psgrid[subgriddiv=1,griddots=10,gridlabels=10pt](-4,-4)(4,4)
%\psgrid(-4,-4)(4,4)
%
\psarc[linewidth=2pt](0,0){3}{0}{180}
\pscircle(0,0){3}
\pscircle(0,0){1.5}
\uput[d](0,-3.0){$\Gamma_{\mathrm{N}$}}
\uput[d](0,4){$\Gamma_{\mathrm{D}$}}
\uput[l](-1.4,0){$\Sigma_C$}
\uput[ur](1.5,0){\textit\textbf{n}}
\psline{->}(1.5,0)(2.5,0)
\psline{->}(3,0)(4,0)
\def\blockpt{
\multirput(0,0)(0.2,0){5}{\psdot[dotsize=2pt]}
}
\multirput(-0.9,1)(1,0){2}{\blockpt}
\multirput(-1.4,0.5)(1,0){3}{\blockpt}
\multirput(-1.45,0.0)(1,0){3}{\blockpt}
\multirput(-1.4,-0.5)(1,0){1}{\blockpt}
\multirput(0.6,-0.5)(1,0){1}{\blockpt}
\multirput(-0.9,-1)(1,0){2}{\blockpt}
\uput[d](1,-1.5){$\Omega_{\mathrm{S}}$}
\uput[d](0,0.1){$\Omega_{\mathrm{F}}$}
\end{pspicture}
\end{pdfpic}
\caption{Solid domain~$\OS$ and fluid domain~$\OF$}\label{fig:Ome}
\end{center}
\end{figure}

\subsection{Model problem}\label{modelproblem.sec}
For some sufficiently small and nonnegative $\eps$ we aim to compute the linear oscillations of an
elastic structure encircling in its interior an inviscid fluid appearing in the fluid-solid perturbed
domain $\O^{\eps}(\omega):=\OS^{\eps}(\omega)\cup\S^{\eps}(\omega)\cup\OF^{\eps}(\omega)$, under the
action of a given time sinusoidal body force prescribed in the solid domain whose amplitude is $\bF: B_R
\to \R^d$, which is assumed to be independent of $\omega$.
Having introduced these perturbed domains and boundaries, the model problem 
is to find the solid displacement field $\bu^{\eps}$ and the fluid pressure $p^{\eps}$ satisfying
%The mathematical model associated to the physical phenomenon under interest is given by the set of equations
%The perturbed problem associated via the map $\bT^{\eps}$ under interest
%corresponding to the \ada{perturbed interface $\Sigma_{C}^{\eps}$}
%is given by the set of equations
\begin{subequations}\label{equ:perturbed prob}
\begin{align}
  \bdiv \bsig^\eps(\vecx,\omega)+\mu^2\rS \bu^\eps (\vecx,\omega)=\bF(\vecx) & \qin\OS^\eps(\omega),
  \label{dual-problem_1a.ep}
 \\
 \bsig^\eps(\vecx,\omega) = \cC \beps(\bu^\eps(\vecx,\omega)) & \qin\OS^\eps(\omega),
 \label{dual-problem_1b.ep}
 \\
\bu^\eps(\vecx,\omega) =\0 & \qon\GD^\eps(\omega),
\label{dual-problem_2.ep}
\\
\bsig^\eps(\vecx,\omega) \bn^{\eps}=\0 & \qon\GN^\eps(\omega),
\label{dual-problem_3.ep}
\\
\Delta p^\eps(\vecx,\omega) +\frac{\mu^2}{c^2}p^\eps(\vecx,\omega)=0 & \qin\OF^\eps(\omega),
\label{dual-problem_4.ep}
\\
\bsig^\eps(\vecx,\omega) \bn^{\eps}+p^{\eps}(\vecx,\omega) \bn^{\eps}=\0 & \qon\S^\eps(\omega),
\label{dual-problem_6.ep}
\\
\frac{\partial p^\eps}{\partial\bn^{\eps}}(\vecx,\omega)
- \mu^2 \rF \bu^\eps(\vecx,\omega) \cdot\bn^{\eps}=0 & \qon\S^\eps(\omega),
\label{dual-problem_7.ep}
\end{align}
\end{subequations}
where $\0$ stands for a generic null vector or tensor.
Here the stress tensor $\bsig^{\eps}$ is defined by the linearised strain tensor 
$\mathcal{E}(\u^{\eps})$ and the
Hooke operator $\mathcal{C}:\R^{d \times d} \rightarrow  \R^{d \times d}$ defined by
$$
\mathcal{E}(\u^{\eps}):=\dfrac{1}{2} \Big( \nabla \u^{\eps}+(\nabla \u^{\eps})^{\top} \Big)
\quad\text{and}\quad
\mathcal{C} \btau:=\lambda(\mbox{Tr}\, \btau)I+2\nu \btau \quad \forall \, \btau \in \R^{d \times d}.
$$
Here $\lambda, \nu>0$ are the Lam\'{e} constants and~$\text{Tr}\btau$ denotes the trace of~$\btau$.
The remaining physical coefficients are the solid and
fluid densities $\rS>0$ and $\rF>0$, respectively, the acoustic speed $c>0$, the input frequency $\mu$,
and the gravity constant $g$. 
%Here, $\partial/\partial\vecn^\eps$ denotes the normal derivative on~$\S^\eps(\omega)$.

In the next subsection, we introduce the function spaces in the deterministic set-up needed for the analysis.

\subsection{Function spaces and weak formulation.}
We denote by $C_c^{\infty}(\Rd;\Rd)$ as the space of all $\Rd-$valued compactly supported $C^{\infty}$ functions in~$\Rd.$
In what follows we will denote the vectorial and tensorial counterparts of order
$d$ ($d=2,3$) of a given  Hilbert space $\H$  by $\H^d$ and $\H^{\dxd},$ respectively. We use  standard notation for the Hilbertian Sobolev space $\H^s(D)$, $s\geq 0$, defined on a $C^2$ bounded domain $D\subset \R^d$ and
denote by  $\|\cdot\|_{s,D}$ the norms in $\mathrm{H}^s(\Omega)$,  $\mathrm{H}^s(D)^d$ and $\mathrm{H}^s(D)^{\dxd}$.

%The component-wise inner product of two matrices $\bsig, \,\btau \in\R^{\dxd}$
%is denoted $\bsig:\btau:= \tr( \bsig^{\t} \btau)$, 
%where $\tr\btau:=\sum_{i=1}^d\tau_{ii}$ and $\btau^{\t}:=(\tau_{ji})$ 
%stand for the trace and the transpose of $\btau = (\tau_{ij})$ respectively.
For $\bsig:D\to \R^{\dxd}$ and $\bu:D\to \R^d$, we define the row-wise divergence
$\bdiv \bsig:D \to \R^d$  and  the gradient $\nabla \bu:D \to \R^{\dxd}$ by,
\[
 (\bdiv \bsig)_i := \sum_j   \partial_j \sigma_{ij} \qquad \text{and}  \qquad (\nabla \bu)_{ij} 
 := (\nabla^{\top} \otimes \bu)_{ij} 
 =\partial_j u_i.
\]
We  introduce for $s\geq 0$ the Hilbert space
\[
 \H^{s}(\mathbf{div};D):=\set{\btau\in\H^s(D)^{\dxd}:\ \bdiv\btau\in\H^s(D)^d}
\]
endowed with the norm  $\|\btau\|^2_{\H^{s}(\mathbf{div};D)}
:=\|\btau\|_{s,D}^2+\|\bdiv\btau\|^2_{s,D}$, and we use the convention
$\H(\mathbf{div};D):={\H^0(\mathbf{div};D)}$. 

%\newline
%Given two Hilbert spaces $\mathbb{S}_1$ and $\mathbb{S}_2$ and a bounded bilinear form $c:\mathbb{S}_1 \times \mathbb{S}_2 \rightarrow \R,$ we will denote 
%$$\mbox{Ker}(c):=\{s \in \mathbb{S}_1:c(s,t)=0 \quad \forall t \in \mathbb{S}_2 \}.   $$
\par
%\section{Functional spaces}
The stress tensor $\bsig^{\eps}$, which is imposed here 
as a primary unknown in the solid, will be sought in  the Sobolev space 
\[
\bcW^{\eps} := \Big\{\btau\in \mathbf{H}(\bdiv, \OS^{\eps}): \,\, \btau \bn^{\eps} = \0 \quad \text{on $\GN^{\eps}$} \Big\}.
\]
The fluid main variable is the pressure $p^{\eps} \in \H^{1}(\OF^{\eps})$.
For convenience we introduce the product space 
\[
\widetilde {\mathbb X}^{\eps}:= \bcW^{\eps}\times \mathrm{H}^1(\OF^{\eps})
\] 
endowed with the Hilbertian norm 
\[
\|(\btau, q)\|_{\eps}^2
:=\|\btau \|_{\H(\bdiv,\OS^{\eps})}^2+\|q\|_{1,\OF^{\eps}}^2 \quad \forall \, (\btau, q) \in \widetilde {\mathbb X}^{\eps}.
\]

%\norm{(\btau, q)}^2:
%= \norm{\btau}^2_{\H(\bdiv, \OS)} %+ \norm{q}^2_{1,\OF}.
%\] 
In articles \cite{Bermu,Feng}, displacement formulation in the solid  combined with a formulation using
the acoustic pressure (or the fluid displacement) as main variables in the
fluid domain is studied.
In recent years, there are extensive studies of the stress-pressure formulation weakly imposing the
symmetry of the
stress tensor; see for instance \cite{Ar1,Ar2,GMM,gmm2011}. The dual-mixed formulation which
approximates the elastic Cauchy stress tensor is emphasized
in the literature; see e.g. \cite{Ar1,MR3376135} and the references therein.  
%The second author of this paper and his collaborators in \cite{MR3376135} have presented the dual mixed formulation of the above described problem in harmonic regime and have obtained quasi-optimal error estimates in suitable norm.

 As we are dealing with a dual formulation in $\OS^{\eps}$, the transmission condition \eqref{dual-problem_6.ep} becomes essential (cf. \cite{gmm2011,MMR3}), it should then be strongly imposed in the continuous energy space  
\[
\mathbb X^{\eps} := \Big\{(\btau, q)\in \widetilde {\mathbb X}^{\eps}:\quad \btau\bn^{\eps} + q\bn^{\eps}  = \0 \quad \text{on} \,\, \S^{\eps} \Big\}.
\]  
It is natural \cite{Ar1,BoffiBrezziFortin,CGG,abd} to take into consideration the symmetry of the stress tensor  weakly through the introduction of a Lagrange multiplier, which is given by the rotation 
$
\br^{\eps}:=\frac{1}{2}\big\{\nabla \bu^{\eps} - (\nabla \bu^{\eps})^{\top}\big\}
$ and belongs to the  space $\bcQ^{\eps}$ of skew symmetric tensors 
\[
\bcQ^{\eps}:= \big\{\bs \in [\L^2(\OS^{\eps})]^{\dxd}: \quad \bs^{\top} = -\bs\big\}.
\]
For brevity of notations we denote the Hilbertian product norm in
$\widetilde{\mathbb X}^{\eps}\times\bcQ^{\eps}$ by 
\[
\norma{ ((\btau,q),\bs) }_{\eps}^2
:=\|(\btau,q) \|_{\eps}^2+\|\bs\|_{0,\OS^{\eps}}^2
\quad \forall \, ((\btau,q),\bs) \in \widetilde {\mathbb X}^{\eps} \times \bcQ^{\eps}.
%:=\norm{(\btau,q)}^2+\norm{\bs}^2_{0,\OS^{\eps}}.
\].
We define the following bounded bilinear forms 
\begin{subequations}\label{def.adb}
\begin{align} &a_1^{\eps}:\mathbf{H}(\bdiv, \OS^{\eps}) \times \mathbf{H}(\bdiv, \OS^{\eps}) \rightarrow \R \quad \mbox{by} \quad a_1^{\eps}(\bsig,\btau)=
\int_{\OS^{\eps}}\frac{1}{\rS}\bdiv\bsig(\vecx) \cdot\bdiv\btau(\vecx)d\vecx, \\
 &a_2^{\eps}:\mathrm{H}^1(\OF^{\eps}) \times \mathrm{H}^1(\OF^{\eps}) \rightarrow \R \quad \mbox{by} \quad a_2^{\eps}(p,q)=
\int_{\OF^{\eps}}\frac{1}{\rF}\nabla p(\vecx)\cdot\nabla q(\vecx)d\vecx, \\
&d_1^{\eps}:\mathbf{H}(\bdiv, \OS^{\eps}) \times \mathbf{H}(\bdiv, \OS^{\eps})
\rightarrow \R \quad \mbox{by} \quad
d_1^{\eps}(\bsig,\btau)=
\int_{\OS^{\eps}}\cC^{-1}\bsig(\vecx):\btau(\vecx)d \vecx, \\
&d_2^{\eps}:\mathrm{H}^1(\OF) \times \mathrm{H}^1(\OF) \rightarrow \R \quad \mbox{by} \quad
d_2^{\eps}(p,q)=
\int_{\OF^{\eps}}\frac{1}{\rF c^2}p(\vecx)q(\vecx)d\vecx, \\
&b^{\eps}:\mathbf{H}(\bdiv, \OS^{\eps}) \times \bcQ^{\eps} \rightarrow \R \quad \mbox{by} \quad
b^{\eps}(\btau,\bs)
 =\int_{\OS^{\eps}}\btau(\vecx):\bs(\vecx)d\vecx, \\
 &\ell^{\eps}: \mathbf{H}(\bdiv, \OS^{\eps}) \rightarrow \R \quad \mbox{by} \quad
\ell^{\eps}(\btau^{\eps})= 
 \int_{\OS^{\eps}} \frac{1}{\rS} \bF \cdot \bdiv \btau^{\eps},
\end{align}
\end{subequations}
and denote
\begin{subequations}\label{equ:all bil for}
\begin{align}
&
a^{\eps}\big((\bsig,p),(\btau,q)\big) :=
a_1^{\eps}(\bsig,\btau)+a_2^{\eps}(p,q),
\label{equ:a a1 a2}
\\
&A^{\eps}\big((\bsig,p),(\btau,q)\big) 
:=
a^{\eps}\big((\bsig,p),(\btau,q)\big)
+
d_1^{\eps}(\bsig,\btau)+d_2^{\eps}(p,q),
\label{def.aa}\\
&\bbA^{\eps}\Big(\big((\bsig^{\eps}, p^{\eps}), \br^{\eps}\big), \big((\btau^{\eps}, q^{\eps}), \bs^{\eps}\big)\Big)
 :=A^{\eps}((\bsig^{\eps},p^{\eps}),(\btau^{\eps},q^{\eps}))
+b^{\eps}(\btau^{\eps},\br^{\eps})
+b^{\eps}(\bsig^{\eps},\bs^{\eps}),\label{def.bbA.ep}
\\[1ex]
&\bbB^{\eps}\Big(\big((\bsig^{\eps}, p^{\eps}), \br^{\eps}\big), \big((\btau^{\eps}, q^{\eps}), \bs^{\eps}\big)\Big)
 :=d_1^{\eps}(\bsig^{\eps},\btau^{\eps})+d_2^{\eps}(p^{\eps},q^{\eps})
+b^{\eps}(\btau^{\eps},\br^{\eps})
+b^{\eps}(\bsig^{\eps},\bs^{\eps}),\label{def.bbB.ep}\\[1ex]
&\mathbb D^{\eps}\Big(\big((\bsig^{\eps}, p^{\eps}), \br^{\eps}\big), \big((\btau^{\eps}, q^{\eps}), \bs^{\eps}\big)\Big)=
\bbA^{\eps}\Big(\big((\bsig^{\eps}, p^{\eps}), \br^{\eps}\big), \big((\btau^{\eps}, q^{\eps}), \bs^{\eps}\big)\Big)\notag\\
& \qquad \qquad \qquad \qquad \qquad \qquad  \qquad \quad - (1+\mu^2) \bbB^{\eps}\Big(\big((\bsig^{\eps}, p^{\eps}), \br^{\eps}\big), \big((\btau^{\eps}, q^{\eps}), \bs^{\eps}\big)\Big)\label{def.DD.eps}.
\end{align}
\end{subequations}
We point out that the kernel 
$\ker(a^{\eps}):= \set{(\btau,q)\in \mathbb{X}^{\eps}:
\ a^{\eps}\big((\btau, q), (\btau, q)\big) = 0}$ 
of the bilinear form $a^{\eps}(\cdot, \cdot)$ in $\mathbb X^{\eps}$ is given by 
\[
 \ker(a^{\eps})= \set{(\btau,q)\in \mathbb{X}^{\eps}:
\ \text{$\bdiv\btau=\0\text{ in }\OS^{\eps}$ and $q$ constant in $\OF^{\eps}$}}.
\]
Let us introduce the orthogonal complement to $\ker(a^{\eps}) \times \bcQ^{\eps}$ in $\mathbb{X}^{\eps} \times \bcQ^{\eps} $ with respect to the bilinear form $\mathbb{B}^{\eps}$
by
\begin{align}\label{equ:ker a eps} 
[\ker(a^{\eps}) \times \bcQ^{\eps}]^{\perp_{\mathbb{B^{\eps}}}}:=\{\big((\bsig^{\eps},p^{\eps}),\br^{\eps}\big) \in \mathbb{X}^{\eps} \times \bcQ^{\eps}: \mathbb{B}^{\eps}&\Big(\big((\bsig^{\eps},p^{\eps}\big),r^{\eps}),
(\big(\btau^{\eps},q^{\eps}),\bs^{\eps}\big)\Big)=0
\nonumber\\
&  \forall \big((\btau^{\eps},q^{\eps}),\bs^{\eps}\big)\in \ker(a^{\eps}) \times \bcQ^{\eps} \}.
\end{align}
%The following lemma shows that the bilinear form~$\bbA^\eps(\cdot,\cdot)$ restricted
%to~$[\ker(a^{\eps}) \times \bcQ^{\eps}]^{\perp_{\mathbb{B^{\eps}}}}$ is positive definite.
%\begin{lem} \label{lem.coerc.A} For all nonzero 
%$\big((\bsig^{\eps},p^{\eps}),\br^{\eps}\big) \in \kerab$,
%\[
%\mathbb{A}^{\eps}\Big(\big((\bsig^{\eps},p^{\eps}),\br^{\eps}\big),\big((\bsig^{\eps},p^{\eps}),\br^{\eps}\big)\Big) \geq \mathbb{B}^{\eps}\Big(((\bsig^{\eps},p^{\eps}),\br^{\eps}),((\bsig^{\eps},p^{\eps}),\br^{\eps})\Big)>0.
%\]
%\end{lem}
%\begin{proof}
%See \cite[Lemma 4.2]{MMR3}.
%\end{proof}

With all the bilinear forms defined in~\eqref{equ:all bil for}, we are now able to write the
weak formulation of problem~\eqref{equ:perturbed prob}.
Considering the body force $\bF\in \L^2(B_R)^d$, it is direct\dela{show that
 the variational formulation of \eqref{equ:perturbed prob} is given by (see \cite{gmm2011,MMR3} for more details): Find $(\bsig^{\eps},p^{\eps})\in \mathbb{X}^{\eps}$ and $\br^{\eps} \in \bcQ^{\eps}$ such that
\begin{equation}\label{varForm1.ep}
\begin{aligned}
a^{\eps}\big((\bsig^{\eps},p^{\eps}),(\btau^{\eps},q^{\eps})\big)
- \mu^2 \Big( b^{\eps} \big((\btau^{\eps},q^{\eps}),\br^{\eps}\big) + d^{\eps} \big((\bsig^{\eps},p^{\eps}),(\btau^{\eps},q^{\eps})\big)\Big) &= \underbrace{\displaystyle\int_{\OS^{\eps}}\frac{1}{\rS} \bF \cdot \bdiv \btau^{\eps}}_{l^{\eps}(\btau^\eps):=},
 \\[1ex]
b^{\eps}\big((\bsig^{\eps},p^{\eps}),\bs^{\eps}\big)
& = 0,
\end{aligned}
\end{equation}
for all $(\btau^{\eps},q^{\eps})\in\mathbb{X}^{\eps}$ and $\bs^{\eps}\in\bcQ^{\eps}$.
It will be convenient} to write \eqref{equ:perturbed prob} in the  equivalent tensorial form (see \cite{gmm2011,MMR3} for more details): Find $\big((\bsig^{\eps}, p^{\eps}), \br^{\eps}\big)\in \mathbb{X}^{\eps}\times\bcQ^{\eps}$ such that 
\begin{equation}\label{compactVar.eps}
\begin{array}{l}
 \mathbb D^{\eps}\Big(\big((\bsig^{\eps}, p^{\eps}), \br^{\eps}\big), \big((\btau^{\eps}, q^{\eps}), \bs^{\eps}\big)\Big)=  \ell^{\eps}(\btau^{\eps}) \quad \forall \,
 ((\btau^{\eps}, q^{\eps}), \bs^{\eps})\in\mathbb{X}^{\eps}\times\bcQ^{\eps}.
 \end{array}
\end{equation}

\begin{rem} 
\begin{itemize}
\item[(i)]	We note that variational formulation \eqref{compactVar.eps} has been designed in terms of the stress tensor $\bsig$ of the solid (not displacement vector field $\bu$) and pressure $p$ of the fluid. However, using the equilibrium equation \eqref{dual-problem_1a.ep} one can recover the displacement vector field.
%\medskip
\item[(ii)] According to \cite{MMT} and references cited therein, for $\OS \in C^2,$ the displacement field $\bu$ that solves this problem belongs to $\H^{1+\alpha}(\OS)^d$ for all $\alpha \in (1/2,1)$.
\dela{
 for the unperturbed problem we denote $D^0=D,$ where $D$ can be the domain or the solution space. It is noteworthy to mention that $\eps=0$ corresponds to  the unperturbed problem (see Problem \ref{compactVar.rep} in Section \ref{sec.shape.der} which is true for all 
$((\btau, q), \bs)\in\mathbb{X}\times\bcQ$). When there is no ambiguity, we use $D$ for $\OS$
or $\OF$ or $\mathbb{X}$ or $\bcQ$.
\item[(iv)] Similar to \eqref{def.adb}, the bilinear forms $a_1,\,d_1,\, b,\,\ell$ and $d_2$ can be defined 
on domains $\OS$ and $\OF$ respectively for $\eps=0.$ The bilinear maps $\mathbb D,\,\bbA$ and $\bbB$ corresponds for the unperturbed problem (when 
$\eps=0$).}
 \end{itemize}
\end{rem}

\section{The solution operators~$\mathcal{S}^\eps$}\label{sec:spe sol ope}
The shape calculus technique to be used in the next section
is originated from shape optimization in the deterministic framework; see \cite{Has,SokZol92}. 
For this reason and for simplicity, we temporarily escape randomness and consider only deterministic
perturbed interfaces.

\subsection{Representation of perturbed deterministic model}
We now present some properties of perturbed interfaces which are required for the subsequent analysis.
Let $\tilde{\kappa}$ and $\tilde{\n}$ be any smoothness preserving extension of $\kappa$ and $\n$ on $\mathbb{R}^3 $ such that $\tilde{\kappa} \in \mathbb{W}^{1,\infty}(\R^3)\cap C^{2,1}(\R^3).$
For~$\eps\ge0$, we define $\bT^{\eps}:\R^3 \rightarrow \R^3$ by
\begin{align}\label{equ:T eps define}
\bT^{\eps}(\vecx)=\vecx+\eps \tilde{\kappa}(\vecx)\tilde{\n}(\vecx) \quad\forall \, \vecx \in \R^3.
\end{align}
Without loss of generality we assume that the extension $\tilde{\kappa}$ vanishes outside a sufficiently large ball~$B_R$ (with origin as the centre and radius $R$) containing $\OS^{\eps} \cup \OF^{\eps}$ for $0 \leq \eps \leq \eps^0,$ for some $\eps^0>0.$ This implies that the perturbation mapping $\bT^{\eps}(\vecx)$ 
is an identity in the complement $B_R^c := \R^3 \setminus \overline{B_R}$, i.e.
\begin{equation}\label{equ:T e cond}
\bT^{\eps}(\vecx) 
=
\vecx
\quad
\forall \, \vecx
\in  {B_{R}^c}.
\end{equation}
For ease of notation, throughout the paper, we denote by $\O^{\eps}:=\bT^{\eps}(\O)$ either
$\OS^{\eps}$ or $\OF^{\eps}$ when there is no ambiguity.
For convenience, we abbreviate
\begin{equation}\label{equ:V def.main}
\mathbf{V}(\vecx)
:=
\tilde\kappa(\vecx)\tilde\vecn{}(\vecx),
\quad 
\vecx\in \R^3.  
\end{equation}
In~\cite{SokZol92}, the field $\mathbf{V}$ is called \emph{the velocity 
field} of the mapping $\bT^\eps$ and in \cite{H10,HSS08}, $\mathbf{V}$ is known as the boundary perturbation field in the normal direction.
From \eqref{equ:V def.main}, one can observe that $\kappa= \langle \bV,\bn \rangle.$
\begin{rem}\label{rem:eps 0}
When $\eps=0,$ we omit the superscript $\eps$ in the notations of the spaces, norms and bilinear forms.
\end{rem} 

\subsection{The solution operators and their spectra}\label{subsec.spectralT}
%Spectral properties of solution operator} 
In this section we define the solution operators $\mathcal{S}^{\eps}$ and study their properties which will
be used to prove the existence of the material derivative in the next section. These properties have
been studied in~\cite{MR3376135,MMR3}. However, since we need to apply these results with different
values of~$\eps>0$ and pass to the limit when~$\eps\to0$, it is important to check the
estimates to ensure that they are independent of~$\eps$.

For each $\epsilon \geq 0,$ let us introduce the operator 
\begin{align*}
 \mathcal{S}^{\eps}:\ \mathbb{X}^{\eps} \times\bcQ^{\eps} & \longrightarrow \mathbb{X}^{\eps}\times\bcQ^{\eps},
\\
\big((\boldsymbol{F}^{\eps},f^{\eps}),\boldsymbol{G}^{\eps}\big) & \longmapsto 
\big(\left(\bsig^{\eps}_*,p^{\eps}_*\right),\br^{\eps}_*\big)
=
\mathcal{S}^{\eps}\big((\boldsymbol{F}^{\eps},f^{\eps}),\boldsymbol{G}^{\eps}\big)
\end{align*}
where $\big((\bsig^{\eps}_*,p^{\eps}_*),\br^{\eps}_*\big)\in
\mathbb{X}^{\eps}\times\bcQ^{\eps}$ satisfies, for all
$\big((\btau^{\eps},q^{\eps}),\bs^{\eps}\big)\in\mathbb{X}^{\eps} \times \bcQ^{\eps}$,
\begin{align}\label{defT.eps}
\bbA^{\eps}\Big(
\big(\left(\bsig^{\eps}_*,p^{\eps}_*\right),\br^{\eps}_*\big),
\big((\btau^{\eps},q^{\eps}),\bs^{\eps}\big)
\Big)
&= \bbB^{\eps}\Big(\big((\boldsymbol{F}^{\eps},f^{\eps}),\boldsymbol{G}^{\eps}\big),\big((\btau^{\eps},q^{\eps}),\bs^{\eps}\big)\Big)
\dela{\quad \forall \big((\btau^{\eps},q^{\eps}),\bs^{\eps}\big)\in \mathbb X^{\eps} \times \bcQ^{\eps}}.
\end{align}
The well definiteness and symmetry with respect to the bilinear
form~$\bbA^{\eps}(\cdot,\cdot)$ of this operator~$\mathcal{S}^{\eps}$ is proved
in~\cite[Lemma~3.2]{MMT}.
To focus on the solution of the problem, we first characterize the spectral properties of the
operator $\mathcal{S}^{\eps}$ for each $\eps\ge0.$

\begin{lem}\label{specT}
For~$\eps\ge0$, the spectrum $\sp(\mathcal{S}^{\eps})$ of $\mathcal{S}^{\eps}$ decomposes as follows 
 \[
  \sp(\mathcal{S}^{\eps}) = \set{0, 1} \cup \set{\eta_k(\eps)}_{k\in \mathbb{N}}
 \]
where~$\set{\eta_k(\eps)}_{k\in \mathbb{N}}$ satisfying
\begin{align} \label{ev.dec}
1 > \eta_1(\eps) \ge \cdots \ge \eta_k(\eps) \ge \cdots > 0
\end{align}
is a decreasing sequence of finite-multiplicity eigenvalues of~$\mathcal{S}^{\eps}$ which
converges to~0. 
Moreover,~$1$ is an infinite-multiplicity eigenvalue of $\mathcal{S}^{\eps}$ while $0$ is not 
an eigenvalue. The associated eigenspace of the eigenvalue~1 is~$\ker(a^{\eps}) \times
\bcQ^{\eps}$.
\end{lem}
\begin{proof}
See \cite[Section~4]{MMR3}.
\end{proof}

It is proved in \cite[Theorem~3.1]{MMT} that if the input frequency~$\mu$,
see~\eqref{def.DD.eps}, is chosen such that~$1/(1+\mu^2)\notin\sp(\mathcal{S}^{\eps})$ then the
problem~\eqref{compactVar.eps} is well posed. To ensure that such a choice of~$\mu$ is possible
for all~$\eps\ge0$ sufficiently small, it is necessary to prove that there exists~$\eps_0>0$
such that 
%\begin{equation}\label{equ:sp S eps}
\[
\overline{\bigcup_{0\le\eps\le\eps_0} \sp(\mathcal{S}^\eps)} \not= [0,1].
\]
%\end{equation}
In fact, we will prove a stronger result that there exists~$\delta>0$ such that,
for all nonnegative~$\eps$ sufficiently small, all the eigenvalues~$\eta_k(\eps)$ 
are crowded to the left of~$\eta_1(0)+\delta$.
This result, which has its own interest, is stated in the following proposition.
\begin{prop}\label{prop.spT}
For each $\delta\in (0,1-\eta_1(0))$, there exists $\eps_0>0$ such that 
\begin{align} \label{spT.prop1.mod}
(\eta_1(0)+\delta,1) \subset [0,1]\setminus B 
\end{align}
where $B:=\overline{\bigcup_{0 \le \eps \leq \eps_0}
  \sp(\mathcal{S}^{\eps})}$.
\end{prop}
\begin{proof}
Noting the decrease property~\eqref{ev.dec}, it suffices to prove that
\begin{equation}\label{equ:eta eps 0}
\liminf_{\eps \rightarrow 0} 
\dfrac{1}{\eta_1(\eps)} \geq \dfrac{1}{\eta_1(0)}.
\end{equation}
Indeed, assume that~\eqref{equ:eta eps 0} holds. Let us show that~\eqref{spT.prop1.mod} holds.
For each $\delta\in (0,1-\eta_1(0))$, 
let~$\delta_1={\delta}/{[\eta_1^2(0)+\eta_1(0)\delta]} > 0$.
By the definition of $\liminf$, there exists $\eps_0>0$ such that
\begin{align*}
\dfrac{1}{\eta_1(0)}-\delta_1 < \dfrac{1}{\eta_1(\eps)} \quad \forall \, 0 < \eps \leq \eps_0,
\end{align*}
proving that
\[ 
\eta_1(\eps) < \eta_1(0)+\delta \quad \forall\, 0 < \eps \leq \eps_0,
\] 
which concludes \eqref{spT.prop1.mod} due to~\eqref{ev.dec}. 

%Noting the decrease property~\eqref{ev.dec}, it suffices to prove that
%\begin{equation}\label{equ:eta eps 0}
%\eta_1(\eps) \rightarrow \eta_1(0) \quad\text{as}\quad \eps \rightarrow 0.
%\end{equation}
We now prove~\eqref{equ:eta eps 0}.
Since~$\ker(a^\eps)\times\bcQ^\eps$ is the eigenspace associated with the eigenvalue~1, see
Lemma~\ref{specT}, the eigenspace associated with~$\eta_1(\eps)$ is a subspace
of~$\kerab$ which is defined in~\eqref{equ:ker a eps}.
As a consequence, we derive that for all~$((\bsig^{\eps},p^{\eps}),\br^{\eps}) \in \kerab$
\[
\mathbb{A}^{\eps}
\Big(
\big((\bsig_1^{\eps},p_1^{\eps}),\br_1^{\eps}\big),
\big((\bsig^{\eps},p^{\eps}),\br^{\eps}\big)
\Big) 
=
\dfrac{1}{\eta_1(\eps)} 
\mathbb{B}^{\eps}
\Big(
\big((\bsig_1^{\eps},p_1^{\eps}),\br_1^{\eps}\big),
\big((\bsig^{\eps},p^{\eps}),\br^{\eps}\big)
\Big)
\]
where $\big((\bsig_1^{\eps},p_1^{\eps}),\br_1^{\eps}\big)$ is an eigenvector associated with
the eigenvalue~$\eta_1(\eps)$. The Rayleigh quotient gives
%\begin{align}\label{def.R.eta}
\[
\dfrac{1}{\eta_1(\eps)}
=
\min_{\0 \neq ((\bsig^{\eps},p^{\eps}),\br^{\eps}) \in \kerab} 
\dfrac{\mathbb{A}^{\eps}\Big(((\bsig^{\eps},p^{\eps}),\br^{\eps}),((\bsig^{\eps},p^{\eps}),\br^{\eps})\Big)}{\mathbb{B}^{\eps}\Big(((\bsig^{\eps},p^{\eps}),\br^{\eps}),((\bsig^{\eps},p^{\eps}),\br^{\eps})\Big)}.
\]
%\end{align}
Denoting
\[
\mathbb{R}^{\eps}((\bsig^{\eps},p^{\eps}),\br^{\eps})
:=
\dfrac{
\mathbb{A}^{\eps}\Big(((\bsig^{\eps},p^{\eps}),\br^{\eps}),((\bsig^{\eps},p^{\eps}),\br^{\eps})\Big)}
{\mathbb{B}^{\eps}\Big(((\bsig^{\eps},p^{\eps}),\br^{\eps}),((\bsig^{\eps},p^{\eps}),\br^{\eps})\Big)}
\]
and using~\eqref{def.adb} and~\eqref{equ:all bil for} we deduce
\begin{equation}\label{sum.defi}
\mathbb{R}^{\eps}((\bsig^{\eps},p^{\eps}),\br^{\eps})=
\dfrac{a_1^{\eps}(\bsig^{\eps},\bsig^{\eps})+a_2^{\eps}(p^{\eps},p^{\eps})+d_1^{\eps}(\bsig^{\eps},\bsig^{\eps})+d_2^{\eps}(p^{\eps},p^{\eps})+2b^{\eps}(\bsig^{\eps},\br^{\eps})}{d_1^{\eps}(\bsig^{\eps},\bsig^{\eps})+d_2^{\eps}(p^{\eps},p^{\eps})+2b^{\eps}(\bsig^{\eps},\br^{\eps})}.
\end{equation}
We will show that (noting the notation convention in Remark~\ref{rem:eps 0})
\begin{align}\label{sum.Ri}
\mathbb{R}^{\eps}((\bsig^{\eps},p^{\eps}),\br^{\eps})
=
\mathbb{R}((\bsig^{\eps}\circ\bT^{\eps},p^{\eps}\circ \bT^{\eps}),\br^{\eps}\circ \bT^{\eps})
+\eps G((\bsig^{\eps}\circ\bT^{\eps},p^{\eps}\circ \bT^{\eps}),\br^{\eps}\circ \bT^{\eps})
\end{align}
where $G$ is a mapping from~$[\ker(a)\times\bcQ]^{\perp_{\mathbb{B}}}$ to~$\R$.
%is the combination of $g_i's;i=1,\dots,5$ and $g_i's$ are given in \eqref{def.g1} and \eqref{def.gi}.
Letting~$\bcQ_a^\eps:=\ker(a^{\eps}) \times \bcQ^{\eps}$,
then~$[\bcQ_a^\eps]^{\perp_{\mathbb{B^\eps}}}$. Since
\begin{align*}
\left.\bT^{\eps}\right|_{\bcQ_a} : \bcQ_a \to \bcQ_a^\eps
\quad\text{and}\quad
\left.\bT^{\eps}\right|_{[\bcQ_a]^{\perp_{\mathbb{B}}}} 
: [\bcQ_a]^{\perp_{\mathbb{B}}}
\to [\bcQ_a^\eps]^{\perp_{\mathbb{B^\eps}}},
\end{align*}
are bijective, it follows from~\eqref{sum.Ri} that
\begin{align*}
\inf_{\0 \neq ((\bsig^{\eps},p^{\eps}),\br^{\eps}) \in [\bcQ_a^\eps]^{\perp_{\mathbb{B^\eps}}}} 
\mathbb{R}^{\eps}((\bsig^{\eps},p^{\eps}),\br^{\eps})
%&=
%\inf_{\0 \neq ((\bsig,p), \br ) \in [\bcQ_a]^{\perp_{\mathbb{B}}}}
%\Big(
%\mathbb{R}((\bsig,p),\br)
%+\eps G((\bsig,p),\br)
%\Big)
%\\
&\ge
\inf_{\0 \neq ((\bsig,p), \br ) \in [\bcQ_a]^{\perp_{\mathbb{B}}}}
\mathbb{R}((\bsig,p),\br)
\\
&\quad
+
\eps
\inf_{\0 \neq ((\bsig,p), \br ) \in [\bcQ_a]^{\perp_{\mathbb{B}}}}
G((\bsig,p),\br)
\end{align*}
which proves \eqref{equ:eta eps 0} by letting~$\eps\to0$.
%One can observe that there is a one-to-one correspondence between the spaces $\mbox{ker}(a^{\eps}) \times \bcQ^{\eps}$ and $\mbox{ker}(a) \times \bcQ$ and between the spaces $\kerab$ and $[\mbox{ker}(a) \times \bcQ]^{\perp_{\mathbb{B}}}.$ 
%Hence using equations \eqref{def.R.eta}, \eqref{sum.defi}, we see that 
%\eqref{sum.defi}, \eqref{K1.eps.change}-\eqref{Ki.eps.change}, we see that 
%\begin{align*}
%&\dfrac{1}{\eta_1(\eps)}
%=
%\inf_{\0 \neq ((\bsig^{\eps} \circ \bT^{\eps},\big(p^{\eps} \circ \bT^{\eps}\big)),\br^{\eps} \circ \bT^{\eps}) \in [\mbox{ker}(a) \times \bcQ]
%^{\perp_{\mathbb{B}}}} 
%\mathbb{R}^{\eps}((\bsig^{\eps},p^{\eps}),\br^{\eps})+o(\eps)
%\geq \dfrac{1}{\eta_1(0)}+O(\eps),
%\end{align*}
%which implies \eqref{equ:eta eps 0}.
%\adda{STILL TO BE CHECKED!}
\par
We now move to prove \eqref{sum.Ri}. Following the notations of Kronecker product mentioned in
Appendix \ref{sec.tensor}, we now
employ  change of variables $\vecx=\bT^{\eps}(\vecy)$ to the bilinear form
$a_1^\eps(\cdot,\cdot)$. We
recall here Appendix \ref{sec.app.shape} to introduce
$\gamma(\eps,\cdot),J_{\bT^\eps},J_{\bT^\eps}^{-1},\mathcal{A}$ and $\tilde{\mathcal{A}}$.  
\dela{There is a bijection between the elements of $\mathbb X^{\eps} \times \bcQ^{\eps}$ and
$\mathbb X \times \bcQ$ via the map $\bT^{\eps}.$ Hence for
$\big((\btau^{\eps},q^{\eps}),\bs^{\eps}  \big) \in \mathbb X^{\eps} \times \bcQ^{\eps},$ we
denote the elements $\big((\btau^{\eps} \circ \bT^{\eps},q^{\eps} \circ \bT^{\eps}),\bs^{\eps}
\circ \bT^{\eps} \big)$ by $\big((\btau,q),\bs\big)$ in $\mathbb{X} \times \bcQ.$}Using $|\bdiv
(\bsig(\vecx))|^2=| \cL_{\bf I}(\bsig(\vecx))|^2$ in $a_1^{\eps}(\bsig^{\eps},\bsig^{\eps})$
and equation \eqref{equ:vbn 4} and Lemma \ref{lem.conv.sig}, we see that 
%\begin{align*}
%a_1^{\eps}(\bsig^{\eps},\bsig^{\eps})
%&
%=\int_{\OS}\frac{1}{\rS}\gamma(\eps,\vecy)
%\Big|\cL_{J_{\bT^{\eps}}^{-1}}
%(\bsig^{\eps} \circ \bT^{\eps}(\vecy))\Big|^2 d \vecy\notag\\
%&
%=\int_{\OS}\frac{1}{\rS}\gamma(\eps,\vecy)
%\Big|\big(\cL_{J_{\bT^{\eps}}^{-1}}
%-\cL_{\bf I}\big)(\bsig^{\eps} \circ \bT^{\eps}(\vecy))\Big|^2 d \vecy\notag\\
%& \quad+
%\int_{\OS}\frac{2}{\rS}\gamma(\eps,\vecy)\Big[\big(\cL_{J_{\bT^{\eps}}^{-1}}
%-\cL_{\bf I}\big)(\bsig^{\eps} \circ \bT^{\eps}(\vecy)): \cL_{\bf I}(\bsig^{\eps} \circ \bT^{\eps}(\vecy)) \Big]d \vecy \notag\\
%& \quad+ \int_{\OS}\frac{1}{\rS}\gamma(\eps,\vecy) \Big|\cL_{\bf I}(\bsig^{\eps} \circ \bT^{\eps}(\vecy))\Big|^2  d \vecy\notag\\
%\dela{&=\int_{\OS}\frac{1}{\rS}|\bdiv \bsig^{\eps} \circ \bT^{\eps}(\vecy)|^2  \gamma(\eps,\vecy) d \vecy +O(\eps)\notag\\
%&=\int_{\OS}\frac{1}{\rS}|\bdiv \bsig^{\eps} \circ \bT^{\eps}(\vecy)|^2 d \vecy 
%+\int_{\OS}\frac{1}{\rS}|\bdiv \bsig^{\eps} \circ \bT^{\eps}(\vecy)|^2  (\gamma(\eps,\vecy)-1) d \vecy +O(\eps)\notag\\}
%&=\int_{\OS}\frac{1}{\rS}|\bdiv \big(\bsig^{\eps} \circ \bT^{\eps}(\vecy)\big)|^2 d \vecy+O(\eps)\notag\\
%&=a_1(\bsig^{\eps} \circ \bT^{\eps},\bsig^{\eps} \circ \bT^{\eps})+\eps g_1(\bsig^{\eps} \circ \bT^{\eps},\bsig^{\eps} \circ \bT^{\eps})
%\end{align*}
\begin{align*}
a_1^{\eps}(\bsig^{\eps},\bsig^{\eps})
&
=\int_{\OS}\frac{1}{\rS}\gamma(\eps,\vecy)
\Big|\cL_{J_{\bT^{\eps}}^{-1}}
(\bsig^{\eps} \circ \bT^{\eps}(\vecy))\Big|^2 d \vecy\notag\\
&
=\int_{\OS}\frac{1}{\rS}\big(1+\eps \tilde{\gamma}(\eps,\vecy)\big)
\Big|\cL_{\bf I}
(\bsig^{\eps} \circ \bT^{\eps}(\vecy))
+\eps \cL_{\hat{\bV}_1}(\bsig^{\eps} \circ \bT^{\eps}(\vecy))
\Big|^2 d \vecy\notag\\
&
=\int_{\OS}\frac{1}{\rS}\big(1+\eps \tilde{\gamma}(\eps,\vecy)\big)
\Big\{
\Big|\cL_{\bf I}
(\bsig^{\eps} \circ \bT^{\eps}(\vecy))\Big|^2
 +\eps^2 \Big|\cL_{\hat{\bV}_1}(\bsig^{\eps} \circ \bT^{\eps}(\vecy))
\Big|^2
\notag\\ &\quad \quad
+2 \eps 
\cL_{\bf I} (\bsig^{\eps} \circ \bT^{\eps}(\vecy)): 
 \cL_{\hat{\bV}_1}(\bsig^{\eps} \circ \bT^{\eps}(\vecy))
\Big\} d \vecy\notag\\
&
=\int_{\OS}\frac{1}{\rS}
\Big|\cL_{\bf I}
(\bsig^{\eps} \circ \bT^{\eps}(\vecy))\Big|^2 d \vecy
 + \int_{\OS}\frac{1}{\rS} \eps \tilde{\gamma}(\eps,\vecy) \Big|\cL_{\bf I}
(\bsig^{\eps} \circ \bT^{\eps}(\vecy))\Big|^2 d \vecy 
\notag\\ & \quad
 + \eps \int_{\OS} \frac{1}{\rS} 
 \big(1+\eps \tilde{\gamma}(\eps,\vecy)\big) \Big[2
\cL_{\bf I} (\bsig^{\eps} \circ \bT^{\eps}(\vecy)): 
 \cL_{\hat{\bV}_1}(\bsig^{\eps} \circ \bT^{\eps}(\vecy))+ \eps \Big|\cL_{\hat{\bV}_1}(\bsig^{\eps} \circ \bT^{\eps}(\vecy)) \Big|^2
\Big\} d \vecy\notag\\
&
=a_1(\bsig^{\eps} \circ \bT^{\eps},\bsig^{\eps} \circ \bT^{\eps})+ \eps g_1 (\bsig^{\eps} \circ \bT^{\eps})
\end{align*}
where
\begin{align*}%\label{def.g1}
g_1(\bsig^{\eps} \circ \bT^{\eps})&= 
  \int_{\OS}\frac{1}{\rS} \tilde{\gamma}(\eps,\vecy) \Big|\cL_{\bf I}
(\bsig^{\eps} \circ \bT^{\eps}(\vecy))\Big|^2 d \vecy 
 + \int_{\OS} \frac{1}{\rS} 
 \big(1+\eps \tilde{\gamma}(\eps,\vecy)\big) \Big[2
\cL_{\bf I} (\bsig^{\eps} \circ \bT^{\eps}(\vecy)) 
\notag\\ & \quad \quad
 :\cL_{\hat{\bV}_1}(\bsig^{\eps} \circ \bT^{\eps}(\vecy))+ \eps \Big|\cL_{\hat{\bV}_1}(\bsig^{\eps} \circ \bT^{\eps}(\vecy)) \Big|^2
\Big\} d \vecy
\end{align*}
%\[g_1(\bsig^{\eps} \circ \bT^{\eps},\bsig^{\eps} \circ \bT^{\eps})
%=C(\eps)C^2 \eps +\dfrac{2}{\rS} C(\eps) \Big(\int_{\OS} |\bdiv (\bsig^{\eps} \circ \bT^{\eps})|^2 dx \Big)^{1/2}
%+\dfrac{1}{\rS} \int_{\OS} \tilde{\gamma}(\eps,\vecy) |\bdiv (\bsig^{\eps} \circ \bT^{\eps})|^2 dx \Big)^{1/2}
%\]
Repeating the similar arguments for each bounded bilinear maps on the
right hand side of \eqref{sum.defi}, one achieve
%\begin{subequations}\label{def.gi}
\begin{align*}
a_2^{\eps}(p^{\eps},p^{\eps})
&=a_2(p^{\eps} \circ \bT^{\eps},p^{\eps} \circ \bT^{\eps})+\eps g_2(p^{\eps} \circ \bT^{\eps}),\\
d_1^{\eps}(\bsig^{\eps},\bsig^{\eps})
&=d_1(\bsig^{\eps} \circ \bT^{\eps},\bsig^{\eps} \circ \bT^{\eps})+\eps g_3(\bsig^{\eps} \circ \bT^{\eps}),
\\
d_2^{\eps}(p^{\eps},p^{\eps}) &=d_2(p^{\eps} \circ \bT^{\eps},p^{\eps} \circ \bT^{\eps})+\eps g_4(p^{\eps} \circ \bT^{\eps}),\\
b^{\eps}(\bsig^{\eps},\br^{\eps})&=b({\bsig^{\eps} \circ \bT^{\eps},\br \circ \bT^{\eps}})+\eps g_5({\bsig^{\eps} \circ \bT^{\eps},\br \circ \bT^{\eps}}) .
\end{align*}
%\end{subequations}
for some bounded functions $g_i;i=2,\dots,5$.
This proves \eqref{sum.Ri}, finishing the proof of the proposition.
%\begin{align*}
%a_2^{\eps}(p^{\eps},p^{\eps})
%& = \int_{\OF}\frac{1}{\rF}\Big[\Big|(J_{\bT^{\eps}}^{-T}-I) \nabla \big(p^{\eps} \circ \bT^{\eps}\big)(\vecy)\Big|^2 +\Big|\nabla \big(p^{\eps} \circ \bT^{\eps}\big)(\vecy)\Big|^2\Big] \gamma(\eps,\vecy) d\vecy \notag\\
%& \quad + \int_{\OF}\frac{2}{\rF}(J_{\bT^{\eps}}^{-T}-I) \nabla \big(p^{\eps} \circ \bT^{\eps}\big)(\vecy)\cdot \nabla \big(p^{\eps} \circ \bT^{\eps}\big)(\vecy) \gamma(\eps,\vecy) d\vecy
%\notag\\
%&=\int_{\OF}\frac{1}{\rF}\Big|\nabla \big(p^{\eps} \circ \bT^{\eps}\big)(\vecy)\Big|^2 d\vecy+O(\eps)\notag\\
%&=a_2(\big(p^{\eps} \circ \bT^{\eps}\big),\big(p^{\eps} \circ \bT^{\eps}\big))+O(\eps),\\
%d_1^{\eps}(\bsig^{\eps},\bsig^{\eps})&=\int_{\OS}\cC^{-1}
%\bsig^{\eps} \circ \bT^{\eps}(\vecy)
%:\bsig^{\eps}\circ \bT^{\eps}(\vecy) \gamma(\eps,\vecy) d \vecy \notag\\
%&=d_1(\bsig^{\eps} \circ \bT^{\eps},\bsig^{\eps} \circ \bT^{\eps})+O(\eps),\\
%d_2^{\eps}(p^{\eps},p^{\eps})&= \int_{\OF^{\eps}}\frac{1}{\rF c^2}|p^{\eps} \circ \bT^{\eps}(\vecy)|^2 \gamma(\eps,\vecy) d\vecy\notag\\
%&=d_2(\big(p^{\eps} \circ \bT^{\eps}\big),\big(p^{\eps} \circ \bT^{\eps}\big))+O(\eps),\\
%b^{\eps}(\bsig^{\eps},\br^{\eps})&=b({\bsig^{\eps} \circ \bT^{\eps},\br \circ \bT^{\eps}})+O(\eps).
%\end{align*}This completes the proof.
 \end{proof}

The following result is similar to~\cite[Proposition 2.4]{MR3376135}.
However, here it is necessary to check that the constant is independent of~$\eps$. 
\begin{prop}\label{specT1}
If $1/(1+\mu^2)>\eta_1(0)$ then
% where~$B$ is defined in Proposition~\ref{prop.spT},
there exist~$\eps_0>0$ and a constant $C$ depending only on~$\eps_0$ such that
for all~$\eps\in[0,\eps_0]$ the following inequality holds
\begin{equation}\label{resolvent.eps}
\norma{\big(\frac{\mathbf{I}}{1 + \mu^2}-\mathcal{S}^{\eps}\big) \big( (\bsig^{\eps}, p^{\eps}), \br^{\eps} \big)}_{\eps}
\ge\, C \,  \norma{\big( (\bsig^{\eps}, p^{\eps}), \br^{\eps}\big)}_{\eps} \quad \forall \big( (\bsig^{\eps}, p^{\eps}), \br^{\eps}\big)\in 
\mathbb{X}^{\eps} \times \bcQ^{\eps}.
\end{equation}
\end{prop}

\begin{proof}	
Proposition~2.4 in~\cite{MR3376135} states that 
\begin{equation}\label{con.eps.ind}
	\norma{\big(\frac{\mathbf{I}}{1 + \mu^2}-\mathcal{S}^{\eps}\big) \big( (\bsig^{\eps},
p^{\eps}), \br^{\eps} \big)}_{\eps} \ge\, C^\ast(\eps)
	\delta_\mu(\mathcal{S}^{\eps})\,  \norma{\big( (\bsig^{\eps}, p^{\eps}), \br^{\eps}\big)}_{\eps} \quad \forall \big( (\bsig^{\eps}, p^{\eps}), \br^{\eps}\big)\in 
	\mathbb{X}^{\eps} \times \bcQ^{\eps},
	\end{equation}  
where~$C^\ast(\eps)$ is a positive constant independent of~$ \big( (\bsig^{\eps}, p^{\eps}),
\br^{\eps}\big)$ and
\[
0<\delta_\mu(\mathcal{S}^{\eps}) :=\dist\big(\frac{1}{1+ \mu^2},  \sp(\mathcal{S}^{\eps}) \big)<1
\] 
represents the distance between ${1}/(1+ \mu^2)$ and the spectrum  of $\mathcal{S}^{\eps}$.
First we show that~$\delta_{\mu}(\mathcal{S}^{\eps})$ is bounded below by a constant
independent of~$\eps$. Due to the assumption ${1}/(1+\mu^2)>\eta_1(0)$, 
we can invoke Proposition~\ref{prop.spT} to obtain $\eps_0>0$ satisfying
$$
\eta_1(0) < \eta_1(\eps) <\eta_1(0)+\delta \le \dfrac{1}{1+\mu^2}<1\quad \forall \eps \in
[0,\eps_0]
$$
where~$\delta$ is some positive number.
By virtue of~\eqref{ev.dec}, we have
\begin{align*}
\eta_k(\eps) 
<\eta_1(0)+\delta
<\dfrac{1}{1+\mu^2}<1 \quad \mbox{for all} \quad k \geq 1 \quad \mbox{and}\quad \eps \in [0,\eps_0].
\end{align*}
Hence,
\begin{align*}\label{dist.eq.eps}
\delta_\mu(\mathcal{S}^{\eps}) 
&=
\min_{k \geq 1} \Big(1-\dfrac{1}{1+\mu^2},\dfrac{1}{1+\mu^2}-\eta_k(\eps)\Big)
\geq \min \Big(\dfrac{\mu^2}{1+\mu^2},\dfrac{1}{1+\mu^2}-(\eta_1(0)+\delta)\Big):=c.
\end{align*}

Next we trace the constant~$C^\ast(\eps)$ in~\eqref{con.eps.ind} to show that it is bounded
below by a constant independent
of~$\eps$. Following the proof of Proposition~2.4 in~\cite{MR3376135} this constant~$C^\ast$
depends on the constant~$c_1$ in Proposition~2.1 of the same paper. This constant in turn
depends on~$\alpha$ in~\cite[Lemma~2.1]{MedMorRod13}. This constant~$\alpha$ depends on the
constant~$c$ in~\cite[Proposition~IV.3.1]{BrezziFortin91} and~$c_2$ in~\cite[Lemma~2.2]{Gat06}.
Proposition~IV.3.1 of~\cite{BrezziFortin91} is in fact Lemma~III.3.2 of~\cite{Gal11}. Tracing
all these constants one can check that they depend continuously on the measure of the domain,
namely~$|\OS^\eps|$. Since~$|\OS^\eps|$ depends continuously on~$\eps$, see \cite{Gal11}, so does the
constant~$C^\ast(\eps)$. Because~$\eps\in[0,\eps_0]$, this continuity implies
that~$C^\ast(\eps)$ has a minimum value which is positive. This proves the proposition.
\end{proof}

\begin{prop}\label{wellposed}
Let $\eps_0$ and $B$ be given in Proposition \ref{specT1}. If $(\frac{1}{1+\mu^2},1) \subset
[0,1]\setminus B,$  then there exists a positive constant $C$ depending only on~$\eps_0$
such that, for any $\bF\in (\mathrm{L}^2(B_R))^d$ and any~$\eps\in[0,\eps_0]$, the solution
$\big( (\bsig^{\eps}, p^{\eps}), \br^{\eps}\big)\in \mathbb{X}^{\eps}\times \bcQ^{\eps}$ of
\eqref{compactVar.eps} satisfies
\begin{equation}\label{estimsigr.Teps}
 \norma{\big( (\bsig^{\eps}\circ \bT^{\eps}, p^{\eps}\circ \bT^{\eps}), \br^{\eps}\circ \bT^{\eps}\big)} \leq \frac{C}{ 1+\mu^2}  \|\bF\|_{0,\OS}.
\end{equation}
\end{prop}

\begin{proof} 
By following the proof of~\cite[Theorem~3.1]{MMT} we can prove that
\[
\norma{\big( (\bsig^{\eps}, p^{\eps}), \br^{\eps}\big)}_{\eps} \leq \frac{C}{ 1+\mu^2}
\|\bF\|_{0,\OS},
\]
where the constant~$C$ comes from Proposition~\ref{specT1} which is independent of~$\eps$.
Then by change of variable formula, we have \eqref{estimsigr.Teps}. This completes the proof.
\end{proof}

%\begin{lem}\label{lem.changeofvariable} Assume $\bF \in (\mathrm{L}^2(B_R))^d \cap (\mathbb{X} \times \bcQ)^{*}$ and $\kappa \in C^1(\partial \OS)$. Let $((\bsig^{\eps},p^{\eps}),\br^{\eps}) \in  \mathbb X^{\eps} \times \bcQ^{\eps} $ be solution of \eqref{compactVar.eps} and $((\bsig,p),\br) \in  \mathbb X \times \bcQ $ be solution of the unperturbed problem (i.e., \eqref{compactVar.eps} for $\eps=0$). Then
%\begin{align}\label{eq.changeofvariable}
% \norma{\big(\mathbf{I}-(1+\mu^2) \mathcal{S}\big)\big( (\bsig^{\eps}\circ \bT^{\eps}-\bsig, p^{\eps}\circ \bT^{\eps}-p), \br^{\eps}\circ \bT^{\eps}-\br \big)} \leq \frac{C \eps}{ 1+\mu^2}  \|\bF\|_{0,\OS},
%\end{align}
%where $(\mathbb{X} \times \bcQ)^{*}$ denotes the dual space of $\mathbb{X} \times \bcQ.$
%\end{lem}
%

%\begin{proof}
%Since $((\bsig^{\eps},p^{\eps}),\br^{\eps}) \in  \mathbb X^{\eps} \times \bcQ^{\eps} $ is
%solution of \eqref{compactVar.eps}, 
%$((\bsig^{\eps},p^{\eps}),\br^{\eps}) \in  \mathbb X^{\eps} \times \bcQ^{\eps} $ 
%it satisfies \eqref{compactVar.eps}
%for all $((\btau^{\eps},q^{\eps}),\bs^{\eps}) \in\mathbb{X}^{\eps}\times \bcQ^{\eps}.$
\section{Shape calculus}\label{shape.sec}
In the present section we derive the shape derivative and shape Hessian for the solution
$((\bsig^\eps,p^{\eps}),\br^{\eps})$ of \eqref{equ:perturbed prob}.
\subsection{Material derivative}\label{sec.mat.der}
This section is devoted to a rigorous proof and characterization of the material derivative of
\eqref{compactVar.eps}. 
%The following proposition is easy to prove.
\begin{prop} \label{thm.conv}
Let $((\bsig^{\eps},p^{\eps}),\br^{\eps}) \in  \mathbb X^{\eps} \times \bcQ^{\eps} $ be solution of \eqref{compactVar.eps} and $((\bsig,p),\br) \in  \mathbb X \times \bcQ $ be solution of the unperturbed problem (i.e., \eqref{compactVar.eps} for $\eps=0$). Assume $\bF \in (\mathrm{L}^2(B_R))^d \cap (\mathbb{X} \times \bcQ)^{*}$ and $\kappa \in C^1(\partial \OS)$.
Then
\begin{align}\label{eqn:con}
\lim_{\eps \rightarrow 0} \norma{ \big((\bsig^{\eps}\circ \bT^{\eps},p^{\eps}\circ \bT^{\eps}),\br^{\eps}\circ \bT^{\eps}\big) - \big((\bsig,p),\br\big) }=0.
\end{align}
\end{prop}
\begin{proof}
To prove \eqref{eqn:con}, we first aim to prove that 
\begin{align}\label{eqn:conv A}
\bbA
\Big(
\big(\mathbf{I}-(1+\mu^2)\mathcal{S}\big)
\big(
(\bsig^{\eps}\circ \bT^{\eps}-\bsig, p^{\eps}\circ \bT^{\eps}-p), 
\br^{\eps}\circ \bT^{\eps}-\br 
\big),((\btau,q),\bs)\Big) \rightarrow 0
\quad \mbox{as} \quad \eps \rightarrow 0.
\end{align}
As a next step, we deduce from the well-posedness of the unperturbed problem (i.e., \eqref{compactVar.eps} for~$\eps=0$) the existence of~$C>0$, independent of $\eps$ such that
\begin{align*}%\label{at.coer}
&C\norma{\big(\mathbf{I}-(1+\mu^2)\mathcal{S}\big)
\big(
(\bsig^{\eps}\circ \bT^{\eps}-\bsig, p^{\eps}\circ \bT^{\eps}-p), 
\br^{\eps}\circ \bT^{\eps}-\br 
\big)} \notag\\
&\leq \sup_{\0 \neq ((\btau,q),\bs) \in \bcX \times \bcQ 
}\dfrac{ \bbA \Big(\big(\mathbf{I}-(1+\mu^2)\mathcal{S}\big)
\big(
(\bsig^{\eps}\circ \bT^{\eps}-\bsig, p^{\eps}\circ \bT^{\eps}-p), 
\br^{\eps}\circ \bT^{\eps}-\br 
\big),((\btau,q),\bs)
\Big) }{\norma{((\btau,q),\bs)}}.
\end{align*}
Hence, there exists $((\btau,q),\bs) \in  \bcX \times \bcQ $ such that
\begin{align}\label{arg:inf-sup1}
&\norma{\big(\mathbf{I}-(1+\mu^2)\mathcal{S}\big)
\big(
(\bsig^{\eps}\circ \bT^{\eps}-\bsig, p^{\eps}\circ \bT^{\eps}-p), 
\br^{\eps}\circ \bT^{\eps}-\br 
\big)}-\gamma \notag\\
&\leq \dfrac{ \bbA \Big(\big(\mathbf{I}-(1+\mu^2)\mathcal{S}\big)
\big(
(\bsig^{\eps}\circ \bT^{\eps}-\bsig, p^{\eps}\circ \bT^{\eps}-p), 
\br^{\eps}\circ \bT^{\eps}-\br 
\big),((\btau,q),\bs) 
\Big) }{C\norma{((\btau,q),\bs)}},
\end{align}
where $\gamma>0$ is arbitrary.
On letting $\eps \rightarrow 0$, using \eqref{eqn:conv A} we have
\begin{align}\label{arg:inf-sup2}
&\limsup_{\eps \rightarrow 0} \norma{\big(\mathbf{I}-(1+\mu^2)\mathcal{S}\big)
\big(
(\bsig^{\eps}\circ \bT^{\eps}-\bsig, p^{\eps}\circ \bT^{\eps}-p), 
\br^{\eps}\circ \bT^{\eps}-\br 
\big)}-\gamma \leq 0.
\end{align}
Since $\gamma>0$ is arbitrary,
\begin{align}\label{arg:inf-sup3}
&\limsup_{\eps \rightarrow 0} \norma{\big(\mathbf{I}-(1+\mu^2)\mathcal{S}\big)
\big(
(\bsig^{\eps}\circ \bT^{\eps}-\bsig, p^{\eps}\circ \bT^{\eps}-p), 
\br^{\eps}\circ \bT^{\eps}-\br 
\big)}= 0.
\end{align}
We then use equation \eqref{resolvent.eps} in Proposition \ref{specT1} to obtain \eqref{eqn:con}.

We are now left to prove \eqref{eqn:conv A}. To begin with, we first estimate
\begin{align}\label{eq.sub.D}
\mathbb{D}^{\eps}\Big(\big((\bsig^{\eps}, p^{\eps}), \br^{\eps}\big), \big((\btau^{\eps}, q^{\eps}),\bs^{\eps}\big)\Big)-\mathbb{D}\Big(\big((\bsig, p), \br\big), \big((\btau, q), s\big)\Big)=\ell^{\eps}(\btau^{\eps})-\ell(\btau),
\end{align}
Using the fact that $\bF \cdot \bdiv \btau(\vecy)=\big({\bf I} \otimes \bF(\vecy)\big):(\nabla \otimes \btau(\vecy)),$ and exploiting Lemmas \ref{lem.conv.sig}--\ref{lem.conv.tensor},
we have for all $\btau \in \bcW,$
 \begin{align*}
& \ell^{\eps}(\btau^{\eps})-\ell(\btau)=
 \int_{\OS}\frac{1}{\rS}
 \Big[\Big(\gamma(\eps,\vecy)\big( J_{\bT^\eps}^{-1} \otimes \bF(\bT^{\eps}\vecy)\big) : \big(\nabla \otimes  \btau(\vecy)\big) \Big)
 -\bF \cdot \bdiv \btau(\vecy) \Big]
 d \vecy\notag\\
 &=\int_{\OS}\frac{1}{\rS}
 \Big[\Big(\gamma(\eps,\vecy)\big(J_{\bT^\eps}^{-1} \otimes \bF(\bT^{\eps}\vecy)\big): \big(\nabla \otimes \btau(\vecy)\big) \Big)
 -\Big(\big({\bf I} \otimes \bF(\vecy)\big): \big(\nabla \otimes  \btau(\vecy)\big) \Big)\Big]
 d \vecy \notag\\
&\quad \rightarrow 0 \quad \mbox{as} \quad \eps \rightarrow 0. 
  \end{align*}
Passing to the limit as $\eps \rightarrow 0$ in \eqref{eq.sub.D}, we arrive at \eqref{eqn:conv A}. This completes the proof.
\end{proof}

%\newline

%Before proceeding to the following proposition, let us recall that 
\begin{lem}\label{lem:equ mat der}
There exists a unique solution~$((\hat{\bsig},\hat{p}),\hat{\br})\in\mathbb{X}\times\bcQ$ to the
following equation for all $\big((\btau,q),\bs\big) \in \mathbb{X} \times \bcQ$
\begin{align}\label{eq.mat.der}
&\mathbb{D}\Big(((\hat{\bsig},\hat{p}),\hat{\br}),((\btau,q),s)\Big)=- \int_{\OS}\frac{1}{\rS}
\Big[\tilde{\mathcal{A}}'(0,\vecy)\bsig(\vecy):\big(\mathbf{I} \otimes (\mathbf{I} \cdot \nabla )^{\top}\btau^{\top}(\vecy)\big)^{\top}\Big]d \vecy\notag\\&
\quad+ \int_{\OS}\frac{1}{\rS}
\Big( \bdiv \bV(\vecy) \big(\mathbf{I} \otimes \bF(\vecy)\big)+\tilde{\bV}_1(\vecy)\otimes \bF(\vecy)+\mathbf{I} \otimes \big(\nabla \bF(\vecy)\cdot \bV\big)
 \Big)
 :\big(\nabla \otimes \btau^{\top}(\vecy)\big) d \vecy \notag\\&
-\int_{\OS} \frac{1}{\rF}\Big[ \mathcal{A}'(0,\vecy) \nabla p (\vecy) \cdot \nabla q(\vecy)\Big] d \vecy 
+ \mu^2 \Big[ \int_{\OS} \cC^{-1}\gamma_1(\vecy) \bsig(\vecy):\btau(\vecy) d \vecy\notag\\&
\quad+ \int_{\OF} \dfrac{1}{\rF c^2} \gamma_1(\vecy) p(\vecy) q(\vecy) d \vecy  
+\int_{\OS} \btau(\vecy):\gamma_1(\vecy) \br(\vecy) d \vecy+\int_{\OS}\gamma_1(\vecy) \bsig(\vecy):\bs(\vecy)d \vecy
\Big],
\end{align} 
where $\mathcal{A}',\tilde{\mathcal{A}}',\gamma_1$ and $\tilde{V}_1$ are given in 
Appendix~\ref{sec.app.shape}.
\end{lem}
\begin{proof}
Existence and uniqueness of the solution of \eqref{eq.mat.der} in the space $ \mathbb{X} \times \bcQ$ relies on 
the well-known Babu\v{s}ka-Brezzi theory (see \cite[Theorem 3.1]{MMT}).
\end{proof}

 \begin{theorem} \label{thm.mat.der}
 Assume that $\bF \in \H^1(B_R) \cap
{(\mathbb{X} \times \bcQ)}^{*}$ and $\kappa \in C^1(\partial \OS)$. Let $((\bsig^{\eps},p^{\eps}),\br^{\eps}) \in  \mathbb X^{\eps} \times \bcQ^{\eps} $ and
$((\bsig,p),\br) \in  \mathbb X \times \bcQ $ be solutions of\eqref{compactVar.eps} and the unperturbed
problem (i.e., \eqref{compactVar.eps} for $\eps=0$), respectively. Then
$((\bsig^{\eps},p^{\eps}),\br^{\eps})$ has a material derivative $((\dot{\bsig},\dot{p}),\dot{\br})$  in $\mathbb{X} \times \bcQ$
which satisfies~\eqref{eq.mat.der}.
%which is solution to the following equation with unknown $((\tilde{\bsig},\big(p^{\eps} \circ \bT^{\eps}\big)),\tilde{\br})$
 \end{theorem}
 
 \begin{proof}
It suffices to show that
\begin{align}\label{lim.mat}
\lim_{\eps\to\0}
\norma{
\Big(
\big(
\dfrac{\bsig^{\eps} \circ \bT^{\eps}-\bsig}{\eps},
\dfrac{p^{\eps} \circ \bT^{\eps}-p}{\eps}
\big),
\dfrac{\br^{\eps} \circ \bT^{\eps}-\br}{\eps}
\Big)
-
\big((\hat{\bsig},\hat{p}),\hat{\br} \big) 
}
= 0,
\end{align}
where $((\hat{\bsig},\hat{p}),\hat{\br})$ is the solution of \eqref{eq.mat.der}. 
Using \eqref{def.DD.eps} and \eqref{defT.eps} we have
\begin{align}\label{eq.mat}
&\mathbb{D}\Big(\Big(\Big(\dfrac{\bsig^{\eps} \circ \bT^{\eps}-\bsig}{\eps}- \hat{\bsig},\dfrac{ p^{\eps}\circ \bT^{\eps}- p}{\eps}-\hat{p}\Big),\dfrac{ \br^{\eps}\circ \bT^{\eps}- \br}{\eps}-\hat{\br}\Big),((\btau,q),s)\Big)\notag\\&=
\bbA\Big(\big(\mathbf{I}-(1+\mu^2)\mathcal{S}\big)\big(\big(\dfrac{\bsig^{\eps}\circ \bT^{\eps}-\bsig}{\eps}-\hat{\bsig}, \dfrac{p^{\eps}\circ \bT^{\eps}-p}{\eps}-\hat{p}\big), \dfrac{\br^{\eps}\circ \bT^{\eps}-\br}{\eps}-\hat{\br} \big), \big((\btau,q),\bs\big)\Big)\notag\\&=
\int_{\OS} \dfrac{1}{\rS}\Big[ \bdiv\Big (\dfrac{\bsig^{\eps} \circ \bT^{\eps}-\bsig}{\eps}- \hat{\bsig} \Big)\Big] \cdot \bdiv \btau d \vecy + \int_{\OF}\frac{1}{\rF}\nabla \Big(\dfrac{ p^{\eps}\circ \bT^{\eps}- p}{\eps}-\hat{p}\Big)\cdot \nabla q
 d\vecy \notag\\&
 -\mu^2 \int_{\OS}\cC^{-1}
\Big(\dfrac{\bsig^{\eps} \circ \bT^{\eps}-\bsig}{\eps}-\hat{\bsig}\Big)
:\btau d \vecy 
-\mu^2\int_{\OF}\frac{1}{\rF c^2}
\Big(\dfrac{p^{\eps}\circ \bT^{\eps}-p}{\eps}-\hat{p}\Big)qd\vecy \notag\\
& -\mu^2 \int_{\OS}\btau:\Big(\dfrac{\br^{\eps}\circ \bT^{\eps}-\br}{\eps}-\hat{\br} \Big)d\vecy \notag\\& 
= \int_{\OS}\frac{1}{\rS}
 \Big[\dfrac{\gamma(\eps,\vecy)J_{\bT^\eps}^{-1} \otimes \bF(\bT^{\eps}\vecy)
 -\mathbf{I} \otimes \bF}{\eps} -\Big( \bdiv \bV(\vecy) \big(\mathbf{I} \otimes \bF(\vecy)\big)+\tilde{\bV}_1(\vecy)\otimes \bF(\vecy)
\notag\\&\quad \quad 
 +\mathbf{I} \otimes \big(\nabla \bF(\vecy)\cdot \bV\big)
 \Big) \Big]: \nabla \otimes  \btau \, d \vecy 
 - \int_{\OS}\frac{1}{\rS}
 \Big[\Big(\dfrac{\tilde{\mathcal{A}}(\eps,\vecy)(\bsig^{\eps} \circ \bT^{\eps})-\tilde{J}(\bsig^{\eps} \circ \bT^{\eps})}{\eps}
 -\tilde{\mathcal{A}}'(\0,\vecy) \bsig \Big) \notag\\ &\quad  \quad \quad :\mathbf{I} \otimes 
\cL_{\bf I} \btau\Big] d \vecy 
 - \int_{\OF}\frac{1}{\rF}\dfrac{\Big(\mathcal{A}(\eps,\vecy)-\mathbf{I}\Big)\nabla (p^{\eps}\circ \bT^{\eps})}{\eps} -\mathcal{A}'(\0,\vecy) \nabla p
  \cdot \nabla q d \vecy \notag\\&
  +\mu^2\Big\{ \int_{\OS}\cC^{-1}
\Big(\dfrac{(\gamma(\eps,\vecy)-1)\bsig^{\eps} \circ \bT^{\eps}}{\eps} -\gamma_1(\vecy) \bsig\Big)
:\tilde{\btau}d \vecy \notag\\
&+ \int_{\OF}\frac{1}{\rF c^2} \Big[\dfrac{(\gamma(\eps,\vecy)-1)
p^{\eps}\circ \bT^{\eps}}{\eps}-\gamma_1(\vecy)p\Big]qd\vecy + \int_{\OS}\btau:\Big(\dfrac{(\gamma(\eps,\vecy)-1) \br^{\eps}\circ \bT^{\eps}}{\eps}-\gamma_1(\vecy)\br \Big)d\vecy \notag\\
&+ \int_{\OS}\Big(\dfrac{(\gamma(\eps,\vecy)-1)\bsig^{\eps} \circ \bT^{\eps}(\vecy)}{\eps}-\gamma_1(\vecy)\bsig\Big):\bs d\vecy\Big\}.
\end{align}
 Using equation \eqref{equ:vbn 4}, Proposition \ref{thm.conv}, Lemma \ref{lem.conv.sig}--Lemma \ref{lem.conv.tensor}, and letting $\eps \rightarrow 0$, we see that the right hand side of \eqref{eq.mat} $\rightarrow 0.$ 
Repeating the arguments as in \eqref{arg:inf-sup1}-\eqref{arg:inf-sup3} for~$\big(
\big(
(\bsig^{\eps} \circ \bT^{\eps}-\bsig)/\eps,
(p^{\eps} \circ \bT^{\eps}-p)/\eps
\big),
\big(\br^{\eps} \circ \bT^{\eps}-\br)/{\eps}
\big)$ and exploiting equation \eqref{resolvent.eps}, we prove \eqref{lim.mat}.
This completes the proof.
\dela{ \begin{align*}
 &\lim_{\eps \rightarrow 0}
\bbA\Big(\big(I-(1+\mu^2)\mathcal{S}\big)\big(\big(\dfrac{\bsig^{\eps}\circ \bT^{\eps}-\bsig}{\eps}-\dot{\bsig}, \dfrac{p^{\eps}\circ \bT^{\eps}-p}{\eps}-\dot{p}\big), \dfrac{\br^{\eps}\circ \bT^{\eps}-\br}{\eps}-\dot{\br} \big), \big((\btau,q),\bs\big)\Big)=0.
 \end{align*}}
 \end{proof}
\subsection{Shape derivative}\label{sec.shape.der}
This section is devoted to the existence and characterization of shape derivative and shape Hessian of the deterministic solution of the considered model problem. We denote 
\begin{align*}
\mathcal{J}_1(v^{\eps},\O^{\eps},w):= \int_{\O^{\eps}} v^{\eps} \star w \quad \mbox{and} \quad
\mathcal{J}_2(v^{\eps},\partial \O^{\eps},w):= \int_{\partial \O^{\eps}} v^{\eps} \star w   
\end{align*}
where $\star$ stands for the usual product of two scalar functions, or the dot product of two vector functions, or the component-wise inner product of two tensor functions.
\noindent
Let us define the following spaces:
\begin{align*}
\bar{\mathcal{D}}_1&:= \Big\{\btau \in   C_c^{\infty}(\R^d;\R^{d \times d}):\big \langle D \btau^{(k)}\bn,\bn \big \rangle=0  \,\,\mbox{on} \, \partial \OS,\,\, \forall \, k=1,\ldots, d   \Big\},\notag\\
\bar{\mathcal{D}}_2&:=\Big\{q \in  C_c^{\infty}(\R^d;\R): \dfrac{\partial q}{\partial \bn}=0 \,\, \mbox{on}\,\, \partial \OF  \Big\}.
\end{align*}

\begin{theorem}\label{thm.sh.der}
 Under the assumptions of Theorem \ref{thm.mat.der}, $((\bsig^{\eps},p^{\eps}),\br^{\eps})$ has a shape derivative $\big((\bsig', p'), \br'\big)$ belonging in $(\H(\mathbf{div};\OS) \times \mathrm{H}^1(\OF)) \times \bcQ$ that satisfies
 %with unknown $((\tilde{\bsig},\big(p^{\eps} \circ \bT^{\eps}\big)),\tilde{\br})$ and 
%for all $((\btau,q),\bs) \in \Big((\bar{\mathcal{D}_1} \times \bar{\mathcal{D}_2})\cap \mathbb{X} \Big)  \times \bcQ,$ 
\begin{subequations}\label{eq.sh.der.D1}
\begin{align}
  \bdiv \bsig'+\mu^2\rS \bu' &=\0 \qquad \qquad \qquad \qquad \qquad \qquad \qquad \qquad   \quad \, \qin \OS,
\label{eq.dom.S} \\
 \bsig' &= \cC \beps(\bu') \qquad \qquad \qquad \qquad \qquad \qquad  \qquad \quad \qin\OS,
\label{eq.sig.dom} \\
\bu'&=-\kappa (\nabla \bu) \bn
+\dfrac{1}{\rS \mu^2} \bdiv_{\partial \OS}(\bn) \bG
  \qquad  \quad \,\,\,\, \quad  \qon  \GD, \label{eq.bc.u}\\
\bsig' \bn&= (\nabla^{\top} \otimes \bsig) (\bV \otimes \bn)\quad \qquad \qquad \qquad   \qquad  \quad \, \qon \GN, \label{eq.bc.sign}\\
\Delta p' +\frac{\mu^2}{c^2}p'&=0 \qquad 
\quad \qquad \qquad \qquad   \quad \quad \qquad  \qquad \quad \,\,\,
\quad \qin\OF,\label{eq.dom.F}\\
\bsig' \bn+p' \bn &=-(\nabla^{\top} \otimes \bsig) (\bV \otimes \bn) - (\bV^{\top} \nabla p) \bn \qquad \qquad   \qon \S,\label{eq.bc.sig.prime}\\
 \dfrac{1}{\rF} 
\dfrac{\partial p'}{\partial \bn}
-\mu^2 \bu' \cdot \bn
&
=\dfrac{1}{\rS} (\bdiv_{\partial \OF}\bn) \bG \cdot \bn  - \mu^2 \kappa (\nabla \bu \bn)\cdot \bn \notag\\
& \quad+\dfrac{1}{\rF} \Big[ \divv \big(\kappa \nabla p   \big)
-\dfrac{\partial}{\partial \bn} \big(\kappa \nabla p  \big) \cdot \bn-div_{\partial \OF} (\bn) \kappa \dfrac{\partial p}{\partial \bn} +\dfrac{\mu^2}{c^2} p  \kappa\Big]
\notag\\
&\qquad \qquad  \qquad \qquad \qquad \qquad \qquad \qquad \qquad \qquad \quad \,\,
 \hbox{on } \S,\label{eq.sd.bc2}
\end{align}
\end{subequations}
where \begin{align}\label{def.G}
\bG:= \bdiv \big(\kappa \bsig \big)
- \dfrac{\partial \kappa}{\partial \bn}  \bsig \bn+ \kappa \bsig \big(\dfrac{\partial}{\partial \bn} (\bn)\big) -\bsig \big(\nabla_{\partial \OS} \kappa\big).
 \end{align}
\end{theorem}

\begin{proof}
We split the proof of \eqref{eq.sh.der.D1} in three steps. In the first step, we will derive equations satisfied by the shape derivative $\big((\bsig', p'), \br'\big)$ on the domains $\OS$ and $\OF.$ In the next two steps, we will derive the boundary conditions satisfied by $\big((\bsig', p'), \br'\big)$. 
We use Lemma \ref{pro:mat sha pro} in Appendix \ref{sec.app} to prove \eqref{eq.sh.der.D1}.

\underline{{\bf Step I:} }  \quad
In this step, our goal is to prove \eqref{eq.dom.S}, \eqref{eq.sig.dom} and \eqref{eq.dom.F}.
Existence of the shape derivative $\big((\bsig', p'), \br'\big)$ follows from Theorem \ref{thm.mat.der} and Definition \ref{def.sd}. We choose $((\btau,q),\bs) \in \Big((\bar{\mathcal{D}_1} \times \bar{\mathcal{D}_2})  \cap\mathbb{X}\Big) \times \bcQ ,$ in \eqref{compactVar.eps}.\dela{and have
 that
\begin{equation}
 \mathbb{D}^{\eps}\Big(\big((\bsig^{\eps}, p^{\eps}), \br^{\eps}\big), \big((\btau, q), \bs\big)\Big)
=  \ell^{\eps}(\btau^{\eps}). \label{compactVar.eps.shape}
\end{equation} }
In order to apply Lemma \ref{pro:mat sha pro}, we consider each of the bilinear forms present in the definition of $\mathbb{D}^{\eps}.$
%left hand side of \eqref{compactVar.eps}.
%This essentially implies that
Using integration by parts (see Lemma~\ref{lem.int.by.parts} in Appendix \ref{sec.tensor}), we have
\begin{align*}
a_1^{\eps}(\bsig^{\eps},\btau)\dela{:=\int_{\OS^{\eps}}\frac{1}{\rS}\bdiv\bsig^{\eps}(\vecx) \cdot\bdiv\btau(\vecx)d\vecx \notag\\}
&=\underbrace{-\int_{\OS^{\eps}}\frac{1}{\rS}\bsig^{\eps} :\nabla
(\bdiv\btau)}_{\mathcal{J}_{1}\big(\bsig^{\eps},\OS^{\eps}, \frac{\nabla (\bdiv
\btau)}{\rS}\big):=}+ \underbrace{\int_{\partial \OS^{\eps}}\frac{1}{\rS}\bsig^{\eps}
\bn^{\eps} \cdot \bdiv\btau d \mathcal{S}}_{\mathcal{J}_{2}(\bsig^{\eps} \bn^{\eps},\partial
\OS^{\eps},\frac{\bdiv \btau}{\rS}):=}.
\end{align*}
Using part (iv) of Lemma \ref{pro:mat sha pro}, we have
\begin{align}
d \mathcal{J}_{1}(\bsig^{\eps},\OS^{\eps}, \frac{\nabla (\bdiv \btau)}{\rS})\Big|_{\eps=0}&=-\int_{\OS}\frac{1}{\rS}\bsig' : \nabla(\bdiv\btau)-\int_{\partial \OS}\frac{1}{\rS}\bsig:\nabla (\bdiv \btau) \kappa\, d \mathcal{S},\label{bar.I11}\\
d \mathcal{J}_{2}(\bsig^{\eps} \bn_S^{\eps},\partial \OS^{\eps},\frac{\bdiv \btau}{\rS})\Big|_{\eps=0}&=\int_{\partial \OS}\frac{1}{\rS}\Big[\Big(\bsig' \bn \cdot  \bdiv\btau+\bsig \bn'_{\mas}\cdot \bdiv \btau \Big) +\kappa\dfrac{\partial}{\partial \bn}\Big(\bsig \bn \cdot \bdiv\btau \Big)\notag\\
& \qquad +\kappa\Big(\bdiv( \bn)\bsig \bn^{\top}\cdot \bdiv\btau \Big)\Big]  d \mathcal{S}.\label{bar.I12}
\end{align}
From \eqref{bar.I11} and \eqref{bar.I12} and using tangential Green's formula 
(see Lemma~\ref{lem.int.by.parts}), we have
\begin{align}\label{add.barI1}
&d \mathcal{J}_{1}(\bsig^{\eps},\OS^{\eps}, \frac{\nabla (\bdiv \btau)}{\rS})\Big|_{\eps=0}+d \mathcal{J}_{2}(\bsig^{\eps} \bn_S^{\eps},\partial \OS^{\eps},\frac{\bdiv \btau}{\rS})\Big|_{\eps=0}\notag\\
&=\int_{\OS}\frac{1}{\rS}\bdiv\bsig' \cdot\bdiv\btau
+\int_{\partial \OS}\frac{1}{\rS} \bdiv\big(\kappa  \bsig \big) \cdot \bdiv \btau
-\int_{\partial \OS}\frac{1}{\rS}
\dfrac{\partial \kappa}{\partial \bn} \bsig \bn \cdot \bdiv \btau d \mathcal{S}
\notag\\& \quad
-\int_{\partial \OS}\frac{1}{\rS}
\bsig \big(\nabla_{\partial \OS}  \kappa \big) \cdot \bdiv \btau  d\mathcal{S}
 +\int_{\partial \OS}\frac{1}{\rS} \bsig \big(\dfrac{\partial}{\partial \bn}( \bn)\big) \cdot 
\bdiv \btau \, \kappa  \, d\mathcal{S}.
\end{align}
As in $a_1^{\eps}(\bsig^{\eps},\btau),$ using integration by parts 
(see Lemma~\ref{lem.int.by.parts} in Appendix \ref{sec.tensor}), we have
\begin{align*}
a_2^{\eps}(p^{\eps},q)\dela{&=\int_{\OF^{\eps}}\frac{1}{\rF}\nabla p^{\eps}(\vecx)\cdot\nabla q(\vecx)d\vecx \notag\\}
&=\underbrace{-\int_{\OF^{\eps}}\frac{1}{\rF} p^{\eps}\Delta
q}_{\mathcal{J}_1(p^{\eps},\OF^{\eps},\frac{\Delta q}{\rF}):=} \underbrace{+\int_{\partial
\OF^{\eps}}\frac{1}{\rF} p^{\eps}\dfrac{\partial q}{\partial
\bn^{\eps}}d\mathcal{S}}_{\mathcal{J}_2(p^{\eps}\frac{\partial q}{\partial \bn^{\eps}},\partial
\OF^{\eps},\frac{1}{\rF}):=}.
\end{align*}
Exploiting parts (iv) and (v) of Lemma \ref{pro:mat sha pro}, we have
\begin{align}
d\mathcal{J}_1(p^{\eps},\OF^{\eps},\frac{\Delta q}{\rF}) \Big|_{\eps=0}&=-\int_{\OF}\frac{1}{\rF} p'\Delta q -\int_{\partial \OF}\frac{1}{\rF} p \Delta q \, \kappa d\mathcal{S},\label{bar.I21}\\
d\mathcal{J}_2(p^{\eps}\frac{\partial q}{\partial \bn^{\eps}},\partial \OF^{\eps},\frac{1}{\rF}) \Big|_{\eps=0}&=\int_{\partial \OF}\frac{1}{\rF}\Big(
p'\dfrac{\partial q}{\partial \bn} -p \Big(\nabla_{\partial \OF}q \cdot \nabla_{\partial \OF}  \kappa \Big) \Big) d\mathcal{S} \notag\\
&\quad+\int_{\partial \OF}\frac{1}{\rF}\Big(\dfrac{\partial}{\partial \bn}\Big(p \dfrac{\partial q}{\partial \bn}\Big)+\bdiv_{\partial \OF} (\bn)p \dfrac{\partial q}{\partial \bn} \Big) \kappa d\mathcal{S}.\label{bar.I22}
\end{align}
In the second term of $d\mathcal{J}_1(p^{\eps},\OF^{\eps},\frac{\Delta q}{\rF})$ we use the identity for $\Delta q$ on $\partial \OF,$ which is, $\Delta q=\Delta_{\partial \OF}q + \bdiv_{\partial \OF} (\bn) \dfrac{\partial q}{\partial \bn}+\dfrac{\partial ^2 q} {(\partial \bn)^2}.$
Hence on further using Green's and tangential Green's formula, \eqref{bar.I21} and \eqref{bar.I22} reduce to
\begin{align}\label{add.barI2}
&d\mathcal{J}_1(p^{\eps},\OF^{\eps},\frac{\Delta q}{\rF}) \Big|_{\eps=0}+d\mathcal{J}_2(p^{\eps}\frac{\partial q}{\partial \bn^{\eps}},\partial \OF^{\eps},\frac{1}{\rF}) \Big|_{\eps=0}=\int_{\OF}\frac{1}{\rF} \nabla p' \cdot \nabla q +\int_{\partial \OF}\frac{1}{\rF}
\nabla p \cdot \nabla q \,\kappa d\mathcal{S}.
\dela{
\dfrac{\partial q(\vecx)}{\partial \bn_F}d\mathcal{S}(\vecx)
 \Delta_{\partial \OF}q(\vecx)d \mathcal{S}(\vecx)\notag\\
&-\int_{\partial \OF}\frac{1}{\rF} p(\vecx) 
\Big(\nabla q(\vecx)\cdot \nabla_{\partial \OF} \langle V,\bn \rangle \Big)\Big) d\mathcal{S}(\vecx)+\int_{\partial \OF}\frac{1}{\rF}  \dfrac{\partial p}{\partial \bn}\cdot\dfrac{\partial q}{\partial \bn}\langle V,\bn \rangle d\mathcal{S}(\vecx)}
\end{align}
Results similar to \eqref{add.barI1} and \eqref{add.barI2} can be obtained for the other bilinear forms in the definition of 
$\mathbb{D}^{\eps}$. Hence we obtain
%It can be observed that while passing the limit as $\eps \rightarrow 0$ to all the remaining bilinear forms in the definition of $\mathbb{D}^{\eps}$ are straightforward without using integration by parts. Hence using \eqref{add.barI1} and \eqref{add.barI2}, we have
\begin{align}\label{add.bc}
&\int_{\OS}\frac{1}{\rS}\bdiv \bsig' \cdot\bdiv\btau +\int_{\partial \OS}\frac{1}{\rS} \bG \cdot \bdiv \btau d \mathcal{S}
+\int_{\OF}\frac{1}{\rF} \nabla p'\cdot \nabla q 
-\mu^2 \int_{\OS}\cC^{-1}\bsig':\btau
\notag\\&
-\mu^2 \int_{\OF}\frac{1}{\rF c^2}p'q
-\mu^2 \int_{\OS}\btau:\br'+\int_{\partial \OF}\frac{1}{\rF}
\nabla p \cdot \nabla q \, \kappa d\mathcal{S}
-\mu^2\int_{\partial \OS} \cC^{-1}\bsig:\btau\, \kappa d\mathcal{S}
\notag\\&
-\mu^2 \int_{\partial \OF}\frac{1}{\rF c^2}pq \,\kappa d\mathcal{S}
-\mu^2 \int_{\partial \OS}\btau:\br \, \kappa d\mathcal{S}
=0
\end{align}
with %\eqref{bar.I7}.
\begin{align}\label{bar.I7.new}
\int_{\OS}\bsig':\bs +\int_{\partial \OS}\bsig:\bs\,\kappa d\mathcal{S}=0.
\end{align}
We now choose $((\btau,q),\bs) \in \Big(C_c^{\infty}(\OS;\R^{d \times d}) \times C_c^{\infty}(\OF;\R) \times C_c^{\infty}(\OS;\R^{d \times d})\Big) \cap \Big(\mathbb{X} \times \bcQ\Big),$ which gives us
\begin{align}\label{sh.new}
&\int_{\OS}\frac{1}{\rS}\bdiv\bsig'\cdot \bdiv\btau +\int_{\OF}\frac{1}{\rF} \nabla p'\nabla q -\mu^2\Big[\int_{\OS}\cC^{-1}\bsig':\btau +\int_{\OF}\frac{1}{\rF c^2}p'q+\int_{\OS}\btau:\br'
\dela{+\int_{\OS}\bsig'(\vecx):\bs(\vecx)d\vecx}\notag\\&
+\int_{ \OS}\bsig':\bs \Big]
=0.
\end{align}
Using density argument we obtain \eqref{eq.dom.S} and \eqref{eq.dom.F}.
Next, using \eqref{dual-problem_1b.ep} and integration by parts we achieve
\begin{align} \label{eq.dom1}
\int_{\OS^\eps} \big( \cC^{-1} \bsig^{\eps} -\br^{\eps}\big):\btau =-
\int_{\OS^\eps} \bu^{\eps} \cdot \bdiv \btau +\int_{\partial \OS^\eps} \bu^{\eps}\cdot \btau \bn^{\eps} d\mathcal{S}.
\end{align}
Again exploiting Lemma \ref{pro:mat sha pro} to \eqref{eq.dom1} and using Tangential Green's formula, we have \eqref{eq.sig.dom}.
\dela{\begin{align}
\int_{\OS} \big( \cC^{-1} \bsig' -\br'\big):\btau 
+ \int_{\OS} \nabla \bu:\btau \kappa d \mathcal{S}
&=-
\int_{\OS} \bu' \cdot \bdiv \btau -\int_{\partial \OS} \bu \cdot \bdiv \btau \kappa d\mathcal{S}
\notag\\
& \quad +\int_{\partial \OS} (\bu' \cdot \btau \bn + \bu \cdot \btau \bn')d \mathcal{S} \notag\\&
\quad
+\int_{\partial \OS} \Big[ \dfrac{\partial}{\partial \bn} (\bu \cdot \btau \bn)+\bdiv_{\partial \OS} (\bn)\bu \cdot \btau \bn \Big] \kappa d\mathcal{S}
.
\end{align}}
\medskip
 %\newline
\underline{{\bf Step II:}}   \quad
In this step, we aim to prove \eqref{eq.sd.bc2}. In order to find the required boundary condition, we choose $((\btau,q),\bs) \in \Big((\bar{\mathcal{D}_1} \times \bar{\mathcal{D}_2})\cap \mathbb{X} \Big)  \times \bcQ,$ and use \eqref{sh.new} to obtain from \eqref{add.bc} and \eqref{bar.I7.new},
\begin{align*}%\label{eq.sd.11}
&\int_{\partial \OS} \dfrac{1}{\rS} \bdiv \bsig' \cdot \bdiv \btau  d \mathcal{S}+\int_{\partial \OS}\frac{1}{\rS} \bG \cdot \bdiv \btau  d \mathcal{S}
-\int_{\OF} \dfrac{1}{\rF} \Delta p' \cdot q
-\mu^2 \int_{\OF} \dfrac{1}{\rF c^2} p'q
\notag\\&
+ \int_{\partial \OF}\frac{1}{\rF}
\nabla p \cdot \nabla q \,\kappa d\mathcal{S}
+\int_{\partial \OF}\dfrac{1}{\rF}q \dfrac{\partial p'}{\partial \bn}d \mathcal{S}
+\mu^2 \int_{\partial \OS} \bu' \cdot\bdiv \btau - \mu^2 \int_{\partial \OS}
\btau \bn \cdot \bu' d \mathcal{S}\notag\\&
-\mu^2 \int_{\partial \OS} \nabla \bu :\btau\,
\kappa d\mathcal{S}
-\mu^2 \int_{\partial \OF}\frac{1}{\rF c^2}pq \, \kappa d\mathcal{S}=0.
% \quad \quad \mbox{on} \quad %\bar{\mathcal{D}}_1 \times %\bar{\mathcal{D}}_2.
\end{align*} 
Using \eqref{eq.dom.S} and \eqref{eq.dom.F} we have
\begin{align}\label{eq.sd}
&\int_{\partial \OS}\frac{1}{\rS} \bG \cdot \bdiv \btau  d \mathcal{S} 
 - \mu^2 \int_{\partial \OS}
\btau \bn \cdot \bu' d \mathcal{S}
-\mu^2 \int_{\partial \OS} \nabla \bu :\btau
\kappa d\mathcal{S}
+\int_{\partial \OF}\dfrac{1}{\rF}q \dfrac{\partial p'}{\partial \bn}d \mathcal{S}
\notag\\&
+ \int_{\partial \OF}\frac{1}{\rF}
\nabla p \cdot \nabla q \,\kappa d\mathcal{S}
-\mu^2 \int_{\partial \OF}\frac{1}{\rF c^2}pq \, \kappa d\mathcal{S}=0.
% \quad \quad \mbox{on} \quad %\bar{\mathcal{D}}_1 \times %\bar{\mathcal{D}}_2.
\end{align} 
Using Lemma \ref{lem.int.by.parts} and using $\dfrac{\partial \btau}{\partial \bn}\bn=0$ and $\dfrac{\partial q}{\partial \bn}=0$ to the first and fifth terms of \eqref{eq.sd} respectively, we achieve
\begin{align}\label{int.first}
\int_{\partial \OS}\frac{1}{\rS} 
\bG \cdot \bdiv \btau  d \mathcal{S} =-\int_{\partial \OS}\frac{1}{\rS}
\Big[
 \nabla \bG : \btau 
+ \btau \bn\cdot \dfrac{\partial \bG}{\partial \bn} + 
\bdiv_{\partial \OS} (\bn) \btau \bn \cdot \bG \Big] d \mathcal{S},
\end{align}
\begin{align}\label{int.fifth}
\int_{\partial \OF}\frac{1}{\rF}
\nabla p \cdot \nabla q \,\kappa d\mathcal{S}
&=- \int_{\partial \OF}\frac{1}{\rF}
\Big[
q \bdiv (\nabla p \,\kappa)+ \Big(
\dfrac{\partial}{\partial \bn}(\nabla p \, \kappa)\cdot \bn 
\Big) q 
+ \bdiv_{\partial \OF} (\bn) \dfrac{\partial p}{\partial \bn}\kappa \,q \Big]d \mathcal{S}.
\end{align}
We now choose $\btau \in \bcW$ such that $\btau=0$ on $\GD \cup \GN.$ This yields $ q=-(\btau \bn) \cdot \bn$ on $\S$.
Substituting this value of $q,$ equations \eqref{int.first} and \eqref{int.fifth} in \eqref{eq.sd} on $\S=\partial \OF,$ we have \eqref{eq.sd.bc2}.
\newline
\medskip
%\newline
\underline{{\bf Step III:}}  \quad
In this step, we prove 
\eqref{eq.bc.u}, \eqref{eq.bc.sign} and \eqref{eq.bc.sig.prime}.
Using \eqref{eq.sd.bc2} and $\btau \bn=0$ on~$\GN,$ we have 
\begin{align*}
-\mu^2 \int_{\GD} \btau \bn \cdot \bu' d \mathcal{S}&=\int_{\GD \cup \GN}\dfrac{1}{\rS} \nabla \bG :\btau d \mathcal{S}
-\int_{\GD} \dfrac{1}{\rS}\btau \bn\cdot \dfrac{\partial \bG}{\partial \bn}d \mathcal{S}\notag\\& \quad
-\int_{\GD} \dfrac{1}{\rS} \bdiv_{\partial \OS} (\bn)\btau \bn \cdot \bG d \mathcal{S}
+\mu^2 \int_{\GD \cup \GN} \kappa \nabla \bu: \btau d \mathcal{S}.
\end{align*}
Since $\GD \cap \GN =\varnothing,$ choosing $\btau=0$ on $\GN,$ we have 
\begin{align*}
-\mu^2 \bn^{\top} \otimes \bu'=
\dfrac{1}{\rS} \Big[\nabla \bG -\bn^{\top} \otimes \dfrac{\partial \bG}{\partial \bn}-\bdiv_{\partial \OS}(\bn)\bn^{\top} \otimes \bG
\Big] +\mu^2 \kappa \nabla \bu, \quad \mbox{on} \quad \GD
\end{align*}
which proves \eqref{eq.bc.u}.
By Definition \ref{def.sd} we have
\[
\sigma'_{ij}=\dot{\sigma}_{ij} -\nabla \sigma_{ij}\cdot \bV.
\]
Using $\dot{\sigma} \bn=0$ on $\GN,$ 
we have
\[
(\bsig' \bn)_{i}=-\sum_{j=1}^d \nabla \sigma_{ij}\cdot \bV \bn^j,
\]
and this directly implies \eqref{eq.bc.sign}.
Again since $(\dot{\bsig},\dot{p}) \in \mathbb{X},$ $\dot{\bsig}\bn+\dot{p}\bn=0,$
we have
\[
(\bsig' \bn+p' \bn)_{i}=-\sum_{j=1}^d \Big(\nabla \sigma_{ij}\cdot \bV \bn^j+\nabla p \cdot \bV
\bn_j\Big),
\]
which directly implies \eqref{eq.bc.sig.prime}.
%This proves \eqref{eq.sd.bc2}.
Hence combining all the three steps, we conclude that the shape derivative $\big((\bsig', p'), \br'\big)$ satisfies \eqref{eq.sh.der.D1} for all $((\btau,q),\bs) \in 
\Big((\bar{\mathcal{D}_1} \times \bar{\mathcal{D}_2}) \cap \mathbb{X} \Big) \times \bcQ.$ 
This completes the proof.
%which is true for all $((\btau,q),\bs) \in \mathbb{X} \times \bcQ$.
%Hence we have \eqref{eq.sh.der.D2}.
\dela{\begin{align}
\mathcal{D}_2 \big(((\btau_{1},q_{1}),\br_{1}),((\btau,q),\bs)\big)&=
-\int_{\partial \OS} \dfrac{1}{\rS} \bdiv \bsig(\vecx) \cdot \bdiv \btau(\vecx) \langle \mathbf{V},\bn \rangle d \mathcal{S}(\vecx)\notag\\&\quad-\int_{\partial \OF}\frac{1}{\rF}
\nabla p(\vecx) \cdot \nabla q(\vecx)\langle \mathbf{V},\bn_F \rangle d\mathcal{S}(\vecx)\notag\\&\quad
+\mu^2\Big[\int_{\partial \OS} \cC^{-1}\bsig(\vecx):\btau(\vecx)\kappa d\mathcal{S}(\vecx)\notag\\&\quad
+\int_{\partial \OF}\frac{1}{\rF c^2}p(\vecx)q(\vecx) \kappa d\mathcal{S}(\vecx)\notag\\ &\quad
+\int_{\partial \OS}\btau(\vecx):\br(\vecx)\kappa d\mathcal{S}(\vecx)
+\int_{\partial \OS}\bsig(\vecx):\bs(\vecx)\kappa d\mathcal{S}(\vecx)\Big]
\end{align}}
\dela{
Using tangential Green's formula on second and fourth term of \eqref{eq.sd} and using density argument we get \eqref{sd.bc}, which is true for all $((\btau,q),\bs) \in \mathbb{X} \times \bcQ$.
Hence we conclude that the shape derivative $\big((\bsig', p'), \br'\big)$ satisfies \eqref{eq.sh.der} with boundary conditions \eqref{sd.bc}.
This completes the proof.}
\end{proof}
Before proceeding to the next theorem, let us consider the perturbation of the domain $\O:=\OS\cup\OF$ with respect to $\bT^{\delta}$ (where $\bT^{\delta}$  is given by \eqref{equ:T eps define} and \eqref{equ:T e cond}). We consider another boundary variation $\mathbf{V}_1$ which is of the form
\[
\bV_1(\vecx)
:=
\kappa_1(\vecx)\vecn{}(\vecx),\,\,i=1,2
\]
where $\kappa_1$ has the same regularity as in \eqref{equ:V def}.
We refer to Appendix \ref{sec.app} for further details about second order variations.
We will show that shape Hessian $\big((\bsig'', p''), \br''\big)$ exists. For this, we need to consider the shape derivative $\big((\bsig'_{\delta}, p'_{\delta}), \br'_{\delta}\big)$ exists in $\H(\mathbf{div};\OS^{\delta}) \times \mathrm{H}^1(\OF^{\delta}) \times \bcQ^{\delta}$ and satisfies
\eqref{eq.sh.der.D1}. Furthermore, following Theorem \ref{thm.mat.der}, one can prove existence of second order material derivative $((\ddot{\bsig},\ddot{p}),\ddot{\br})$  in $\mathbb{X} \times \bcQ.$ In the current paper, we omit the characterization of the second order material derivative.

\begin{theorem} \label{thm.sh.hes} 
 Under the above mathematical settings as in Theorem \ref{thm.sh.der}, shape Hessian $\big((\bsig'', p''), \br''\big)$ belonging in $\big((\H(\mathbf{div};\OS) \times \mathrm{H}^1(\OF)) \times \bcQ\big)$ 
that satisfies
\begin{subequations}\label{eq.sh.hes.D1}
\begin{align}
  \bdiv \bsig''+\mu^2\rS \bu'' &=\0 \qquad \qquad \qquad \qquad \qquad \qquad \qquad \qquad   \quad \, \qin \OS,
\label{eq.dom.S.h} \\
 \bsig'' &= \cC \beps(\bu'') \qquad \qquad \qquad \qquad  \qquad \qquad \qquad \,\,\, \qin\OS,
\label{eq.sig.dom.h} \\
\bu''&=\dfrac{1}{\mu^2}( \bdiv_{\partial \OS} \bn) \mathbb{H}_1+\dfrac{1}{\mu^2} \mathbb{H}_3 \bn \quad \qquad \qquad
  \quad \,\, \, \qon  \GD, \label{eq.bc.u.h}\\
\bsig'' \bn&= \ddot{\bsig}\bn - (\nabla^{\top} \otimes \bsig) (\dot{\bV} \otimes \bn)
-(\nabla^{\top} \otimes \bsig') (\bV_1 \otimes \bn)\notag\\
 &\quad -\mathcal{M}(\nabla^{\top} \otimes \bsig) (\bV \otimes \bn)\qquad \qquad  \qquad  \quad \quad  \,\qon \GN, \label{eq.bc.sign.h}\\
\Delta p'' +\frac{\mu^2}{c^2}p''&=0 \qquad 
\quad \qquad \qquad \qquad   \qquad \qquad \qquad  \quad \,\, \,\,
\quad \qin\OF,\label{eq.dom.F.h}\\
\bsig'' \bn+p'' \bn &=\ddot{\bsig}\bn - (\nabla^{\top} \otimes \bsig) (\dot{\bV} \otimes \bn)
-(\nabla^{\top} \otimes \bsig') (\bV_1 \otimes \bn)- \mathcal{M}(\nabla^{\top} \otimes \bsig) (\bV \otimes \bn)
\notag\\
& \quad +\ddot{p}\bn - \Big(\bV^{\top}\dot{(\nabla p)} +\bV_1^{\top}(\nabla p')+\dot{\bV}_1
^{\top} \nabla p \Big)\bn
   \,\,  \qon \S,\label{eq.bc.sig.prime.h}\\
 \dfrac{1}{\rF} 
\dfrac{\partial p''}{\partial \bn}
-\mu^2 \bu'' \cdot \bn
&
=-\dfrac{1}{\rS} (\bdiv_{\partial \OF}\bn) \mathbb{H}_1 \cdot \bn-\mathbb{H}_3 \bn \cdot \bn 
-\bdiv \mathbb{H}_2 +\bdiv_{\partial \OF}(\bn)\mathbb{H}_2 \cdot \bn 
\notag\\ & \quad +\mathbb{H}_4
\qquad 
\quad \qquad \qquad \qquad   \quad \quad \qquad \qquad 
\quad \, \qon \S,\label{eq.sd.bc2.h}
\end{align}
\end{subequations}
where 
\begin{subequations}\label{def.Hi}
\begin{align}
\dela{\tilde{G}&:= \bsig' \nabla_{\partial \OS} \kappa
- \dfrac{\partial}{\partial \bn} \big(\kappa\big) \bsig' \bn+\bdiv \big(\kappa \bsig' \big)-\bsig' \kappa 
\dfrac{\partial}{\partial \bn} (\bn),\notag\\}
\mathbb{H}_1&:=-\dfrac{1}{\rS}\Big[ \bdiv(\kappa \bsig')-\dfrac{\partial (\kappa)}{\partial \bn}\bsig' \bn- \bsig' (\nabla_{\partial \OS} \kappa)+ \bsig' \dfrac{\partial (\bn)}{\partial\bn}
+\bG' +\kappa_1 \dfrac{\partial \bG}{\partial \bn}+\kappa_1 \bdiv_{\partial \OS}(\bn) \bG
\Big],\\
\mathbb{H}_2&:=-\Big[ \dfrac{\kappa}{\rF}(\nabla p)'+\dfrac{\kappa_1 \kappa}{\rF} \dfrac{\partial (\nabla p)}{\partial \bn}+\kappa_1 \dfrac{\partial \kappa}{\partial \bn}\nabla p+\kappa_1 \kappa \bdiv_{\partial \OF}(\bn)\nabla p  \Big],\\
\mathbb{H}_3&:= -\mu^2\Big[ (\kappa_1 \nabla \bu'+\kappa \nabla \bu' +\kappa_1 \Big(\dfrac{\partial}{\partial \bn}(\kappa \nabla \bu)+\bdiv_{\partial \OS}(\bn)\kappa \nabla \bu \Big)
\Big],\\
\mathbb{H}_4&:= \mu^2 \Big[\dfrac{1}{\rF c^2}(\kappa_1p'+\kappa p')+\dfrac{\kappa_1}{\rF c^2}\dfrac{\partial}{\partial \bn}(\kappa p)+\bdiv_{\partial \OS}(\bn)\kappa p \Big],
\end{align}
\end{subequations}
and $\mathcal{M}(f)$ denotes the material derivative of a given function $f$.
\end{theorem}
%where $H:=\nabla p$
\begin{proof}
Existence of shape Hessian $\big((\bsig'', p''), \br''\big)$ follows from Theorem \ref{thm.sh.der} and \eqref{def.app.sh}. 
We choose $((\btau,q),\bs) \in \Big((\bar{\mathcal{D}_1} \times \bar{\mathcal{D}_2}) \cap \mathbb{X} \Big)  \times \bcQ,$ and proceed in the similar lines as in \eqref{add.bc} and \eqref{bar.I7.new} to have
\begin{align*}
&\int_{\OS}\frac{1}{\rS}\bdiv \bsig'_{\delta} \cdot\bdiv\btau +\int_{\partial \OS}\frac{1}{\rS} \bG^{\delta} \cdot \bdiv \btau d \mathcal{S}
+\int_{\OF}\frac{1}{\rF} \nabla p'_{\delta}\cdot \nabla q 
-\mu^2 \int_{\OS}\cC^{-1}\bsig'_{\delta}:\btau
\notag\\&
-\mu^2 \int_{\OF}\frac{1}{\rF c^2}p'_{\delta}q
-\mu^2 \int_{\OS}\btau:\br'_{\delta}+\int_{\partial \OF}\frac{1}{\rF}
\nabla p^{\delta} \cdot \nabla q \, \kappa d\mathcal{S}
-\mu^2\int_{\partial \OS} \cC^{-1}\bsig^{\delta}:\btau\, \kappa d\mathcal{S}
\notag\\&
-\mu^2 \int_{\partial \OF}\frac{1}{\rF c^2}p^{\delta}q \,\kappa d\mathcal{S}
-\mu^2 \int_{\partial \OS}\btau:\br^{\delta} \, \kappa d\mathcal{S}+
\int_{\OS}\bsig'_{\delta}:\bs +\int_{\partial \OS}\bsig_{\delta}:\bs\,\kappa d\mathcal{S}=0.
\end{align*}
Proceeding in the same lines as in {\bf  Step I} following \eqref{add.bc} and \eqref{bar.I7.new}
 in the proof of Theorem \ref{thm.sh.der} for $\big((\bsig'_{\delta}, p'_{\delta}), \br'_{\delta}\big)$ and passing to the limit as $\delta \rightarrow 0,$ we achieve 
\eqref{eq.dom.S.h}, \eqref{eq.sig.dom.h} and \eqref{eq.dom.F.h}.
\par
Now we move to prove the boundary conditions satisfied by $\big((\bsig'', p''), \br''\big).$
Recurrent use of the same aguments used in {\bf Step II}, we have
\begin{align}\label{eq.bc.tau.h}
&\int_{\partial \OS} \dfrac{1}{\rS}\btau \bn \cdot \bdiv \bsig'' d \mathcal{S}+ \int_{\partial \OF} \dfrac{1}{\rF}  \dfrac{\partial p''}{\partial \bn}q d \mathcal{S}
\notag\\
&= \int_{\partial \OS} \mathbb{H}_1 \cdot \bdiv \btau d \mathcal{S}+ \int_{\partial \OF} \mathbb{H}_2 \cdot \nabla q d \mathcal{S}+ \int_{\partial \OS} \mathbb{H}_3 : \btau d \mathcal{S}+\int_{\partial \OF} \mathbb{H}_4 q d \mathcal{S},
\end{align}
where $\mathbb{H}_i$s are given by \eqref{def.Hi} for $i=1,2,3,4.$
Using Lemma \ref{lem.int.by.parts}, tangential Green's formula, $\dfrac{\partial \btau}{\partial \bn}\bn=0$ and $\dfrac{\partial q}{\partial \bn}=0$,  we obtain from \eqref{eq.bc.tau.h}
\begin{align}\label{all.h}
&\int_{\partial \OS} \dfrac{1}{\rS}\btau \bn \cdot \bdiv \bsig'' d \mathcal{S}+ \int_{\partial \OF} \dfrac{1}{\rF}  \dfrac{\partial p''}{\partial \bn}q d \mathcal{S} 
\notag\\
&=-\int_{\partial \OS}\Big[ \nabla \mathbb{H}_1 : \btau +
\btau \bn\cdot \dfrac{\partial \mathbb{H}_1}{\partial \bn} 
+
\bdiv_{\partial \OS} (\bn) \btau \bn \cdot \mathbb{H}_1 +\mathbb{H}_3 : \btau\Big]d \mathcal{S}
\notag\\
& \quad
- \int_{\partial \OF} \Big[\frac{1}{\rF}
q \bdiv \mathbb{H}_2 
+ \dfrac{1}{\rF}  \bdiv_{\partial \OF} (\bn) \mathbb{H}_2 \cdot \bn  q + \mathbb{H}_4 q\Big] d \mathcal{S}.
\end{align}
Recalling that $\partial \OS=\GD \cup \GN \cup \S$ and $\partial \OF=\S$ and choosing $\btau \in \bcW$ such that $\btau=0$ on $\GD \cup \GN$, we obtain the same equation as \eqref{all.h} with all integrals replaced by integrals over $\S$.  
Now by choosing $ q=-(\btau \bn) \cdot \bn$ in this resulting equation, we 
obtain \eqref{eq.sd.bc2.h}.
%\newline
%We now prove \eqref{eq.bc.u.h}, \eqref{eq.bc.sign.h} and \eqref{eq.bc.sig.prime.h}.

Using \eqref{eq.sd.bc2.h}, \eqref{all.h} and the same argument as in the proof of 
\eqref{eq.bc.u} we obtain \eqref{eq.bc.u.h}.
With the help of \eqref{def.app.sh}, we have \eqref{eq.bc.sign.h} and  \eqref{eq.bc.sig.prime.h}.
%This proves \eqref{eq.sd.bc2}.
Hence, the shape Hessian $\big((\bsig'', p''), \br''\big)$ satisfies \eqref{eq.sh.hes.D1}, completing the proof.
%which is true for all $((\btau,q),\bs) \in \mathbb{X} \times \bcQ$.
%Hence we have \eqref{eq.sh.der.D2}.
\end{proof}

\begin{rem}
Both problems \eqref{eq.sh.der.D1} and  \eqref{eq.sh.hes.D1} for the shape derivative and shape Hessian, respectively can be solved by using the methods in \cite{Garcia,Gat06,GMM,MMT} etc.
\end{rem}

\subsection{Computation of  stochastic moments} \label{sec.comp}
In Sections \ref{sec.mat.der} and \ref{sec.shape.der}, we have defined material derivative, shape derivative and shape Hessian in which the quantities $\kappa$ and $\kappa_1$ are deterministic. 
Since \eqref{equ:perturbed prob} is posed on a domain with uncertainty located boundaries (see \eqref{equ:rand inter}), these derivatives also depend on $\omega.$ Thus, we compute the mean and the variance of the random solutions. The main result of this paper is stated below.

\begin{theorem} \label{thm.main}
Let $((\bsig^{\eps}(\omega),p^{\eps}(\omega)),\br^{\eps}(\omega)) $ be solution of \eqref{compactVar.eps} with the random interface $\partial \OS^{\eps}(\omega)$ given by \eqref{equ:rand inter}, and let $((\bsig,p),\br)$ be solution of the unperturbed problem (i.e., \eqref{compactVar.eps} for $\eps=0$) with reference interface $\partial \OS.$ 
Assume that the perturbation function $\kappa=\kappa_1$ belongs to $L^k(\mathfrak{U},C^{2,1}(\partial \OS))$ for an integer $k$
and $\bF \in (\mathrm{L}^2(B_R))^d \cap (\mathbb{X} \times \bcQ)^{*}.$ 
Then for sufficiently small $\eps \geq0,$ there exists compact set $K \subset (\OS \cup \OF) \cap (\OS^{\eps} \cup \OF^{\eps})$ such that 
\begin{itemize}
\item[1.] \label{lem.stoc.tay}
 $((\bsig^{\eps}(\omega),p^{\eps}(\omega)),\br^{\eps}(\omega))$ admits the asymptotic expansion
$\mP-\mbox{a.e.}\,\omega \in \mathfrak{U}$, i.e. for $\vecx \in (\OS \cup \OF) \cap (\OS^{\eps} \cup \OF^{\eps})$
\begin{align}\label{stoc.tay.}
((\bsig^{\eps}(\vecx,\omega),p^{\eps}(\vecx,\omega)),\br^{\eps}(\vecx,\omega))
&=((\bsig(\vecx),p(\vecx)),\br(\vecx))
+\eps ((\bsig'(\vecx,\omega),p'(\vecx,\omega)),\br'(\vecx,\omega))
\notag\\ & \quad
+\dfrac{\eps^2}{2}((\bsig''(\vecx,\omega),p''(\vecx,\omega)),\br''(\vecx,\omega))+O(\eps^3).
\end{align}
%p^{\eps}(\vecx,\omega)&=p(\vecx)+\eps p'(\vecx,\omega)+\dfrac{\eps^2}{2}p''(\vecx,\omega)+O(\eps^3),\\
%\br^{\eps}(\vecx,\omega)&=\br(\vecx)+\eps \br'(\vecx,\omega)+\dfrac{\eps^2}{2}\br''(\vecx,\omega)+O(\eps^3).
%\end{align}\end{subequations}
\item[2.] \label{meanvalue.lem} 
The mean and variance of the solution $((\bsig^{\eps}(\omega),p^{\eps}(\omega)),\br^{\eps}(\omega))$ can be approximated, respectively, by 
\begin{align}
\mE[((\bsig^{\eps},p^{\eps}),\br^{\eps})]&=((\bsig,p),\br)+O(\eps^2),\label{meanvalue}\\
\big((\var(\bsig^{\eps}),\var(p^{\eps})),\var(\br^{\eps})\big)&=\eps^2\Big( \big(\mE [(\bsig')^2],\mE[(p')^2] \big), \mE[(\br')^2] \Big)+O(\eps^3),\label{var.main}
\end{align}
where for any tensor $\btau$, we denote by $\btau^2$ the tensor product $\btau:\btau$.

%\end{subequations}
%\item[3.] \label{var.main.lem} Variance $\var\Big((\bsig^{\eps}(\vecx),p^{\eps}(\vecx)),\br^{\eps}(\vecx)\Big)$ of the random solution of \eqref{compactVar.eps} satisfies 
%\begin{align}
%
%\end{align}
\end{itemize}
\end{theorem}

\begin{proof}
We first start with proving \eqref{stoc.tay.}.
With the shape derivative and shape Hessian of $((\bsig,p),\br)$ given in 
Theorems \ref{thm.sh.der} and \ref{thm.sh.hes} and equation \eqref{equ:V def}, using the Taylor expansion \eqref{det.tay.} for an arbitrary, fixed realization $\kappa(\cdot,\omega),\omega \in \mathfrak{U},$ we have the stochastic counterpart \eqref{stoc.tay.}. 
%Hence, we infer that 
%$((\bsig^{\eps},p^{\eps}),\br^{\eps})$ admits `shape Taylor expansion'  \eqref{stoc.tay.}.
%\item
%Taking into consideration \eqref{equ:V def} and \eqref{reg.kappa}, by virtue of  \ref{lem.stoc.tay}., 
\par
We now move to prove \eqref{meanvalue}.
On taking expectation, we have
%$\mP-\mbox{a.e.}\,\omega \in \mathfrak{U}$
\begin{align*}%\label{stoc.tay.sig}
\mE[\bsig^{\eps}(\vecx,\cdot)]=\mE\Big[\bsig(\vecx)+\eps \bsig'(\vecx,\cdot)+\dfrac{\eps^2}{2}\bsig''(\vecx,\cdot)+O(\eps^3)\Big]\quad \mbox{for} \quad \vecx \in (\OS \cup \OF) \cap (\OS^{\eps} \cup \OF^{\eps}).
\end{align*}
Since $\big((\bsig', p'), \br'\big)$ depends linearly on $\kappa$, exploiting $\mE [\kappa] = 0$, see~\eqref{equ:kappa sym}, it can be seen that $\mE[\big((\bsig', p'), \br'\big)]$~satisfies~\eqref{eq.sh.der.D1} with zero boundary data, and hence $\mE [\big((\bsig', p'), \br'\big)] = 0.$
This  proves \eqref{meanvalue}.
\item
We note that the quantity $\var(\bsig^{\eps})$ for a tensor $\bsig^{\eps}$ is given by the following
\begin{align}\label{var.sig.def}
\var(\bsig^{\eps}):=\mE \Big[(\bsig^{\eps}-\mE[\bsig^{\eps}]):(\bsig^{\eps}-\mE[\bsig^{\eps}])  \Big]
=\mE[\bsig^{\eps}:\bsig^{\eps}] -\Big(\mE[\bsig^{\eps}]:\mE[\bsig^{\eps}]\Big).
\end{align}
In similar manner, $\var(\br^{\eps})$
is given by 
 \begin{align*}
\var(\br^{\eps})
=\mE[\br^{\eps}:\br^{\eps}] -\Big(\mE[\br^{\eps}]:\mE[\br^{\eps}]\Big)
\end{align*}
Using stochastic Taylor expansion \eqref{stoc.tay.} we note that $\mathbb{P-}$a.e. $\omega \in \mathfrak{U}$
\begin{align}\label{stoc.sig.tay}
\bsig^{\eps}(\vecx,\omega)=\bsig(\vecx)+\eps \bsig'(\vecx,\omega)+\dfrac{\eps^2}{2}\bsig''(\vecx,\omega)+O(\eps^3).
\end{align}
Keeping in mind $\bsig^2(\vecx)=\bsig(\vecx):\bsig(\vecx)$ and $(\bsig'(\vecx,\omega))^2=\bsig'(\vecx,\omega):\bsig'(\vecx,\omega),$ we see that $\mathbb{P-}$a.e. $\omega \in \mathfrak{U}$
\begin{align}\label{eq.sh.hes}
&(\bsig^{\eps}(\vecx,\omega))^2=\Big(\bsig(\vecx)+\eps \bsig'(\vecx,\omega)+\dfrac{\eps^2}{2}\bsig''(\vecx,\omega)+O(\eps^3)\Big)^2 
\notag\\
&=\bsig^2(\vecx)+\eps^2 (\bsig'(\vecx,\omega))^2
+2 \eps \bsig(\vecx):\bsig'(\vecx,\omega)
+\eps^2 \bsig(\vecx):\bsig''(\vecx,\omega)+O(\eps^3).
\end{align}
%Here, $\big((\bsig', p'), \br'\big)$ is the solution of~\eqref{eq.sh.der.D1}.
 %yielding~\eqref{equ:Exp}.
Hence, on taking expectation on both sides of \eqref{eq.sh.hes}, we have
\begin{align*}%\label{var1}
\mE \Big[\bsig^{\eps}(\vecx)^2\Big]&=\bsig^2(\vecx)+\eps^2\mE\Big[(\bsig'(\vecx))^2\Big] +\eps^2 \bsig(\vecx):\mE[\bsig''(\vecx)] +O(\eps^3).
\end{align*}
On taking expectation and then squaring on both sides of \eqref{stoc.sig.tay} we have
\begin{align*}%\label{var2}
\Big(\mE[(\bsig^{\eps}(\vecx)]\Big)^2&=\Big(\bsig(\vecx)+\dfrac{\eps^2}{2} \mE[\bsig''(\vecx)] +O(\eps^3)\Big)^2=\bsig^2(\vecx)+\eps^2 \bsig(\vecx):\mE[\bsig''(\vecx)] +O(\eps^3).
\end{align*}  
This essentially concludes from \eqref{var.sig.def}
\begin{align}\label{var.sig}
\var(\bsig^{\eps}(\vecx))=\mE\Big[(\bsig^{\eps}(\vecx)^2\Big]
-\Big(\mE[(\bsig^{\eps}(\vecx)]\Big)^2
=\eps^2\mE \Big[(\bsig'(\vecx))^2\Big]+O(\eps^3).
\end{align}
Similarly for $p^{\eps}$ and $r^{\eps}$ we have
\begin{align}
&\var(p^{\eps}(\vecx))=\eps^2\mE \Big[(p'(\vecx))^2\Big]+O(\eps^3),
\quad
\var(\br^{\eps}(\vecx))=\eps^2\mE \Big[(\br'(\vecx))^2\Big]+O(\eps^3).\label{var.r}
\end{align}
Combining \eqref{var.sig}-\eqref{var.r} we
arrive at \eqref{var.main}, finishing the proof of this theorem. 
%This completes the proof.
%\end{itemize}
\end{proof}
\begin{rem} Observing \eqref{eq.cor} and $\mE[\bsig'(\vecx)]=0,$ it can be observed that $$\var(\bsig(\vecx))=\mbox{Cor}(\bsig(\vecx),\bsig(\vecy))|_{\vecx=\vecy}.$$ Similarly, $\var(p(\vecx))=\mbox{Cor}(p(\vecx),p(\vecy))|_{\vecx=\vecy}, \quad\var(\br(\vecx))=\mbox{Cor}(\br(\vecx),\br(\vecy))|_{\vecx=\vecy}.$
\end{rem} 
%\section{Future Plan}
\section{An Example}\label{sec.ex}
In this section, we present a particular example of the solid--fluid problem in a square domain. In this example we will solve a slightly different problem \eqref{equ:perturbed prob} with \eqref{dual-problem_6.ep} replaced by 
\begin{align}\label{nonhomo.bc}
\bsig^{\eps}(\omega)\bn^{\eps}+p^{\eps}(\omega) \bn^{\eps}=g \qon \S^{\eps}(\omega)
\end{align}
and $\GN=\varnothing.$
\begin{figure}[h]
	\centering
	\begin{pdfpic}
\psset{unit=0.5cm}
		\begin{pspicture}
		%\psgrid[subgriddiv=1,griddots=10,gridlabels=10pt](0,0)(12,12)
		%
		\pspolygon(1,10)(1,0)(11,0)(11,10)
		\pspolygon(4,7)(4,3)(8,3)(8,7)
		\psline{->}(6,-1)(6,12)
		\psline{->}(0,5)(12,5)
		
		\def\blockf{
			%\psline(0,0)(10,0)
			\multirput(0,0)(0.5,0){21}{\psline(0,0)(1,1)}
		}
		\rput(1,10){\blockf}
		
		\uput[ul](6,11){$\Gamma_1$}
		\uput[d](6,-1){$\Gamma_3$}
		\uput[l](1,2.5){$\Gamma_4$}
		\uput[r](11,2.5){$\Gamma_2$}
		
		\uput[u](7,7){$\Sigma_1$}
		\uput[d](7,3){$\Sigma_3$}
		
		\uput[r](8,6){$\Sigma_2$}
		\uput[l](4,6){$\Sigma_4$}
		
		\psline{->}(3.5,0)(3.5,-0.8)
		\psline{->}(8.5,0)(8.5,-0.8)
		\psline{->}(1,7.5)(0.2,7.5)
		\psline{->}(11,7.5)(11.8,7.5)
		%\psline{->}(11,10)(11,12)
		
		\uput[d](3.5,-0.8){$(0,-1)$}
		\uput[d](8.5,-0.8){$(0,-1)$}
		\uput[l](0.2,7.5){$(-1,0)$}
		\uput[r](11.8,7.5){$(1,0)$}
		%\uput[u](11,12){$(0,1)$}

		\def\blockpt{
			\multirput(0,0)(0.5,0){7}{\psdot[dotsize=2pt]}
		}
		\multirput(4.5,3.5)(0,0.5){7}{\blockpt}
		\end{pspicture}
	\end{pdfpic}
	\caption{ }
	\label{fig:1}
\end{figure}

All the theoretical results in Section \ref{sec.shape.der} still hold, except that \eqref{dual-problem_6.ep} will have a correction term due to the non-homogeneous condition \eqref{nonhomo.bc}. 
%We will show that the computation of this example supports every theoretical findings established in this paper. 
%This indicates a positive direction of the theoretical findings in a general domain.
We consider \eqref{equ:perturbed prob} for $\eps=0$ on the domains $\OF:=[-1,1]^2\subset \mathbb{R}^2$ and $\OS:=[-2,2]^2\setminus \OF$ and take the parameters $\mu^2=6 \pi^2,\,\rS=3,\,\rF=\lambda=\nu=1$. Then we choose the data $\bF$ so that the exact solution for the displacement, pressure and the stress tensor of the considered unperturbed problem are given, respectively, by
\begin{align*}
\bu(\vecx,\vecy)&=\begin{pmatrix}
\sin\pi \vecx\, \sin \pi \vecy\\
\sin\pi \vecx \,\sin \pi \vecy
\end{pmatrix} \,\,  \forall (\vecx,\vecy)\in \OS 
, \quad \quad
p(\vecx,\vecy)=\cos\pi \vecx\, \cos \pi \vecy, \,\,  \forall (\vecx,\vecy)\in \OF\\
\bsig(\vecx,\vecy)&=\pi\begin{pmatrix}
\sin\pi (\vecx+\vecy)+2\cos \pi \vecx \sin \pi \vecy & \sin \pi(\vecx+\vecy)\\
\sin \pi(\vecx+\vecy) & \sin\pi (\vecx+\vecy)+2\sin \pi \vecx \cos \pi \vecy
\end{pmatrix}
  \,\,  \forall (\vecx,\vecy)\in \OS.
\end{align*}
%if we impose on $\S$ the transmission condition
%\[
%\bsig (\bu)\bn+p \bn=\bsig(\bu_e) \bn+p \bn.
%\]
Let the random interface $\Gamma^\eps(\omega)$ be given~by
\[
\Gamma^\eps(\omega)
=
\{
\vecx + \eps\kappa(\vecx,\omega)\vecn(\vecx)
:
\vecx\in\Gamma
\}
\]
where $\Gamma:=(\cup_{i=1}^4 \Gamma_i) \cup (\cup_{i=1}^4 \Sigma_C^i)$.
%where
%\begin{align*}
%&\Gamma_1:=\{(\vecx,2):-2 \leq \vecx \leq 2   \}, \quad \Gamma_2:=\{(2,\vecy):-2 \leq \vecy \leq 2   \}, \quad \Gamma_3:=\{(\vecx,-2):-2 \leq \vecx \leq 2   \},\\
%&\Gamma_4:=\{(-2,\vecy):-2 \leq \vecy \leq 2   \},\quad \Sigma_C^1:=\{(\vecx,1): -1 \leq x \leq 1  \},\quad  \Sigma_C^2:=\{(1,\vecy): -1 \leq \vecy \leq 1    \},
%\\
%&
%\Sigma_C^3:=\{(\vecx,-1): -1 \leq x \leq 1 \}
% \quad \Sigma_C^4:=\{(-1,\vecy): -1 \leq \vecy \leq 1    \}.
%\end{align*}
Next, let us consider the perturbed domain 
 \[\OF^{\eps}(\omega):=[-1+\eps a(\omega),1+\eps a(\omega)]^2 \quad \mbox{and}\quad \OS^{\eps}(\omega):=[-2+\eps a(\omega),2+2\eps a(\omega)]^2
\setminus \OF^{\eps}(\omega) \] 
where 
%\eqref{equ:lap equ}--\eqref{equ:inft cond 1}
the perturbation parameter $\kappa(\vecx,\omega) = a(\omega)$ has a constant (but random) value over the whole $\partial \OF$. The random variable $a(\omega)$ takes values in $[-1,1]$ and is centred so that $\mE[\kappa] \equiv \mE[a] = 0$. Further, we consider
%\begin{itemize} \item
 $a(\omega)$ is a uniformly distributed random variable with values in $[-1,1]$ and probability density function (PDF) $\rho_1(t) = 1/2$, so that
\begin{equation}\label{cov_rho1}
 \Covv[\kappa](\vecx,\vecy) \equiv \mE[a^2] = \int_{-1}^1 t^2 \rho_1(t) \, dt = \frac{1}{3}.
\end{equation}
Solution for the displacement and pressure, of \eqref{equ:perturbed prob} are given, respectively, by
\begin{align*}
\bu^\eps(\vecx,\vecy,\omega)&=\begin{pmatrix}
\sin\pi (\vecx-\eps a(\omega) )\, \sin\pi (\vecy-\eps a(\omega) ) \\
\sin\pi (\vecx- \eps a(\omega))\,\sin \pi (\vecy-\eps a(\omega))
\end{pmatrix}\quad \forall (\vecx,\vecy)\in \OS^\eps \notag\\ 
p^\eps(\vecx,\vecy,\omega)&=\cos\pi (\vecx-\eps a(\omega) )\, \cos\pi (\vecy-\eps a(\omega) )  \qquad \,\,\, \forall (\vecx,\vecy)\in \OF^\eps.
\end{align*}

We now split the verification in three steps. In the first two steps, we will verify equations satisfied by the shape derivative (Theorem \ref{thm.sh.der}) and shape Hessian (Theorem \ref{thm.sh.hes}). In the last step, we will verify Theorem \ref{thm.main}.

\underline{{\bf Step I:} }  \quad
Exploiting \eqref{def.app.sd} and Lemma \ref{lem.cpt.k}, we see that the shape derivative of $\bu$ and $p$ denoted, 
respectively, by $\u'$ and $p'$ are given by 
\begin{align*}
\u'(\vecx,\vecy,\omega)&=-a(\omega)\pi\begin{pmatrix}
\sin\pi (\vecx+\vecy ) \\
\sin\pi (\vecx+\vecy)
\end{pmatrix}\quad 
\qquad \qquad \qquad \qquad
\forall (\vecx,\vecy)\in \OS, \notag\\ 
\bsig'(\vecx,\vecy,\omega)&=-a(\omega)\pi^2 \begin{pmatrix}
4\cos\pi (\vecx+\vecy ) & 2\cos\pi (\vecx+\vecy) \\
 2\cos\pi (\vecx+\vecy) & 4\cos\pi (\vecx+\vecy)
\end{pmatrix}\quad \forall (\vecx,\vecy)\in \OS, \notag\\ 
%\mbox{and}\quad
p'(\vecx,\vecy,\omega)&=a(\omega) \pi \sin\pi (\vecx+\vecy ) \qquad \qquad \qquad \qquad \qquad \quad \,\, \forall (\vecx,\vecy)\in \OF.
\end{align*}
Hence elementary calculations reveal that $\big((\bu',\bsig'),p'\big)$ satisfies
%\begin{subequations}\label{eq.sh.der.D1.ex}
\begin{align*}
  \bdiv \bsig'+6 \pi^2 \bu' &=\0 \qquad \qquad \qquad \qquad \qquad \quad \, \qin \OS,
%\label{eq.dom.S.ex} \\
\\
 \bsig' &= \cC \beps(\bu') \qquad \qquad \qquad \qquad \,\,\,\, \qin\OS,
%\label{eq.sig.dom.ex} \\
\\
\dfrac{1}{\rS \mu^2} \bdiv_{\partial \OS}(\bn) \bG-\kappa (\nabla \bu) \bn &=\left\{\begin{aligned}
  & \quad a(\omega)\pi \begin{pmatrix}
\sin\pi \vecx  \\
\sin\pi \vecx
\end{pmatrix}
  \quad \quad \quad   \qon  \Gamma_1 \cup \Gamma_3,\\
   &-a(\omega)\pi \begin{pmatrix}
\sin\pi \vecy  \\
\sin\pi \vecy
\end{pmatrix}
  \quad \quad \quad \qon  \Gamma_2 \cup \Gamma_4,  
  \end{aligned}
   \right. %\label{eq.bc.u.ex}\\
\\
\Delta p' +2 \pi^2 p'&=0 \qquad 
\quad \qquad \quad   \quad \quad \qquad \,\, \,\,
\quad \qin\OF,%\label{eq.dom.F.ex}\\
\\
 \dfrac{1}{\rS} (\bdiv_{\partial \OF}\bn) \bG \cdot \bn  - \mu^2 \kappa ((\nabla &\bu) \bn)\cdot \bn+\dfrac{1}{\rF} \Big[ \divv \big(\kappa \nabla p   \big)
-\dfrac{\partial}{\partial \bn} \big(\kappa \nabla p  \big) \cdot \bn-\bdiv_{\partial \OF} (\bn) \kappa \dfrac{\partial p}{\partial \bn} \notag\\
+\dfrac{\mu^2}{c^2} p  \kappa\Big]
&=
\left\{\begin{aligned}
           & a(\omega)\pi^2 \cos \pi \vecx -6\pi^3 a(\omega) \sin \pi \vecx \quad \, \,\,\hbox{on } \S^1,\\
           & a(\omega)\pi^2 \cos \pi \vecy -6\pi^3 a(\omega) \sin \pi \vecy \quad \, \,\,\,\hbox{on }  \S^2, \\ 
             & a(\omega)\pi^2 \cos \pi \vecx +6\pi^3 a(\omega) \sin \pi \vecx \quad \, \,\,\,\hbox{on }  \S^3,\\        
             & a(\omega)\pi^2 \cos \pi \vecy +6\pi^3 a(\omega) \sin \pi \vecy \quad \, \,\,\,\hbox{on }  \S^4.
\end{aligned}
   \right. %\label{eq.sd.bc2.ex}
\end{align*}
%\end{subequations}
%which is same as the right hand side of the corresponding terms in \eqref{eq.sh.der.D1}. 
%By simple computation, one can see that
%right hand side of \eqref{eq.bc.u} and \eqref{eq.sd.bc2} are same as the right hand side of \eqref{eq.bc.u.ex} and \eqref{eq.sd.bc2.ex}. 
Therefore, computation of the shape derivative agrees with our result \eqref{eq.sh.der.D1} in Theorem \ref{thm.sh.der}.
We also note that $\bG$ defined in \eqref{def.G} is given by
\begin{align*}
\bG =2a(\omega)\pi^2 \begin{pmatrix}
\cos \pi(\vecx+\vecy) &-\sin\pi\vecx \sin \pi \vecy\\
\cos \pi(\vecx+\vecy) &-\sin\pi\vecx \sin \pi \vecy
\end{pmatrix}.
\end{align*}
\underline{{\bf Step II:} }  Let us consider another perturbation parameter $\kappa_1(\vecx,\omega) = b(\omega)$ which has a constant (but random) value over the whole $\partial \OF$. The random variable $b(\omega)$ possesses similar properties as $a(\omega).$
Hence, $\mE[\kappa_1]= 0$ and $\Covv[\kappa_1](\vecx,\vecy) \equiv \mE[b^2] = \frac{1}{3}.$ Again using \eqref{def.app.sh} and Lemma \ref{lem.cpt.k}, we now calculate
the shape Hessian of $\bu,\bsig$ and $p$ 
denoted, respectively, by $\u'',\bsig''$ and $p''$. 

One can clearly see that $\bdiv_{\partial \OF} \bn=0$ and $\bdiv_{\partial \OS}\bn=0.$ Hence on computing the right hand side of \eqref{eq.bc.u.h} and \eqref{eq.sd.bc2.h}, there is no contribution coming from the term $\mathbb{H}_1.$
We also observe that on $\S^1,$ using \eqref{def.Hi} we obtain
\begin{align*}
\mathbb{H}_2=a(\omega) b(\omega) \pi^2\begin{pmatrix}
\cos \pi\vecx  \\
0
\end{pmatrix},
\quad
\mathbb{H}_3=-12a(\omega)b(\omega) \pi^4 \cos \pi \vecx,
\quad
\mathbb{H}_4= -4a(\omega)b(\omega) \pi^3 \sin \pi \vecx.
\end{align*}
In a similar manner, one can compute $\mathbb{H}_i$ for $i=2,3,4$ on $\S^2 \cup \S^3 \cup \S^4.$
Therefore, by elementary calculations one can see that $\big((\bu'',\bsig''),p''\big)$  is given by 
\begin{align*}
\u''(\vecx,\vecy,\omega)&=2a(\omega)b(\omega)\pi^2\begin{pmatrix}
\cos\pi (\vecx+\vecy ) \\
\cos\pi (\vecx+\vecy)
\end{pmatrix}\qquad \qquad \qquad \qquad \qquad \qquad \forall (\vecx,\vecy)\in \OS, \notag\\ 
\bsig''(\vecx,\vecy,\omega)&=- \begin{pmatrix}
8a(\omega)b(\omega) \pi^3 \sin\pi (\vecx+\vecy ) & -4a(\omega)b(\omega) \pi^3\sin\pi (\vecx+\vecy) \\
 -4a(\omega)b(\omega)\pi^3 \sin\pi (\vecx+\vecy) & -8a(\omega)b(\omega)\pi^3 \sin\pi (\vecx+\vecy)
\end{pmatrix}\quad \forall (\vecx,\vecy)\in \OS, \notag\\ 
%\mbox{and}\quad
p''(\vecx,\vecy,\omega)&=-2a(\omega)b(\omega) \pi^2 \cos\pi (\vecx+\vecy ) \qquad \qquad \qquad \qquad \qquad \qquad \quad \forall (\vecx,\vecy)\in \OF,
\end{align*}
and satisfies 
%\begin{subequations}\label{eq.sh.hes.D1.ex}
\begin{align*}
  \bdiv \bsig''+6 \pi^2 \bu'' &=\0 \qquad \qquad \qquad \qquad \qquad \qquad
  \quad \quad \, \qin \OS,
%\label{eq.dom.S.h.ex} \\
\\
 \bsig'' &= \cC \beps(\bu'') \qquad \qquad \qquad \qquad \qquad \quad \,\,\, \qin\OS,
%\label{eq.sig.dom.h.ex} \\
\\
\dfrac{1}{\mu^2}( \bdiv_{\partial \OS} \bn) \mathbb{H}_1+\dfrac{1}{\mu^2} \mathbb{H}_3 \bn&=\left\{\begin{aligned}
  &  2ab\pi^2 \begin{pmatrix}
\cos\pi \vecx  \\
\cos\pi \vecx
\end{pmatrix}
  \quad \quad \quad \quad \quad \quad \,\, \,\,\, \qon  \Gamma_1 \cup \Gamma_3,\\
   &2ab\pi^2 \begin{pmatrix}
\cos\pi \vecy  \\
\cos\pi \vecy
\end{pmatrix}
  \quad \quad \quad \quad \qquad\, \,\, \,\,\, \qon  \Gamma_2 \cup \Gamma_4,  
  \end{aligned}
   \right. %\label{eq.bc.u.h.ex}\\
\\
\Delta p'' +2 \pi^2 p''&=0 \qquad 
\quad \qquad \quad   \quad \qquad \qquad \, \,\,
\qquad \quad \qin\OF, %\label{eq.dom.F.h.ex}\\
\\
 -\dfrac{1}{\rS} (\bdiv_{\partial \OF}\bn) \mathbb{H}_1 \cdot \bn-\mathbb{H}_3& \bn \cdot \bn 
-\bdiv \mathbb{H}_2 +\bdiv_{\partial \OF}(\bn)\mathbb{H}_2 \cdot \bn 
+\mathbb{H}_4 \notag\\
&
=
\left\{\begin{aligned}
           & -2ab \pi^3\sin\pi \vecx+12ab \pi^4 \cos \pi \vecx \,\,\,\, \qon \S^1,\\
            & -2ab \pi^3\sin\pi \vecy+12ab \pi^4 \cos \pi \vecy \,\,\,\, \qon \S^2, \\         
             &  -2ab \pi^3\sin\pi \vecx-12ab \pi^4 \cos \pi \vecx \,\,\,\, \qon \S^3,\\
             & -2ab \pi^3\sin\pi \vecy-12ab \pi^4 \cos \pi \vecy \,\,\,\, \qon \S^4.
\end{aligned}
   \right. %\label{eq.sd.bc2.h.ex}
\end{align*}
%\end{subequations}
%By simple computation, one can verify  that right hand side of \eqref{eq.bc.u.h} and \eqref{eq.sd.bc2.h} is same as the right hand side of \eqref{eq.bc.u.h.ex} and \eqref{eq.sd.bc2.h.ex}. 
Therefore, computation of the shape Hessian agrees with our result \eqref{eq.sh.hes.D1} in Theorem \ref{thm.sh.hes}.

\underline{{\bf Step III:} }  \quad
In this step, we choose $\kappa=\kappa_1.$
Using $\mE[a] = 0,$ equation \eqref{cov_rho1}, 
%$\mE \Big[\sin \pi(\vecx-\eps a(\omega))\sin\pi(\vecy-\eps a(\omega))\Big]=\sin \pi \vecx\,\sin \pi \vecy+O(\eps^2)$ 
and 
$\mE \Big[\cos \pi(\vecx-\eps a(\omega))\cos\pi(\vecy-\eps a(\omega))\Big]=\cos \pi \vecx\,\cos \pi \vecy+O(\eps^2),$
one can derive
\begin{align}\label{taylor.p}
p^{\eps}(\vecx,\vecy,\omega)&=p(\vecx,\vecy)+\pi \eps a(\omega) \sin\pi(\vecx+\vecy)-\pi^2 a^2(\omega)\eps^2 \cos \pi(\vecx+\vecy)+a^3 O(\eps^3),
\end{align}
which verifies \eqref{stoc.tay.} for $p^{\eps}(\omega)$.
Proceeding in similar lines, it is easy to observe that $\bu^{\eps}(\omega)$ and $\bsig^{\eps}(\omega)$ admit the asymptotic shape Taylor expansion given by \eqref{stoc.tay.}. 
\newline
Again taking the expectation on \eqref{taylor.p} and using $\mE[a]=0,\mE[a^2]=1/3,$ rudimentary calculations reveal that $\mE[p^{\eps}(\vecx,\vecy,\cdot)]=p(\vecx,\vecy)+O(\eps^2)$ for $(\vecx,\vecy)\in \OF.$
In a similar way, one can verify 
\[\mE[\u^{\eps}(\vecx,\vecy,\cdot)]
=\bu(\vecx,\vecy)+O(\eps^2),\quad
\mE[\bsig^{\eps}(\vecx,\vecy,\cdot)]
=\bsig(\vecx,\vecy)+O(\eps^2), \quad (\vecx,\vecy)\in \OS.
\]
This verifies \eqref{meanvalue} in Theorem \ref{meanvalue.lem}.

Also exploiting $\mE[a] = 0,$ equation \eqref{cov_rho1}, and
$$\mE \Big[\cos \pi(\vecx-\eps a(\omega))\cos\pi(\vecy-\eps a(\omega))\Big]=\cos \pi \vecx\,\cos \pi \vecy -\dfrac{\pi^2 \eps^2}{3}\cos\pi(\vecx+\vecy)+a^3 O(\eps^3),$$
one can observe that for $(\vecx,\vecy) \in \OF$
\begin{align*}
\mE \big[p^{\eps}(\vecx,\vecy)\big]&=p(\vecx,\vecy)-\dfrac{\pi^2 \eps^2}{3}\cos\pi(\vecx+\vecy)+a^3 O(\eps^3)\\
\mE \big[p^{\eps}(\vecx,\vecy)\big]^2&=p^2(\vecx,\vecy)+\dfrac{\pi^2 \eps^2}{3}\sin^2 \pi(\vecx+\vecy)-\dfrac{\pi^2 \eps^2}{3}p(\vecx,\vecy)\cos \pi(\vecx+\vecy)+O(\eps^3).
\end{align*}
This implies,
%\begin{align}\label{eq.varp}
$\var[p^{\eps}]= \mE \big[p^{\eps}\big]^2-\big[\mE (p^{\eps})\big]^2
=\dfrac{\pi^2 \eps^2}{3} \sin^2 \pi(\vecx+\vecy)+O(\eps^3).$
%\end{align}
\newline
Since $p'(\vecx,\vecy)=\pi a (\omega)\sin \pi (\vecx+\vecy),$ one has
$\mE [p']^2=\dfrac{\pi^2}{3} \sin^2 \pi(\vecx+\vecy).$
Hence, we have $\var[p^{\eps}]=\eps^2\mE \Big[(p')^2\Big]+O(\eps^3).$
%
%\begin{align}
%\var[p^{\eps}]=\eps^2\mE \Big[(p')^2\Big]+O(\eps^3).
%\end{align} 
In a similar manner, one can check that $\var[\u^{\eps}]=\eps^2 \mE \Big[(\bu')^2\Big]+O(\eps^3).$
Therefore, approximation of the variance agrees with \eqref{var.main} in Theorem \ref{meanvalue.lem}. Thus we corroborate the theoretical results achieved in this paper on a square domain.

%\begin{align}
%&\var[\u^{\eps}]=\eps^2 \mE \Big[(\u')^2\Big]+O(\eps^3).
%\end{align}

%\end{itemize}
%The interface $\partial \OF^{\eps}(\omega)$
%is the sphere centred at the origin and having radius
%$R_\eps(\omega) = 1+\eps a(\omega)$.

%\subsection{Analytic examples}

%In this section, we choose the right hand side $f$ to be

\begin{appendix}
\section{Appendix}
This section has been split into 3 parts. The very first subsection consists of basic tensor
algebra notations and integration by parts formula for tensor-valued functions. The second
subsection presents a series of lemmas which are involved in the analysis. 
%in Subsection \ref{subsec.spectralT} and Section \ref{sec.shape.der}. 
The last subsection is devoted to the introduction of necessary concepts regarding shape
derivative and shape Hessian for $H^{\alpha}$ functions when $\alpha >0$. 

\subsection{Tensor algebra} \label{sec.tensor}
This section is based on Kronecker product and some of its properties; see \cite{HJ}.
\begin{defi}\label{def:LA}
Let~$A=(a_{ij})_{i,j=1,2}$ be a $2\times2$ matrix and~$\bsig=(\bsig_{ij})_{i,j=1,2}$ be 
a~$2\times2$-tensor-valued function. We define
\[
\cL_A\bsig := 
\left(
A^\top 
\begin{bmatrix}
\partial_1 \\ \partial_2
\end{bmatrix}
\right)^\top
\begin{bmatrix}
\bsig_{11} & \bsig_{12} \\ \bsig_{21} & \bsig_{22}
\end{bmatrix}
:=
\begin{bmatrix}
a_{11}\partial_1\bsig_{11}
+
a_{21}\partial_2\bsig_{11}
+
a_{12}\partial_1\bsig_{12}
+
a_{22}\partial_2\bsig_{12}
\\
a_{11}\partial_1\bsig_{21}
+
a_{21}\partial_2\bsig_{21}
+
a_{12}\partial_1\bsig_{22}
+
a_{22}\partial_2\bsig_{22}
\end{bmatrix}.
\]
We note that when~$A$ is the identity matrix, the operator~$\cL_{\bf I}$ is not the gradient of the
tensor~$\bsig$.
\end{defi}

\begin{defi} Let $A \in \R^{m \times n}$ and $B \in \R^{p \times q}$ be two matrices. The
Kronecker product of $A$ and $B,$ denoted by $A \otimes B,$ which is an element of $\R^{mp
\times nq}$ and is given by
%\begin{align}
\[
A \otimes B:=
  \begin{bmatrix}
    a_{11}B & \ldots & a_{1n}B \\
    a_{21}B & \dots & a_{2n}B\\
    \ldots & \ldots & \ldots\\
    a_{m1}B & \ldots & a_{mn}B
  \end{bmatrix}
\]
Note that $A \otimes B \neq B \otimes A.$
%\end{align}
\end{defi}

The component-wise inner product of two matrices $A,B \in \R^{m \times n}$ is denoted by 
\[
A:B=\mbox{Tr}(A^\top B)
\]
 where Tr denotes the trace of the matrix. 
\begin{lem}
Let $A \in \R^{m \times n},\,B \in \R^{i \times j},\,C \in \R^{k \times l}.$ Then 
\begin{enumerate}
\item
$ A \otimes (B \otimes C)= (A \otimes B) \otimes C, $
\item
$A \otimes (B+C)=A \otimes B+ A \otimes C$ when $i=k$ and $j=l$. %,\quad \mbox{when} \quad i=k,j=l.$
\end{enumerate}
\end{lem}
\begin{lem}\label{lem.int.by.parts} Let $D$ be a $C^2$ domain in $\R^d.$ Then for $\btau,\bsig \in (\mathrm{H}^2(D))^{d \times d},$ 
%\begin{subequations}\label{int.by.parts}
\begin{align*}
\int_D \bdiv \bsig \cdot \bdiv \btau\, dx &=-\int_{D} \bsig : \nabla (\bdiv \btau) \,dx
+\int_{\partial D} \bsig \bn \cdot \bdiv \btau\, d\mathcal{S}(\vecx)
\\
\int_{\partial D} \bdiv \bsig \cdot \bdiv \btau \, d \mathcal{S}
&=\int_{\partial D} \Big[-\bsig:\nabla (\bdiv \btau)+ \mathcal{F}(\bsig) \cdot \bdiv \btau  + (\bsig \bn) \cdot \dfrac{\partial}{\partial \bn}(\bdiv \btau)\Big]d \mathcal{S} \notag\\
&\quad+\int_{\partial D}  \bdiv_{\partial D} (\bn) (\bsig \bn) \cdot (\bdiv \btau)\,d \mathcal{S},
\end{align*}
%\end{subequations}
where $\bn$ is the unit outward normal vector to the boundary $\partial D$
and 
\begin{align}
\bf D \sigma&:= \begin{bmatrix}
\bf D \sigma^{(1)} 
&
\ldots
& 
\bf D \sigma^{(N)} 
\end{bmatrix}^{\top}\label{def:DF1}\\
\mathcal{F}(\bsig)&: = \begin{bmatrix}
\langle \bf D \sigma^{(1)}\bn,\bn \rangle_{\R^N} 
&
\ldots
& 
\langle \bf D \sigma^{(N)} \bn,\bn \rangle_{\R^N} 
\end{bmatrix}^{\top}\label{def:DF2}
\end{align}
%\begin{align*}
% \mathcal{F}(\bsig)&: = \begin{bmatrix}
%\langle D \sigma^{(1)}\bn,\bn \rangle_{\R^N} 
%&
%\ldots
%& 
%\langle D \sigma^{(N)} \bn,\bn \rangle_{\R^N} 
%\end{bmatrix}^{\top}.
%\end{align*}
\end{lem}
%\begin{proof}
%\eqref{int.by.parts} can be obtained by following standard integration by parts formula for vector valued functions.
%\end{proof}
\begin{proof}
For any $\btau,\bsig \in (\mathrm{H}^2(D))^{d \times d}$ we see that
\begin{align}
&\intg \bdiv \bsig \cdot \bdiv \btau \, d \mathcal{S} =\sum_k \intg (\bdiv \bsig)_k (\bdiv \btau)_k \,d \mathcal{S} =\sum_{i,j,k} \intg \dfrac{\partial \sigma_{jk}}{\partial x_j}
\dfrac{\partial \tau_{ik}}{\partial x_i}\,d \mathcal{S}\notag\\
&=-\underbrace{\sum_{i,j,k} \intg \sigma_{ij}
\dfrac{\partial^2 \tau_{ij}}{\partial x_i \partial x_k} \,d \mathcal{S}}_{K_1:=}
+\underbrace{ \sum_{i,j,k} \intg \dfrac{\partial}{\partial \bn} \big(\sigma_{ij} 
\dfrac{\partial \tau_{ij}}{\partial x_k} \big)n_i  \,d \mathcal{S}}_{K_2:=} +\underbrace{\sum_{i,j,k} \intg \kappa \sigma_{ij}
\dfrac{\partial \tau_{ki}}{\partial x_k}n_i \,d \mathcal{S}}_{K_3:=}.
\end{align}
%We note that
%\begin{align}
%\nabla(\bdiv \btau)=
%\begin{bmatrix}
%\dfrac{\partial^2 \tau_{11}}{\partial x_1^2}+\dfrac{\partial^2 \tau_{12}}{\partial x_1 \partial x_2} 
%&
%\dfrac{\partial^2 \tau_{11}}{\partial x_2 \partial x_1}+\dfrac{\partial^2 \tau_{12}}{ \partial x_2^2} 
%\\
%\dfrac{\partial^2 \tau_{21}}{\partial x_1^2}+\dfrac{\partial^2 \tau_{22}}{\partial x_1 \partial x_2} 
%&
%\dfrac{\partial^2 \tau_{21}}{\partial x_2 \partial x_1}+\dfrac{\partial^2 \tau_{22}}{\partial x_2^2} ,
%\end{bmatrix}
%\end{align}
To evaluate $K_1$, using $\big(\nabla(\bdiv \btau)\big)_{ki}=\sum_{j}\dfrac{\partial^2
\tau_{kj}}{\partial x_i \partial x_j}$, we obtain
\[
\bsig: \nabla(\bdiv \btau)
=\sum_{i,j} \sigma_{ij}
\dfrac{\partial (\bdiv \tau)_{i}}{\partial x_j}=\sum_{i,j} \sigma_{ij} \dfrac{\partial}{\partial x_j} \Big(\sum_k 
\dfrac{\partial \tau_{ik}}{\partial x_k}  \Big)
=\sum_{i,j,k} \sigma_{ij} \dfrac{\partial^2 \tau_{ik}}{\partial x_j \partial x_k}.
\]
%Hence,
Again using integration by parts formula, we have
\[
%K_1&=-\intg \bsig: \nabla(\bdiv \btau\,d \mathcal{S})\notag\\
K_2=\underbrace{\sum_{i,j,k} \intg \dfrac{\partial \sigma_{ij}}{\partial \bn} \dfrac{\partial
\btau_{kj}}{\partial x_k} n_i \,d \mathcal{S}}_{K_{21}:=} + 
\underbrace{\intg \sum_{i,j,k} \sigma_{ij} \dfrac{\partial^2 \btau_{kj}}{\partial \bn \partial
x_k} n_i \,d \mathcal{S}}_{K_{22}:=},
\]
which on further simplification yields 
\begin{align*}
K_{21}&=\sum_{i=1}^k \intg \dfrac{\partial \sigma_{ki}}{\partial \bn}n_i \mbox{div}(\btau^{(k)}) \,d \mathcal{S}
=\sum_k \intg \sum_{i,l} \dfrac{\partial \sigma_{ki}}{\partial x_l}\,n_l \, n_i\, \mbox{div}(\btau^{(k)})\,d \mathcal{S}
=\intg \langle \bf D \sigma^{(k)}\bn,\bn \rangle_{\R^N} \,d \mathcal{S},
\notag\\
K_{22}&= \sum_{i,j,k} \intg \sigma_{ki} \dfrac{\partial^2 \tau_{kj} }{\partial \bn \partial x_j} n_i
=\sum_{i,k} \intg \sigma_{ki} n_i \dfrac{\partial}{\partial \bn} (\bdiv \tau^{k} ) =\intg (\bsig \bn) \cdot \dfrac{\partial}{\partial \bn} (\bdiv \btau),\notag\\
K_3&=\sum_{k} \intg \kappa (\bsig^{k} \cdot \bn) \bdiv (\tau^{(k)})=\intg \kappa \big(\bsig \bn  \big)\cdot \bdiv \btau \,d \mathcal{S},
\end{align*}
where $\bf D \sigma$ and $\mathcal{F}(\bsig)$ are given by \eqref{def:DF1} and \eqref{def:DF2} respectively.
Combining all these we have the required result. 
\end{proof}

\subsection{Technical Lemmas}\label{sec.app.shape} 
If $V$ is given by \eqref{equ:V def.main}, until the end of this subsection let us assume that $\bT^\eps$ is defined by~\eqref{equ:T eps define} and~\eqref{equ:T e cond} with 
$\tilde\kappa\in C^{1}(\R^3)$, 
and denote its Jacobian matrix and Jacobian determinant by
$J_{\bT^\eps}$ and $\gamma(\eps,\cdot)$, respectively.
The following result is straightforward. 
\begin{lem}\label{lem:V prop}
{Assuming} $\tilde\kappa\in W^{1,\infty}(\R^3)$ and
$\tilde\kappa(\vecx) = 0$ for ${\vecx \in B_R^c}$, 
there hold $\bV\in 
\big(H^1(\R^3)\big)^3$ 
and
\begin{equation*}
\frac{\partial^m \bV(\vecx)}{\partial x_l^m} = \veczero
\quad
\forall \vecx\in {B_{R}^c},
\quad
l = 1,2,3,
\quad
m=0,1.
\end{equation*}
% and 
% \begin{equation}\label{equ:apr 2}
% V\in 
% \big(H^1(\R^3)\big)^3.
% \end{equation}
 \end{lem}
We denote $\bV(\vecx):= 
(\bV_1(\vecx), \bV_2(\vecx), \bV_3(\vecx))^\top$. For the explicit forms of $J_{\bT^\eps}(\cdot)$ and $\gamma(\eps,\cdot)$, we refer to \cite{CPT}. Exploiting Lemma~\ref{lem:V prop}, we derive that for sufficiently small $\eps>0$, there holds
\begin{equation}\label{equ:vbn 4}
\gamma(\eps,\vecx)
=
1 + \eps \underbrace{\gamma_1(\vecx) + \eps^2\gamma_2(\vecx) + \eps^3\gamma_3(\vecx)
}_{\tilde{\gamma}(\eps,\vecy):=}
\geq c > 0
\qquad
\forall \vecx\in \R^3.
\end{equation}
\begin{lem}\label{lem.conv.sig} For any $\bsig \in \mathbb{X,}$
\begin{align}
&1. \,\, \|J_{\bT^\eps}^{-\top}-\mathbf{I} \|_{L^{\infty}(\R^3)} \leq C \eps,\label{eq.conv.sig1}\\
&2. \,\, \|\cL_{J_{\bT^\eps}^{-1}}
\bsig- \cL_{\bf I} \bsig \|_{{L}^2(\R^3)} \leq C \eps\label{eq.conv.sig2}.
%&3. \,\, 
\end{align}
\end{lem}
\begin{proof}
Recalling the form of $J_{\bT^\eps}$ in \cite{CPT}, we see that
\begin{equation*}%\label{adj:Jc T}
\mbox{Adj}\,J_{\bT^\eps}(\vecx)
=\Big(\mathbf{I}+\eps \tilde{\bV}_1(\vecx)+\eps^2\tilde{\bV}_2(\vecx)\Big)^{\top}
\end{equation*}
where
\begin{equation*}
\tilde{\bV}_1=
\begin{bmatrix}
 \bV_{2,2}+ \bV_{3,3}
&
-\bV_{1,2}
&
-\bV_{1,3} \\
-\bV_{2,1}
&
\bV_{1,1}+\bV_{3,3}
& 
-\bV_{2,3}
\\
-\bV_{3,1}
&
-\bV_{3,2}
&
\bV_{1,1}+\bV_{2,2}
\end{bmatrix}
\end{equation*}
and 
\begin{equation*}
\tilde{\bV}_2=
\begin{bmatrix}
 \bV_{2,2} \bV_{3,3}-\bV_{2,3}\bV_{3,2}
&&
\bV_{1,3}\bV_{3,2}
-\bV_{1,2}\bV_{3,3}
&&
\bV_{1,2}\bV_{2,3}-\bV_{1,3}\bV_{1,2}
 \\
\bV_{2,3} \bV_{3,1}-\bV_{2,1}\bV_{3,3}
&&
\bV_{1,1} \bV_{3,3}-\bV_{1,3}\bV_{3,1}
& &
\bV_{1,3}\bV_{2,1}-\bV_{1,1}\bV_{2,3}
\\
\bV_{2,1}\bV_{3,2}-\bV_{2,2}\bV_{3,1}
&&
\bV_{1,2} \bV_{3,1}-\bV_{1,1}\bV_{3,2}
&&
\bV_{1,1}\bV_{2,2}-\bV_{1,2}\bV_{2,2}
\end{bmatrix}
\end{equation*}
%\begin{align}
%\quad \quad ,
%\end{align}
where $\bV_{i,j}=\dfrac{\partial \bV_i}{\partial x_j}.$
This implies
\begin{align*}%\label{jTeps.inv}
J_{\bT^\eps}^{-1}(\vecx)=\dfrac{1}{\gamma(\eps,\vecx)} \mbox{Adj}\,J_{\bT^\eps}(\vecx)=\dfrac{\mathbf{I}+\eps \tilde{\bV}_1(\vecx)+\eps^2\tilde{\bV}_2(\vecx)}{1+\eps \gamma_1(\vecx)+\eps^2\gamma_2(\vecx)+\eps^3 \gamma_3(\vecx)}=\bf I+\eps \hat{\bV}_1(\vecx,\eps)
\end{align*}
where $\hat{\bV}_1(\vecx,\eps)=\tilde{\bV}_1(\vecx)-\gamma_1 (\vecx)+O(\eps).$
This concludes \eqref{eq.conv.sig1}. Furthermore, we have
\begin{align*}
\cL_{J_{\bT^\eps}^{-1}}
\bsig
 =  
\cL_{\bf I}\bsig + \eps \cL_{\hat{\bV}_1}\bsig
\end{align*}
and this concludes \eqref{eq.conv.sig2}. 
This completes the proof.
\end{proof}
In view of Lemma 3.2 of \cite{CPT}, let us
consider $\mathcal{A}(\eps,\cdot):= \gamma(\eps,\cdot) J_{\bT^\eps}^{-1} J_{\bT^\eps}^{-\top}.$ Furthermore,
we denote $\mathcal{A}'(0,\cdot)$ is the G\^ateaux derivative of $\mathcal{A}(\eps,\cdot)$ 
%(determined by $T^\eps$) 
at $\eps =0$, namely
\[
\mathcal{A}'(0,\vecx)
=
\lim_{\eps\goto 0}
\frac{\mathcal{A}(\eps,\vecx) - \mathbf{I}(\vecx)}{\eps},
\quad\vecx\in \R^3.
\]

\begin{lem} \label{lem.conv.tensor}
Define operators $\tilde{\mathcal{A}}: \mathbb{X}\rightarrow L^2(\OS)$ and $\tilde{J}: \mathbb{X}\rightarrow L^2(\OS)$ by  
\begin{align}
\tilde{\mathcal{A}}(\eps,\vecy) \bsig(\vecy)&:=\gamma(\eps,\vecy)J_{\bT^{\eps}}^{-1} \otimes 
\cL_{J_{\bT^{\eps}}^{-1}}
 \bsig(\vecy),\label{def.ta}\\
\tilde{J}(\vecy)\bsig(\vecy)&:= \mathbf{I} \otimes 
\cL_{\bf I} \bsig(\vecy)
 \label{def.ti},
 \quad \forall \bsig(\vecy) \in \mathbb{X}.
\end{align}
Then \begin{align} \label{equ:st 1.ta}
\lim_{\eps \goto 0} \|\tilde{\mathcal{A}}(\eps,\cdot)\bsig-\tilde{J}(\cdot)\bsig \|_{L^2(\R^3)} =0, \end{align}
\begin{equation}\label{equ:st 2.ta}
\lim_{\eps\goto 0} 
\norm{\dfrac{\tilde{\mathcal{A}}(\eps,\cdot)\bsig - \tilde{{J}}(\cdot)\bsig}{\eps} - \mathcal{A}'(0,\cdot)\bsig}{L^2(\OS)} = 0,
\end{equation}
where we denote
$\tilde{\mathcal{A}}'(0,\cdot)$ as the G\^ateaux derivative of $\tilde{\mathcal{A}}(\eps,\cdot)$ 
%(determined by $T^\eps$) 
at $\eps =0$, namely
\[
\tilde{\mathcal{A}}'(0,\vecy)\bsig(\vecy)
=
\lim_{\eps\goto 0}
\frac{\tilde{\mathcal{A}}(\eps,\vecy)\bsig(\vecy) - \tilde{J}(\vecy) \bsig(\vecy)}{\eps},
\quad\vecy\in \OS.
\]
\end{lem}

\begin{proof}
Using \eqref{def.ta} and \eqref{def.ti} we have
\begin{align*}
&\tilde{\mathcal{A}}(\eps,\vecy) \bsig(\vecy)-\tilde{J}(\vecy)\bsig(\vecy)
\notag\\
&=(\gamma(\eps,\vecy)-1)\Big(
J_{\bT^{\eps}}^{-1} \otimes 
\cL_{J_{\bT^{\eps}}^{-1}} \bsig(\vecy)\Big)
+\big(J_{\bT^{\eps}}^{-1}-\mathbf{I}\big)
 \otimes 
\cL_{J_{\bT^{\eps}}^{-1}} \bsig(\vecy) 
+ \mathbf{I} \otimes 
\Big(\cL_{J_{\bT^{\eps}}^{-1}}\bsig(\vecy)-\cL_{\bf I} \bsig(\vecy)\Big).
\end{align*}
Now using Lemma \ref{lem.conv.sig}, we achieve
\begin{align*}
\|\tilde{\mathcal{A}}(\eps,\cdot)\bsig-\tilde{J}(\cdot)\bsig \|_{L^2(\R^3)} \leq 
O(\eps),
\end{align*}
which proves \eqref{equ:st 1.ta}. Again, using \eqref{def.ta} and Dominated Convergence Theorem we have \eqref{equ:st 2.ta}.
\end{proof}
\subsection{Material and shape derivatives} \label{sec.app}
In this section we give a general overview on first and second order shape calculus (see \cite{SokZol92}). 
These definitions and Lemmas can be introduced for the stress tensor $\bsig,$ pressure $p$ and the skew-symmetric tensor $\bs$ in the spaces $\mathbb{X}$\dela{(or $\H^{\frac{1}{2}}(\partial \OF))$} and $\bcQ$ respectively.
\par
Let $D$ be a deterministic bounded domain in $\R^3$
with boundary $\partial D$ of class $C^k,\,k \geq 2$. For $\eps>0,$ let $D^{\eps}$ be the perturbed domain with respect to $\bT^{\eps}$ (where $\bT^\eps$ is defined by~\eqref{equ:T eps define} and~\eqref{equ:T e cond}). Let us define the boundary variations $\bV_1$ and $\bV_2$ by
\begin{equation}\label{equ:V def}
\bV_i(\vecx)
:=
\kappa_i(\vecx)\vecn{}(\vecx),\,\,i=1,2
, \quad \mbox{where}\quad \dela{\kappa_i(\vecx) \in \R \quad \mbox{such that} \quad} \|\kappa_i  \|_{\mathbb{W}^{2,\infty}(\partial D)\cap C^{2,1}(\partial D)} \leq 1.
\end{equation}

We employ second order variations of the type
\begin{align*}
D_{\eps,\delta}[\kappa_1,\kappa_2]
&:=
(D_{\eps}[\kappa_1])_{\delta}[\kappa_2]:=\Big\{\vecx +\eps \kappa_1(\vecx)\bn(\vecx)+\delta \kappa_2(\vecx)\bn(\vecx): \vecx \in D \Big\},
\\
\partial D_{\eps,\delta}[\kappa_1,\kappa_2]
&:=
(\partial D_{\eps}[\kappa_1])_{\delta}[\kappa_2]:=\Big\{\vecx +\eps \kappa_1(\vecx)\bn(\vecx)+\delta \kappa_2(\vecx)\bn(\vecx): \vecx \in \partial D \Big\}.
\end{align*}
\begin{defi}\label{defi.der}
Let $\alpha > 0.$ For any sufficiently small $\eps,$ let $v^{\eps}(D^{\eps})$ be an element in $\mathrm{H}^{\alpha}(D^{\eps}).$ The {\bf material derivative} weak (strong)\dela{(or local shape derivative)} of $v^{\eps}(D^{\eps})$ in the direction of a vector field $\bV_1$ (given by \eqref{equ:V def}) denoted by $\dot{v}(D)=\dot{v}[\kappa_1,D]$ and is defined by
\[
\dot{v}(D):=\lim_{\eps \rightarrow 0} \dfrac{ v^{\eps}(D^{\eps})\circ \bT^{\eps} - v^0(D)}{\eps},
\]
provided the limit exists in weak (strong) sense in the corresponding space $\mathrm{H}^\alpha(D).$

\begin{rem} The function $\dot{v}(D)$ is the weak (strong) material derivative of
$v^{\eps}(D^{\eps})$ in $\mathrm{H}^{\alpha}(D^{\eps})$ if 
\[
\dfrac{ v^{\eps}(D^{\eps})\circ
\bT^{\eps} - v^0(D)}{\eps}
\]
is weakly (strongly) convergent to $\dot{v}(D)$ in
$\mathrm{H}^{\alpha}(D)$ as $\eps \rightarrow 0$.
\end{rem}
Proceeding in similar lines, for $\beta>0,$ one can define the material derivative weak (strong) of $v^{\eps}(\partial D^{\eps})$ in $\mathrm{H}^{\beta}(\partial D^{\eps})$) in the direction of a vector field $\bV_1$ denoted by $\dot{v}(\partial D)=\dot{v}[\kappa_1,\partial D].$
In the following proposition, let us mention the relation between weak (strong) material derivative on the domain and on the boundary. 
\begin{prop} Let $\dot{v}(D)$ be the weak (strong) material derivative of an element $v^{\eps}(D^{\eps}) \in \mathrm{H}^{\alpha}(D^{\eps})$, in the direction of a vector field $\bV_1$ (given by \eqref{equ:V def}). Then for $\alpha>1/2$, there exists the weak (strong) material derivative $\dot{v}(\partial D)$ of the element $v^{\eps}(\partial D^{\eps})=v^{\eps}(D^{\eps})|_{\partial D^{\eps}}$
 which is given by 
 $$\dot{v}(\partial D)=\dot{v}(D)|_{\partial D}\quad \mbox{in} \quad \mathrm{H}^{\alpha-1/2}(\partial D).$$.
\end{prop}
For proof, we refer to Proposition 2.75 in \cite{SokZol92}. 
\dela{The shape derivative of $v^{\eps}$ is defined by
\begin{align}\label{def.app.sd}
v^{\prime}[\kappa_1](\vecx):=\lim_{\eps \goto 0} \dfrac{v_{\eps}[\kappa_1](\vecx)-v(\vecx)}{\eps},\quad \vecx \in D \cap D_{\eps}.
\end{align}
The second order shape derivative, the `shape Hessian' is a bilinear form on pairs of boundary perturbation fields $(V,V'),$ denoted by $v''=v''[V,V']=v''[\kappa_1,\kappa_2]$ and is defined by
\begin{align}\label{def.app.sh}
v''[\kappa_1,\kappa_2](\vecx)
\dela{=\lim_{\eps \goto 0} \dfrac{(v'_{\eps}[\kappa_2])[\kappa_1](\vecx)-v'[\kappa_1](\vecx)}{\eps}}=\lim_{\eps \goto 0} \dfrac{(v'_{\eps}[\kappa_1])[\kappa_2](\vecx)-v'[\kappa_2](\vecx)}{\eps},\quad \vecx \in D \cap D_{\eps}.
\end{align}
Here $ (v'_{\eps}[\kappa_1])[\kappa_2](\vecx)= \lim_{\delta \goto 0} \dfrac{(v_{\eps}[\kappa_1])_{\delta}[\kappa_2](\vecx)-(v_{\eps}[\kappa_1])(\vecx)}{\delta},\quad \vecx \in D \cap D_{\eps}.$}
\end{defi}

\begin{defi}\label{def.sd}%\begin{itemize}
Let $\alpha>1/2.$ Let the weak material derivative $\dot{v}(D)$ exists in $\mathrm{H}^{\alpha}(D)$ (or $\dot{v}(\partial D) \in \mathrm{H}^{\alpha-1/2}(\partial D)$) and $\nabla v^0 \cdot \bV_1 \in \mathrm{H}^{\alpha}$ for  vector field $\bV_1$ (given by \eqref{equ:V def}). The {\bf shape derivative} of $v^{\eps}(D^{\eps}) \in \mathrm{H}^{\alpha}(D^{\eps})$ (or $ v^{\eps}(\partial D^{\eps}) \in \mathrm{H}^{\alpha-1/2}(\partial D^{\eps})$)
 is given by
\begin{equation}\label{def.app.sd}
 v^{\prime} =
\begin{cases} 
\dot{v}(D)-\nabla v^{0}(D)\cdot \bV_1, \qquad &\text{if}  \quad v^{\eps}(D^{\eps}) \in \mathrm{H}^\alpha(D^{\eps}),\\
\dot{v}(\partial D)- \nabla_{\partial D_0} v^{0}(\partial D)\cdot \bV_1,\quad \quad 
&\text{if} \quad v^{\eps}(\partial D^{\eps}) \in \mathrm{H}^{\alpha-1/2}(\partial D^{\eps}).
\end{cases}                   
\end{equation}\end{defi}
If $(\bV_1,\bV_2)$ are pairs of boundary perturbation fields given by \eqref{equ:V def}, let us consider $(\dot{v})^{\delta}(D^\delta) \in \mathrm{H}^{\alpha}(D^\delta)$ defined in the direction of the vector field $\bV_2$. Then the second order material derivative of $(v^{\eps})^{\delta}$ which is a bilinear form on the pair of vector fields $(\bV_1,\bV_2)$ denoted by $\ddot{v}(D)=\ddot{v}[\kappa_1,\kappa_2,D]$ and is given by
\[
\ddot{v}(D):=\lim_{\delta \rightarrow 0} \dfrac{ (\dot{v})^{\delta}(D^{\delta})\circ
\bT^{\delta} - \dot{v}(D)}{\delta}.
\]
{\bf The shape Hessian} is the second order shape derivative 
%(or shape derivative of the shape derivative), which is a bilinear form on pairs of boundary perturbation fields $(V,V'),$ 
denoted by $v''=v''[\kappa_1,\kappa_2]$ and is defined by
\begin{equation}\label{def.app.sh}
 v^{''} =
\begin{cases} 
\ddot{v}(D)- \dot{(\nabla v)}(D)\cdot \bV_1 -\nabla v(D) \cdot \dot{\bV}_1
- \nabla v'(D) \cdot \bV_2, \qquad & \text{if} \,\, \dot{v}^{\delta} \in \mathrm{H}^\alpha,
\\
\ddot{v}(\partial D)- \mathcal{M}(\nabla_{\partial D_0} v)(D)\cdot \bV_1 -\nabla_{\partial D_0}
v(D) \cdot \dot{\bV}_1 - \nabla_{\partial D_0} v'(D) \cdot \bV_2,\, 
\qquad &\mbox{if } \dot{v}^{\delta} \in \mathrm{H}^{\alpha-1/2},
\end{cases}
\end{equation}
where $\mathcal{M}(f)$ denotes the material derivative of a function $f$.

%$\ddot{y}$

\begin{lem}\label{lem.cpt.k} Let $\alpha >0$. Let $v^{\prime}$ and $v''$ be shape derivative and shape Hessian of $v(D) \in \mathrm{H}^{\alpha}(D),$ then for any compact set $K \subset \subset D$ we have
\begin{align}\label{eq.cpt.k}
v^{\prime}=\lim_{\eps \rightarrow 0} \dfrac{v^{\eps}-v^0}{\eps} \quad \mbox{and} \quad 
v''=\lim_{\delta \rightarrow 0} \dfrac{(v')^{\delta}-v'}{\delta}
\mbox{in} \quad \mathrm{H}^{\alpha}(K).
\end{align} 
\end{lem}
For proof see Lemma 3.6 of \cite{CPT}. 

With \eqref{eq.cpt.k} at hand, we obtain for all $0 \leq \eps < \eps_0,$ the `shape Taylor expansion'
\begin{align}\label{det.tay.}
v^{\eps}(\vecx)=v(\vecx)+\eps v'[\kappa_1](\vecx)+\dfrac{\eps^2}{2}v''[\kappa_1,\kappa_1](\vecx)+O(\eps^3)\quad \mbox{for} \quad \vecx \in K \subset \subset D \cap D_{\eps}.
\end{align}

\dela{
\begin{defi}
For any sufficiently small $\eps,$ let $p^{\eps}$ be an element in $\mathrm{H}^1(\OF^{\eps}))$ or $\H^{\frac{1}{2}}(\partial \OF^{\eps}).$ The material derivative of $p^{\eps},$ denoted by $\dot{p}$ is defined by
\begin{align}
\dot{p}:=\lim_{\eps \rightarrow 0} \dfrac{p^{\eps}\circ \bT^{\eps}-p^0}{\eps},
\end{align}
provided the limit exists in the corresponding space $\mathrm{H}^1(\OF)$ or $\H^{\frac{1}{2}}(\partial \OF).$
The shape derivative of $p^{\eps}$ is defined by
\begin{align}
p^{\prime}_{ij}:=\left\{\begin{aligned}
                  &\dot{p}-\nabla p^{0}\cdot V, \qquad \mbox{if} \quad p^{\eps} \in \mathrm{H}^1(\OF^{\eps}),\\
                   &\dot{p}-\nabla_{\Sigma_0} p^{0}\cdot V,\quad \mbox{if} \quad p^{\eps} \in \H^{\frac{1}{2}}(\partial \OF^{\eps})\quad  \ada{\mbox{or } \quad \H^{\frac{1}{2}}(\Sigma^{\eps})}.
                   \end{aligned}
   \right.
\end{align}
\end{defi}

\ada{
\begin{lem}
If $p^{\prime}$ is a shape derivative of $p^{\eps} \in \mathrm{H}^1(\OF^{\eps}))$, then for any compact set $K_2 \subset \subset \OF$ we have
\begin{align}
p^{\prime}=\lim_{\eps \rightarrow 0} \dfrac{p^{\eps}-p^0}{\eps} \quad \mbox{in} \quad \ada{\H^1(K_2)}.
\end{align} 

\end{lem} 
 }}
 
\begin{lem}\label{pro:mat sha pro}
Let $\alpha>0.$ Let $ \dot v$, $\dot w$ be material derivatives, and
$v'$, $w'$ be shape derivatives of
 $v^\eps$, $w^\eps$ in $\mathrm{H}^{\alpha}(D^\eps)$, $\eps\ge 0$, respectively. 
Then the following statements are true.
\begin{itemize}%[(i)]
\item[(i)] \label{ite:t1}  
The material and shape derivatives of the product $v^\eps w^\eps$ 
are $
\dot v w^0
+
v^0\dot w$ and 
$
v' w^0
+
v^0 w'$, respectively.
\item[(ii)] \label{ite:t2}
The material and shape derivatives of the quotient $\dfrac{v^\eps}{w^\eps}$ 
are 
\[
\dfrac{(\dot v w^0
-
v^0\dot w)}{(w^0)^2}
\quad\text{and}\quad
\dfrac{(v' w^0
-
v^0 w')}{(w^0)^2},
\]
respectively,
provided that all the fractions are well-defined.
\item[(iii)]\label{ite:t3}
If $v^\eps = v$ for all $\eps\ge 0$, then $ \dot v = \nabla v^0\cdot \bV_1
= \nabla v\cdot \bV_1$
and $v' = 0$.
\item[(iv)] \label{ite:t4}
If 
\[
\mathcal{J}_1(D^\eps) := \displaystyle\int_{D^\eps} v^\eps\,d\vecx,
\quad
\mathcal{J}_2(D^\eps) := \displaystyle\int_{\partial D^\eps} v^\eps\,d\sigma,
\]
and
\[
d\mathcal{J}_i(D^\eps)|_{\eps=0}
:=
\lim_{\eps\goto 0}
\frac{\mathcal{J}_i(D^\eps)-\mathcal{J}_i(D^0)}{\eps},
\
i=1,2,
\]
then 
\[
d\mathcal{J}_1(D^\eps)|_{\eps=0}
=
\int_{D^0} v'\,d\vecx
+
\int_{\partial D^0} v^0
\inpro{\bV_1}{{\vecn}}
\,d\sigma
\]
and 
\[
d\mathcal{J}_2(D^\eps)|_{\eps=0}
=
\int_{\partial D^0} v'\,d\sigma
+
\int_{\partial D^0} 
\left(
\frac{\partial v^0}{\partial \bn}
+
\divv_{\partial D^0}(\vecn)\,
v^0
\right)
\inpro{\bV_1}{{\vecn}}
\,d\sigma.
\]
\dela{\item[(v)]\label{ite:t5} If the shape Hessian of the functionals $\mathcal{J}_i(D^{\eps})$ are denoted by $d^2 \mathcal{J}_i(D^\eps)$ and are given as
\[
d^2 \mathcal{J}_i(D^\eps)|_{\eps=0}
:=
\lim_{\eps\goto 0}
\frac{d\mathcal{J}_i(D^\eps)-d\mathcal{J}_i(D^0)}{\eps},
\
i=1,2,
\] then,
\[
d^2\mathcal{J}_1(D^\eps)|_{\eps=0}
=
\int_{D^0} v''\,d\vecx
+\int_{\partial D^0} v' \tilde{\kappa_2}d\sigma+\int_{\partial D^0}(v' \tilde{\kappa_1}+v \tilde{\kappa_1}')\,d\sigma
+
\int_{\partial D^0} \Big(\dfrac{\partial}{\partial \bn}(v \tilde{\kappa_1})+\bdiv_{\partial D^0}(\bn)v \tilde{\kappa_1} \Big) \tilde{\kappa_2} \, d\sigma
\]
and 
\[
d^2\mathcal{J}_2(D^\eps)|_{\eps=0}=
\]}
\item[(v)]\label{ite:t5} The shape derivatives of 
$\left.\dfrac{\partial v}{\partial \bn^{\eps}}\right|_{\partial D^{\eps}}$ and 
$\left. w^{\eps}\dfrac{\partial v}{\partial \bn^{\eps}}\right|_{\partial D^{\eps}}$ are, respectively,
\[
\nabla_{\partial D^{\eps}}v \cdot \nabla_{\partial D^0} \langle \bV_1,\bn \rangle 
\quad\text{and}\quad
\left. w' \dfrac{\partial v}{\partial \bn}\right|_{\partial D^{0}}
-
w^0 \Big(\nabla_{\partial D^0} v \cdot \nabla_{\partial D^0}\langle \bV_1,\bn \rangle   \Big).
\]
\end{itemize}
\end{lem} 
 
\begin{proof}
Statements~(i)--(iii) and (v) can be obtained by using 
elementary calculations. Statement~(iv) are proved in~\cite[pages~113--116]{SokZol92}.
\end{proof}

\begin{lem}\label{lem:n shape}
The material and shape derivatives of the normal field
$\bn^\eps$ are given by 
\[
\dot{\bn} 
=
\bn' = -\nabla_{\partial \OS}\kappa.
\]
\end{lem} 
\begin{proof}
See \cite[Lemma 3.9]{CPT}. 
\end{proof}

%\subsection{Stochastic domain}\label{sec.app.stoc}
%\subsubsection{Statistical Moments}

\end{appendix}

%\bibliographystyle{crelle}
%\bibliography{deb} 

\end{document}